\documentclass[10pt,A4paper]{article}
\usepackage{amssymb}
\usepackage{amsfonts}
\usepackage{amsmath}
\usepackage{amsthm}
\usepackage{latexsym}
\usepackage[dvips]{epsfig}
\usepackage{enumerate}
\usepackage{mathrsfs}
\usepackage{eufrak}
\usepackage{bm}
\usepackage{tikz}
\usepackage{authblk}
\usepackage{cancel}

\usepackage{color}

\usepackage{stmaryrd}

\usepackage[T1]{fontenc}
\usepackage[utf8]{inputenc}

\theoremstyle{plain}
\newtheorem{proposition}{Proposition}
\newtheorem{lemma}{Lemma}
\newtheorem{theorem}{Theorem}

\newtheorem{corollary}{Corollary}

\newtheorem{remark}{Remark}

\def\bma{{\bm a}}
\def\bmb{{\bm b}}
\def\bmc{{\bm c}}

\def\bmg{{\bm g}}

\def\bml{{\bm l}}
\def\bmn{{\bm n}}
\def\bmm{{\bm m}}

\def\bmA{{\bm A}}
\def\bmB{{\bm B}}

\def\bmsigma{{\bm \sigma}}

\def\bmDelta{{\bm \Delta}}

\def\bmpartial{{\bm \partial}}

\def\nablasl{/\kern-0.58em\nabla}
\def\Deltasl{/\kern-0.7em\Delta}
\def\Dsl{/\kern-0.7em D}

\def\sigmasl{/\kern-0.58em\sigma}
\allowdisplaybreaks

\def\TiPsi{{\tilde \Psi}}
\def\Titau{{\tilde \tau}}
\def\Tipi{{\tilde \pi}}

\def\Tiomega{{\tilde \omega}}

\def\omegadg{{\omega^{\dagger}}}

\newcommand{\mthorn}{\textit{þ}} % italic lower thorn
\newcommand{\meth}{\textit{ð}} % italic lower eth
\newcounter{mnotecount}%[section]

\newcommand{\mnotex}[1]%{}
{\protect{\stepcounter{mnotecount}}$^{\mbox{\footnotesize $\bullet$\themnotecount}}$
\marginpar{%\color{red}%
\raggedright\tiny\em
$\!\!\!\!\!\!\,\bullet$\themnotecount: #1} }

\newcommand*{\rom}[1]{\expandafter\@slowromancap\romannumeral #1@}

\usepackage{hyperref}

\begin{document}

\title{\textbf{On the local existence for the characteristic initial value problem for the Einstein-Dirac system}}

\author[,1]{Peng
  Zhao \footnote{E-mail address:{\tt p.zhao@bnu.edu.cn}}}
\author[,2,3]{Xiaoning Wu \footnote{E-mail address:{\tt
      wuxn@amss.ac.cn}}}
\affil[1]{Faculty of Arts and Sciences, Beijing Normal University, Zhuhai, 519087, China.}        
\affil[2]{Institute of Mathematics, Academy of Mathematics and Systems Science 
and State Key Laboratory of Mathematical Sciences, Chinese Academy of Sciences, Beijing 100190, China.}
\affil[3]{School of Mathematical Sciences, University of Chinese Academy of Sciences, 
Beijing 100049, China.}

\maketitle
\begin{abstract}
In this paper, we investigate the characteristic initial value problem for the Einstein–Dirac system, 
a model governing the interaction between gravity and spin-$1/2$ fields. 
We apply Luk's strategy \cite{Luk12} 
and prove a semi-global existence result for this coupled Einstein–Dirac system 
without imposing symmetry conditions. 
More precisely, we construct smooth solutions in a rectangular region 
to the future of two intersecting null hypersurfaces, 
on which characteristic initial data are specified. 
The key novelty is to promote the symmetric spinorial derivatives of the Dirac field to independent variables 
and to derive a commuted “Weyl-curvature-free” evolution system for them. 
This eliminates the coupling to the curvature in the energy estimates 
and closes the bootstrap at the optimal derivative levels. 
The analysis relies on a double null foliation 
and incorporates spinor-specific techniques essential to handling the structure of the Dirac field. 

\end{abstract}

\section{Introduction}
The characteristic initial value problem (CIVP) in general relativity plays a fundamental role 
in understanding spacetime dynamics, particularly in scenarios involving gravitational radiation, 
black hole formation, and stability analyses. 
Rendall \cite{Ren90} first established local existence results near the intersection of two null hypersurfaces in vacuum, 
followed by Luk’s significant contributions \cite{Luk12}, 
which systematically developed robust analytical techniques within a double-null foliation framework. 
Given the physical significance of matter fields in realistic astrophysical and cosmological contexts, 
recent research has extended these methodologies to coupled Einstein–matter systems. 
Notably, this includes the characteristic initial value problems for Yang–Mills fields \cite{PuskarYau22} 
as well as our previous comprehensive study of the Einstein–Maxwell–Complex Scalar (EMS) system \cite{PengXiaoning25}. 
These advancements have laid essential mathematical groundwork for further exploration of gravitational interactions with various matter fields.

In this paper, we focus on the Einstein–Dirac system, 
describing the gravitational interaction with spin-$1/2$
fields governed by the Dirac equation. 
Originally formulated by Dirac in the context of relativistic quantum mechanics, 
the Dirac equation fundamentally characterizes fermionic particles such as electrons, neutrinos, and other 
half-spin particles. 
Its significance spans numerous areas in physics, 
from elementary particle physics and quantum field theory to astrophysical scenarios 
including neutron star models and gravitational collapse involving neutrino emissions. 
In mathematical general relativity, 
the study of Dirac fields on fixed spacetimes is closely connected with fundamental questions of spacetime stability and wave propagation properties. 
Rigorous analyses of these linear problems have provided valuable insights into the stability of important solutions, such as black-hole spacetimes, see \cite{YaohuaXiao18,GeJiangZhang18,MaZhang20,LiZang21,Xuantao23,JiaLi23,HequnXiao24}. 

In our paper, we focus on the fully nonlinear Einstein–Dirac system, 
where the Dirac spinor fields dynamically couple with spacetime geometry. 
The rigorous mathematical results for its characteristic initial value problem remain limited.
The intrinsic nonlinearity and spinorial complexity in this coupled system 
introduce substantial new mathematical challenges. 
Crucially, unlike scalar or electromagnetic fields, 
although the Dirac equation is a first-order PDE system, 
the energy-momentum tensor consists of the product of the Dirac field and its derivative. 
Then one needs control of the Dirac field two order higher than curvature, 
which prevents closing the bootstrap.

To address these fundamental challenges, 
we identify a suitable decomposition of spinor derivatives, 
separating the symmetric and antisymmetric parts. 
Remarkably, the symmetric portion emerges as an independent dynamical variable we denote by $\Upsilon$,
whose equations exhibit a favourable structure enabling us to establish robust energy estimates. 
The key point is part of those equations do not contain curvature, hence we can do the $L^2$ estimate
\begin{align*}
\int_{\mathcal{N}_u}|\Upsilon_L|^2+\int_{\mathcal{N}'_v}|\Upsilon_R|^2\leq
Ini+\int_{\mathcal{D}_{u,v}}(No\ Curvature)
\end{align*}
with lower order requirement of curvature when one estimates higher derivative of $\Upsilon$.
This ensure the necessary closure conditions for our bootstrap argument. 
This technical innovation enables us to rigorously construct semi-global solutions to the Einstein–Dirac characteristic initial value problem. 
Specifically, we prescribe characteristic initial data on two intersecting null hypersurfaces 
and prove the existence of smooth solutions of the Einstein–Dirac system 
in a rectangular neighborhood to the future of their intersection 
without imposing any symmetry assumptions.

The results established in this paper provide a rigorous mathematical foundation 
for studying gravitational interactions involving spinor fields, 
filling a critical gap in the mathematical analysis of coupled Einstein–matter systems. 
Based on this work, in our subsequent study 
we will provide a rigorous proof and characterization of trapped surface formation 
within the Einstein–Weyl system, 
aiming to mathematically understand the physics of black hole formation via spinor field collapse.
Thus, our analysis not only advances the rigorous treatment of fundamental gravitational-spinorial interactions but also paves the way toward exploring new and physically meaningful scenarios in mathematical general relativity.

\paragraph{Conventions.}
In this article, Latin letters $a$, $b$, $c$, ... denote the abstract tensorial indices 
and $\bma$, $\bmb$, $\bmc$, ... denote the tensorial frame indices taking the values 0, ..., 3.
Capital Latin letters $A$, $B$, $C$, ... denote the abstract spinorial indices 
and $\bm{A}$, $\bm{B}$, $\bm{C}$, ... denote the spinorial frame indices taking the values 0,1. 
Let $\epsilon_{AB}$ denote the antisymmetric product of two spinors $\xi$ and $\eta$ as 
$\left\llbracket \xi,\eta \right\rrbracket=\epsilon_{AB}\xi^A\eta^B$. Indices are raised and lowered with $\epsilon^{AB}$ and $\epsilon_{AB}$, e.g. $\xi_B=\xi^A\epsilon_{AB}$. 
Given a spin basis $\{o,\iota\}$, $\epsilon_{AB}$ can be expressed by 
$\epsilon_{AB}=o_A\iota_B-\iota_Bo_A$. 
Denote $\epsilon_{\bm{0}}^{\phantom{\bm{0}}A}=o^A$ and 
$\epsilon_{\bm{1}}^{\phantom{\bm{1}}A}=\iota^A$, we also choose a $\bm{g}$-orthogonal basis 
$e_{\bma}$ and the dual basis $\omega^{\bma}$; that is $g_{\bma\bmb}=\eta_{\bma\bmb}$. 
We make use of the Infeld-van der Waerden symbols $\sigma^{\bma}_{\phantom{\bma}\bm{A}\bm{A}'}$ to connect the $g_{\bma\bmb}$ and $\epsilon_{\bmA\bmB}$ via 
$\epsilon_{\bmA\bmB}\epsilon_{\bmA'\bmB'}=\eta_{\bma\bmb}\sigma^{\bma}_{\phantom{\bma}\bm{A}\bm{A}'}\sigma^{\bmb}_{\phantom{\bmb}\bm{B}\bm{B}'}$ where 
$\sqrt{2}\sigma^{\bma}_{\phantom{\bma}\bm{A}\bm{A}'}$ is the standard Pauli matrices, 
$\sigma_{\bma}^{\phantom{\bma}\bm{A}\bm{A}'}$ is the inverse.
Then we define the spinorial counterpart of a tensor $T_{a}^{\phantom{a}b}$ via 
$T_{\bmA\bmA'}^{\phantom{\bmA\bmA'}\bmB\bmB'}\equiv
 T_{\bma}^{\phantom{\bma}\bmb}\sigma^{\bma}_{\phantom{\bma}\bm{A}\bm{A}'}
 \sigma_{\bmb}^{\phantom{\bmb}\bmB\bmB'}$. 
Hence we can connect between $T_a^{\phantom{a}b}$ and $T_A^{\phantom{A}B}$.
In order to keep consistency with the antisymmetric product $g_{AA'BB'}=\epsilon_{AB}\epsilon_{A'B'}$, 
the signature of metric is $(+,-,-,-)$, the convention of curvature is 
$\nabla_a\nabla_b\omega_c-\nabla_b\nabla_a\omega_c=-R_{abc}^{\phantom{abc}d}\omega_d$.
Throughout, the spinor calculation follow the conventions of \cite{PenRin86,Ste91,CFEBook}.

\subsection{Outline of the article}
In section \ref{EinsteinDiracSys}, we introduce the Einstein-Dirac system in the spinorial form. 
Making irreducible decomposition to the derivative of the Dirac spinor, 
we choose the symmetric part as new variable and derive its equations. 
In section \ref{BasicSetting} we introduce the geometric setting, coordinate choice and 
equations in T-weight formalism. We also formulate a CIVP for Einstein-Dirac system. 
In section \ref{Maintheorem} we present the main theorem of this paper and the skeleton of the proof. 
In section \ref{Mainanalysis} we show the details of the proof.

\subsection{Acknowledgements}
The calculations described in this article have been carried out in the suite xAct for Mathematica, see \cite{Jose14}. 
Peng Zhao was supported by the start-up fund of Beijing Normal University at Zhuhai.
Xiaoning Wu was supported by the National Natural Science Foundation of China (Grant No. 12275350).

\section{Einstein-Dirac system}
\label{EinsteinDiracSys}
In what follows, let $(\mathcal{M},\bm{g})$ denote a 4-dimensional manifold 
which is orientable and time-orientable with vanishing second Stiefel–Whitney class. 
Then there exists a spinor structure globally. 
The Dirac field $\bm\psi$ consists of two two-component spinor fields $(\phi^A,\bar\chi_{A'})$, 
and the equations of motion are
\begin{align}
\nabla_{AA'}\phi^A=-m\bar\chi_{A'}, \quad 
\nabla_{AA'}\bar\chi^{A'}=-m\phi_{A} \label{EOMDirac}
\end{align}
where $m$ is the fixed coupling constant representing the mass of the Dirac field, 
and $\nabla_{AA'}$ is the spinorial counterpart of covariant derivative $\nabla_a$.
Here $\phi^A$ is the left Weyl spinor and $\bar\chi_{A'}$ is the right Weyl spinor. 
In the remainder of this paper, $(\phi_A,\chi_{A})$ are the spinor fields we mainly focus on. 
The energy-momentum tensor is 
\begin{align}
T_{ABA'B'}=-2\mathrm{i}\big(&-\bar\phi_{B'}\nabla_{AA'}\phi_B+\phi_B\nabla_{AA'}\bar\phi_{B'}
-\bar\phi_{A'}\nabla_{BB'}\phi_A+\phi_A\nabla_{BB'}\bar\phi_{A'} \nonumber \\
&+\bar\chi_{B'}\nabla_{AA'}\chi_B-\chi_B\nabla_{AA'}\bar\chi_{B'}
+\bar\chi_{A'}\nabla_{BB'}\chi_A-\chi_A\nabla_{BB'}\bar\chi_{A'})
\label{EinsteinDiracEM}
\end{align}
Then the Einstein field equations $R_{ab}-\frac{1}{2}Rg_{ab}=T_{ab}$ 
can be expressed in the spinorial way
\begin{align}
-2\Phi_{ABA'B'}+6\Lambda\epsilon_{AB}\epsilon_{A'B'}
=T_{ABA'B'} \label{EinsteinDiracEQ}
\end{align} 
where $\Phi_{ABA'B'}$ is the spinorial counterpart of the trace free Ricci tensor, $\Lambda=-R/24$, 
see \cite{PenRin86,Ste91,CFEBook}.

The analysis on the back reaction of Dirac field to the spacetime relies heavily on its derivative . 
In the analysis of Einstein-Scalar system, we focus on the gradient of the scalar field. 
It is therefore natural to introduce the derivative of the Dirac field as an independent variable.
Considering that the equations of motion of Dirac field have laid down the constraints to the 
antisymmetric part of the derivative, 
consequently, the symmetric part emerges as the essential new dynamical quantity requiring independent analysis.

The irreducible decomposition for $\nabla_{AA'}\phi_B$ is 
\begin{align*}
\nabla_{AA'}\phi_B=\nabla_{(A|A'|}\phi_{B)}+\frac{1}{2}\epsilon_{AB}\nabla_{CA'}\phi^C.
\end{align*}
Define its symmetric part by 
\begin{align*}
\zeta_{ABA'}\equiv\nabla_{(A|A'|}\phi_{B)}.
\end{align*}
Then from the Dirac equation \eqref{EOMDirac} one has 
\begin{align}
\zeta_{ABA'}=\nabla_{AA'}\phi_B+\frac{m}{2}\epsilon_{AB}\bar\chi_{A'}.
\label{Definitionzeta}
\end{align}
Similarly we define 
\begin{align*}
\eta_{ABA'}\equiv\nabla_{(A|A'|}\chi_{B)}
\end{align*}
and obtain
\begin{align}
\eta_{ABA'}=\nabla_{AA'}\chi_B+\frac{m}{2}\epsilon_{AB}\bar\phi_{A'}.
\label{Definitioneta}
\end{align}
Then the energy momentum tensor has the following form
\begin{align*}
T_{ABA'B'}=-2\mathrm{i}&\big(\phi_A\bar\zeta_{B'A'B}-\bar\phi_{A'}\zeta_{BAB'}
+\phi_B\bar\zeta_{A'B'A}-\bar\phi_{B'}\zeta_{ABA'} \\
&-\chi_A\bar\eta_{B'A'B}+\bar\chi_{A'}\eta_{BAB'}
-\chi_B\bar\eta_{A'B'A}+\bar\chi_{B'}\eta_{ABA'}\\
&-m\bar\epsilon_{A'B'}\phi_B\chi_A+m\bar\epsilon_{A'B'}\phi_A\chi_B
+m\epsilon_{AB}\bar\phi_{B'}\bar\chi_{A'}-m\epsilon_{AB}\bar\phi_{A'}\bar\chi_{B'} \big)
\end{align*}
where $\bar\zeta_{A'B'A}$ is the conjugate of $\zeta_{ABA'}$.
And the Einstein field equations are
\begin{align}
-2\Phi_{ABB'A'}+6\Lambda\epsilon_{AB}\epsilon_{A'B'}
=-2\mathrm{i}&\big(\phi_A\bar\zeta_{B'A'B}-\bar\phi_{A'}\zeta_{BAB'}
+\phi_B\bar\zeta_{A'B'A}-\bar\phi_{B'}\zeta_{ABA'}
-\chi_A\bar\eta_{B'A'B} \nonumber\\
&+\bar\chi_{A'}\eta_{BAB'}-\chi_B\bar\eta_{A'B'A}+\bar\chi_{B'}\eta_{ABA'}
-m\bar\epsilon_{A'B'}\phi_B\chi_A \nonumber\\
&+m\bar\epsilon_{A'B'}\phi_A\chi_B
+m\epsilon_{AB}\bar\phi_{B'}\bar\chi_{A'}-m\epsilon_{AB}\bar\phi_{A'}\bar\chi_{B'} \big).
\label{EinsteinDiracEQalt}
\end{align} 

We make use of the commutator of derivative of $\phi_A$ to obtain the equations 
which $\zeta_{ABA'}$ satisfies.
From 
\begin{align*}
\nabla_{AA'}\nabla_{BB'}\phi_C-\nabla_{BB'}\nabla_{AA'}\phi_C=
\epsilon_{A'B'}\Box_{AB}\phi_C+\epsilon_{AB}\Box_{A'B'}\phi_C
\end{align*}
where 
\begin{align*}
\Box_{AB}\equiv\nabla_{Q'(A}\nabla_{B)}^{\phantom{B)}Q'}, \quad
\Box_{A'B'}\equiv\nabla_{Q(A'}\nabla_{B')}^{\phantom{B')}Q}
\end{align*}
and 
\begin{align*}
\Box_{AB}\phi_C=\Psi_{ABCD}\phi^D-2\Lambda\phi_{(A}\epsilon_{B)C}, \quad
\Box_{A'B'}\phi_C=\Phi_{CDA'B'}\phi^D,
\end{align*}
and the EOM for $\phi_A$, 
one concludes that $\zeta_{ABA'}$ satisfies the following
\begin{align}
\nabla_{AA'}\zeta_{BCB'}-\nabla_{BB'}\zeta_{ACA'}=&
\Psi_{DCAB}\epsilon_{A'B'}\phi^D+\Phi_{DCA'B'}\epsilon_{AB}\phi^D
-\Lambda\epsilon_{CB}\epsilon_{A'B'}\phi_A
-\Lambda\epsilon_{CA}\epsilon_{A'B'}\phi_B \nonumber\\
&-\frac{m}{2}\bar\eta_{A'B'B}\epsilon_{AC}+\frac{m}{2}\bar\eta_{A'B'A}\epsilon_{BC}
-\frac{m^2}{4}\epsilon_{BC}\epsilon_{A'B'}\phi_A
-\frac{m^2}{4}\epsilon_{AC}\epsilon_{A'B'}\phi_B. \label{commutatorzeta}
\end{align}
Here $\Psi_{ABCD}$ is the spinorial counterpart of the Weyl tensor.
Similarly one has
\begin{align}
\nabla_{AA'}\eta_{BCB'}-\nabla_{BB'}\eta_{ACA'}=&
\Psi_{DCAB}\epsilon_{A'B'}\chi^D+\Phi_{CDA'B'}\epsilon_{AB}\chi^D
-\Lambda\epsilon_{CB}\epsilon_{A'B'}\chi_A
-\Lambda\epsilon_{CA}\epsilon_{A'B'}\chi_B \nonumber\\
&-\frac{m}{2}\bar\zeta_{A'B'B}\epsilon_{AC}+\frac{m}{2}\bar\zeta_{A'B'A}\epsilon_{BC}
-\frac{m^2}{4}\epsilon_{BC}\epsilon_{A'B'}\chi_A
-\frac{m^2}{4}\epsilon_{AC}\epsilon_{A'B'}\chi_B. \label{commutatoreta}
\end{align}
The above two equations are the main equations for analysing $\zeta_{ABA'}$ and $\eta_{ABA'}$.

\section{Basic geometric setting, T-weight formalism and the formulation of CIVP}
\label{BasicSetting}
\subsection{Basic geometric setting}
\label{CoordinateChoice}
We adopt the same geometric setup as in our earlier paper \cite{HilValZha19,PengXiaoning25}, 
i.e. assume that $(\mathcal{M},\bm{g})$ possesses boundary: 
outgoing null edge $\mathcal{N}_{\star}$ and ingoing null edge $\mathcal{N}'_{\star}$ 
and their intersection $\mathcal{S}_{\star}=\mathcal{N}_{\star}\cap\mathcal{N}'_{\star}$. 
We also assume the existence of the double null foliation in the future 
of $\mathcal{N}_{\star}\cup\mathcal{N}'_{\star}$. 
The level sets $u$-surfaces $\mathcal{N}_u$ are outgoing null hypersurfaces and 
$\mathcal{N}'_v$ represent the ingoing null hypersurfaces where 
$\mathcal{N}_0=\mathcal{N}_{\star}$ and 
$\mathcal{N}'_0=\mathcal{N}'_{\star}$. 
Denote $\mathcal{S}_{u,v}=\mathcal{N}_{u}\cap\mathcal{N}'_{v}$ be the 
spacelike topological 2-sphere. 
We also denote $\mathcal{N}_u(v_1,v_2)$ be the part of the hypersurface~$\mathcal{N}_u$ 
with $v_1\leq v\leq v_2$. Likewise~$\mathcal{N}'_v(u_1,u_2)$ has a similar
definition. 
Define the region $\mathcal{D}_{u,v}$ via
\begin{align}
\mathcal{D}_{u,v}\equiv\bigcup_{0\leq v'\leq v, 0\leq u'\leq u}\mathcal{S}_{u',v'}.
\end{align}
Follow the coordinate choice in \cite{HilValZha19,PengXiaoning25} 
we can construct a Newman-Penrose frame~$\{\bml,\bmn,\bm{m},\bar{\bm{m}}\}$ of the form
\begin{align}
  \bml=\bmpartial_v+C^{\mathcal{A}}\bmpartial_{\mathcal{A}},\qquad
  \bmn=Q\bmpartial_u, \qquad
  \bmm=P^{\mathcal{A}} \bmpartial_{\mathcal{A}}, \label{framem}
\end{align}
where~$C^{\mathcal{A}}=0$ on~$\mathcal{N}_{\star}$, and~$Q=1$ on~$\mathcal{N}'_{\star}$.
More discussion can be found in~\cite{HilValZha19}. 
The coordinate choice leads to the following properties of the connection coefficients
\begin{subequations}
\begin{align}
& \kappa=\nu=\gamma=0, \label{spinconnection1}\\
& \rho=\bar{\rho},\ \ \mu=\bar{\mu}, \label{spinconnection2}\\
& \pi=\alpha+\bar{\beta} \label{spinconnection3}
\end{align}
\end{subequations}
in the neighbourhood of~$\mathcal{D}_{u,v}$ and, furthermore, with
\begin{align*}
\epsilon-\bar{\epsilon}=0\ \ \ \textrm{on}\ \ \
\mathcal{D}_{u,v}\cap\mathcal{N}_{\star}. 
\end{align*}
Also one can obtain the equations for the frame coefficient~$Q$, $P^{\mathcal{A}}$
    and~$C^{\mathcal{A}}$:
\begin{subequations}
\begin{align}
  &\Delta C^{\mathcal{A}}=-(\bar{\tau}+\pi)P^{\mathcal{A}}
  -(\tau+\bar{\pi})\bar{P}^{\mathcal{A}}, \label{framecoefficient1} \\
  &\Delta P^{\mathcal{A}}=-\mu P^{\mathcal{A}}
  -\bar{\lambda}\bar{P}^{\mathcal{A}}, \label{framecoefficient2} \\
  &DP^{\mathcal{A}}-\delta C^{\mathcal{A}}=
  (\rho+\epsilon-\bar{\epsilon})P^{\mathcal{A}}
  +\sigma\bar{P}^{\mathcal{A}}, \label{framecoefficient3} \\
&DQ=-(\epsilon+\bar{\epsilon})Q,  \label{framecoefficient4}\\
  &\bar{\delta}P^{\mathcal{A}}-\delta\bar{P}^{\mathcal{A}}=
  (\alpha-\bar{\beta})P^{\mathcal{A}}
  -(\bar{\alpha}-\beta)\bar{P}^{\mathcal{A}}, \label{framecoefficient5} \\
&\delta Q=(\tau-\bar{\pi})Q. \label{framecoefficient6}
\end{align}
\end{subequations}
Details can be found in \cite{HilValZha19}. 

%%%%%%%%%%%%%%%%%%%%%%%%%%%%%%%%%%%%%%%%%%%%%%%%%%%%%%%%%%%%%
\begin{figure}[t]
\centering
\includegraphics[width=0.8\textwidth]{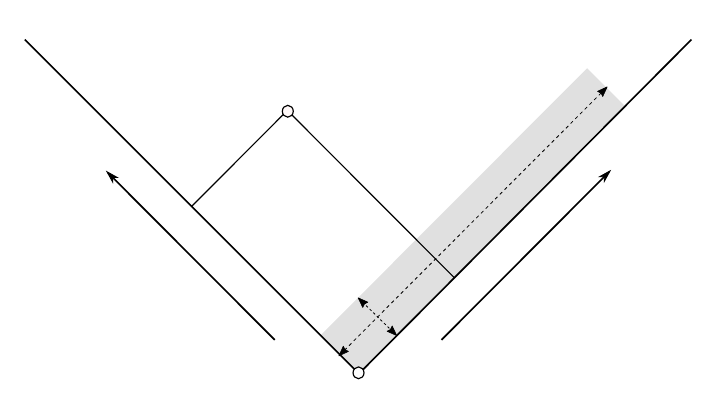}
\put(-60,80){$l^a$, $v$}
\put(-300,80){$n^a$, $u$}
\put(-40,140){$\mathcal{N}_\star$}
\put(-320,140){$\mathcal{N}'_\star$}
\put(-180,5){$\mathcal{S}_{u_\star,v_\star}$}
\put(-190,140){$\mathcal{S}_{u,v}$}
\put(-190,90){$\mathcal{D}_{u,v}$}
\put(-160,50){$\varepsilon$}
\put(-80,120){$v_\bullet$}
\put(-238,120){$\mathcal{N}_u$}
\put(-173,120){$\mathcal{N}'_v$}
\caption{Setup for coordinate gauge choice with a
  double null foliation.}
\label{Fig:CharacteristicSetup}
\end{figure}
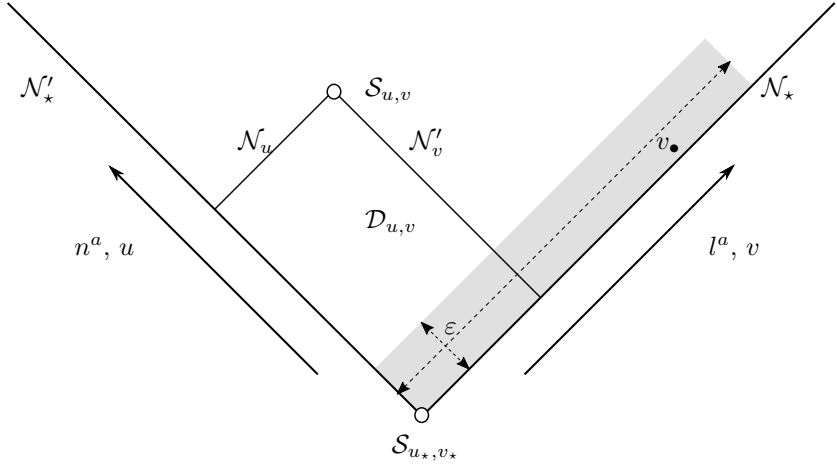
%%%%%%%%%%%%%%%%%%%%%%%%%%%%%%%%%%%%%%%%%%%%%%%%%%%%%%%%%%%%%

\subsection{T-weight formalism and equations}
To fit the PDE analysis, based on the GHP formalism, 
we introduce the T-weight formalism by assigning quantity a 
so-called T-weight $s$
and introducing four new differential operators~$\meth$,
$\meth'$, $\mthorn$ and~$\mthorn'$
\begin{align*}
\meth f&\equiv\delta f+s(\beta-\bar\alpha)f, \quad
\meth' f \equiv\bar\delta f-s(\bar\beta-\alpha)f, \quad
\mthorn f\equiv Df+s(\epsilon-\bar\epsilon)f, \quad
\mthorn' f\equiv\Delta f+s(\gamma-\bar\gamma)f,
\end{align*}
acting on any quantity~$f$ with defined T-weight~$s$. 
The properties of T-weight formalism ensure that 
the norm of such derivative of T-weight quantities is independent of the spherical coordinates choice. 
Then one obtains the following:
\begin{remark}[Covariant derivative $\nablasl$ and norm on $\mathcal{S}$]
\label{RemarkNorm}
  {\em Let $f$ be a T-weight quantity and $T(f)$ be its associated tensor on $\mathcal{S}$, 
  then the norm of~$\nablasl^kT(f)$ can be computed in terms of the
norm of all its components~$...\meth...\meth'...f$, i.e. we have
\begin{align*}
 |\mathcal{D}^kf|^2\equiv\sum_{\alpha}|\mathcal{D}^{k_i}f|^2= |\nablasl^kT(f)|^2,
\end{align*}
where~$\mathcal{D}^{k_i}f$ is a string of order~$k$ of the
operators~$\meth$ and~$\meth'$, and the sum over~$\alpha$ denotes all
such strings. This leads to the definition of norm on $\mathcal{S}$
\begin{align}
||\mathcal{D}^kf||^p_{L^p(\mathcal{S})}\equiv\int_{\mathcal{S}}|\mathcal{D}^kf|^p,\quad 
||\mathcal{D}^kf||_{L^{\infty}(\mathcal{S})}\equiv\sup_{\mathcal{S}}|\mathcal{D}^kf|.
\end{align}
  }
  
\end{remark} 
More discussions of the properties of T-weight formalism can be found in~\cite{HilValZha23}.

\smallskip
\noindent
\textbf{(1) Dirac equations in T-weight formalism}

To expand the spinor equations, one needs introduce a spin basis $\{o,\iota\}$ and 
has the standard convention $\epsilon_{AB}o^A\iota^B=1$. 
In what follows, we follow the conventions of \cite{PenRin86,Ste91,CFEBook}.
The relation with the NP frame is 
\begin{align*}
l^{AA'}\equiv o^A\bar{o}^{A'}, \quad n^{AA'}\equiv\iota^A\bar{\iota}^{A'}, \quad
m^{AA'}\equiv o^A\bar{\iota}^{A'}, \quad \bar{m}^{AA'}\equiv \iota^A\bar{o}^{A'}
\end{align*}
and the NP derivatives are defined by 
\begin{align*}
D=l^{AA'}\nabla_{AA'}, \quad \Delta=n^{AA'}\nabla_{AA'}, \quad 
\delta=m^{AA'}\nabla_{AA'}, \quad \bar\delta=\bar{m}^{AA'}\nabla_{AA'}.
\end{align*}
Then one can define the connection coefficients as
\begin{align*}
\kappa=&o^ADo_A, \quad \epsilon=\iota^ADo_A, \quad
\pi=\iota^AD\iota_A, \quad \tau=o^A\Delta o_A, \quad
\gamma=\iota^A\Delta o_A, \quad \nu=\iota^A\Delta\iota_A, \\
\beta=&\iota^A\delta o_A, \quad \sigma=o^A\delta o_A, \quad
\mu=\iota^A\delta\iota_A, \quad \alpha=\iota^A\bar\delta o_A, \quad 
\rho=o^A\bar\delta o_A, \quad \lambda=\iota^A\bar\delta\iota_A
\end{align*}
The components of the Weyl spinor $\Psi_{ABCD}$ and 
the trace-free Ricci spinor $\Phi_{ABA'B'}$ can be found in \cite{PenRin86,Ste91,CFEBook}.

Define the components of $\phi_A$ and $\chi_A$ with respect to the spin basis $\{o,\iota\}$ by 
\begin{align*}
\phi_0\equiv\phi_Ao^A, \quad \phi_1\equiv\phi_A\iota^A, \quad 
\chi_0\equiv\chi_Ao^A, \quad \chi_1\equiv\chi_A\iota^A.
\end{align*}
Define the components of $\zeta_{ABA'}$ and $\eta_{ABA'}$ 
with respect to the spin basis $\{o,\iota\}$ by 
\begin{align*}
\zeta_0&\equiv\zeta_{ABA'}o^Ao^B\bar{o}^{A'}, \quad
\zeta_1\equiv\zeta_{ABA'}o^A\iota^B\bar{o}^{A'}, \quad
\zeta_2\equiv\zeta_{ABA'}\iota^A\iota^B\bar{o}^{A'}, \\
\zeta_3&\equiv\zeta_{ABA'}o^Ao^B\bar{\iota}^{A'}, \quad
\zeta_4\equiv\zeta_{ABA'}o^A\iota^B\bar{\iota}^{A'}, \quad
\zeta_5\equiv\zeta_{ABA'}\iota^A\iota^B\bar{\iota}^{A'}, \\
\eta_0&\equiv\eta_{ABA'}o^Ao^B\bar{o}^{A'}, \quad
\eta_1\equiv\eta_{ABA'}o^A\iota^B\bar{o}^{A'}, \quad
\eta_2\equiv\eta_{ABA'}\iota^A\iota^B\bar{o}^{A'}, \\
\eta_3&\equiv\eta_{ABA'}o^Ao^B\bar{\iota}^{A'}, \quad
\eta_4\equiv\eta_{ABA'}o^A\iota^B\bar{\iota}^{A'}, \quad
\eta_5\equiv\eta_{ABA'}\iota^A\iota^B\bar{\iota}^{A'}. 
\end{align*}

The T-weight of such quantities are list in the following:
\begin{align*}
&s=-\frac{3}{2}:\quad \zeta_3, \quad \eta_3, \\
&s=-\frac{1}{2}:\quad \phi_0,\quad \zeta_0,\quad \zeta_4,
\quad \chi_0,\quad \eta_0,\quad \eta_4, \\
&s=\frac{1}{2}:\quad \phi_1, \quad \zeta_1, \quad \zeta_5,
\quad \chi_1, \quad \eta_1, \quad \eta_5, \\
&s=\frac{3}{2}:\quad \zeta_2, \quad \eta_2.
\end{align*}

With the definitions of components one can then obtain their equations. 
The Dirac equation \eqref{EOMDirac} has the following form:
\begin{subequations}
\begin{align}
\mthorn\,\phi_{1} &= -m\,\overline{\chi}_{0}
   + \frac{\phi_{0}\,\pi}{2}
   + \phi_{1}\,\rho
   - \frac{\phi_{1}\,\omega}{2}
   + \meth'\phi_{0}\,,\label{EOMDiracL1}\\
\mthorn'\,\phi_{0} &= m\,\overline{\chi}_{1}
   - \phi_{0}\,\mu
   + \frac{\phi_{1}\,\overline{\pi}}{2}
   - \phi_{1}\,\tau
   + \meth\,\phi_{1}\,,\label{EOMDiracL2}\\
\mthorn\,\chi_{1} &= -m\,\overline{\phi}_{0}
   + \frac{\chi_{0}\,\pi}{2}
   + \chi_{1}\,\rho
   - \frac{\chi_{1}\,\omega}{2}
   + \meth'\chi_{0}\,,\label{EOMDiracR1}\\
\mthorn'\,\chi_{0} &= m\,\overline{\phi}_{1}
   - \chi_{0}\,\mu
   + \frac{\chi_{1}\,\overline{\pi}}{2}
   - \chi_{1}\,\tau
   + \meth\,\chi_{1}\,.\label{EOMDiracR2}
\end{align}
\end{subequations}

One can also expand equations which reflect the definitions of $\zeta_{ABA'}$ \eqref{Definitionzeta} 
and $\eta_{ABA'}$ \eqref{Definitioneta}. Such equations can be found in \ref{DefinitionDerDirac}. 

As the equations for $\zeta_{ABA'}$ and $\eta_{ABA'}$, i.e. \eqref{commutatorzeta} and 
\eqref{commutatoreta} are rather lengthy, 
we give the equations only in schematic form here 
and refer the reader to Appendix \ref{Tweightzeta} and \ref{Tweighteta} for the fully explicit expressions.
We denote $\phi_i$ and $\chi_j$ by $\psi$, denote $\zeta_i$ and $\eta_j$ by $\Upsilon$, 
denote connection coefficients by $\Gamma$, denote the Weyl curvatures by $\Psi$, 
then the schematic structure of equations are listed as follows:
\begin{align*}
&\{\mthorn,\mthorn'\}\Upsilon-\meth\Upsilon=\Psi\psi+m\Upsilon+m^2\psi+m\psi^2+\Gamma\Upsilon+\Upsilon\psi^2, \\
&\meth'\Upsilon-\meth\Upsilon=m\Upsilon+m^2\psi+m\psi^2+\Gamma\Upsilon+\Upsilon\psi^2+\Psi\psi.
\end{align*}

\begin{remark}[Weyl-curvature-free pairs]
Among the commuted equations for $\Upsilon=(\zeta_i,\eta_j)$, the pairs 
$(\zeta_0,\zeta_1)$, $(\zeta_1,\zeta_2)$, $(\zeta_3,\zeta_4)$, $(\zeta_4,\zeta_5)$, 
$(\eta_0,\eta_1)$, $(\eta_1,\eta_2)$, $(\eta_3,\eta_4)$ and $(\eta_4,\eta_5)$ 
are free of the Weyl curvature; see App.~\ref{TweightzetaNoCurv} and \ref{TweightetaNoCurv}.
This feature is pivotal for the top-order energy closure.
\end{remark}

\smallskip
\noindent
\textbf{(2) The Einstein field equation} 

Expand the Einstein field equation \eqref{EinsteinDiracEQalt} with the fields $\zeta_{ABA'}$ and $\eta_{ABA'}$, 
one obtains the following

\begin{subequations}
\begin{align}
\Phi_{00} &= 2\mathrm{i}\bigl(\bar\zeta_0\phi_0-\zeta_0\bar\phi_0-\bar\eta_0\chi_0+\eta_0\bar\chi_0\bigr),\label{EDeq1}\\
\Phi_{01} &= \mathrm{i}\bigl(2\bar\zeta_1\phi_0-\zeta_3\bar\phi_0-\zeta_0\bar\phi_1-2\bar\eta_1\chi_0+\eta_3\bar\chi_0+\eta_0\bar\chi_1\bigr),\label{EDeq2}\\
\Phi_{02} &= 2\mathrm{i}\bigl(\bar\zeta_2\phi_0-\zeta_3\bar\phi_1-\bar\eta_2\chi_0+\eta_3\bar\chi_1\bigr),\label{EDeq3}\\
\Phi_{11} &= \mathrm{i}\bigl(\bar\zeta_4\phi_0-\zeta_4\bar\phi_0+\bar\zeta_1\phi_1-\zeta_1\bar\phi_1-\bar\eta_4\chi_0+\eta_4\bar\chi_0-\bar\eta_1\chi_1+\eta_1\bar\chi_1\bigr),\label{EDeq4}\\
\Phi_{12} &= \mathrm{i}\bigl(\bar\zeta_5\phi_0+\bar\zeta_2\phi_1-2\zeta_4\bar\phi_1-\bar\eta_5\chi_0-\bar\eta_2\chi_1+2\eta_4\bar\chi_1\bigr),\label{EDeq5}\\
\Phi_{22} &= 2\mathrm{i}\bigl(\bar\zeta_5\phi_1-\zeta_5\bar\phi_1-\bar\eta_5\chi_1+\eta_5\bar\chi_1\bigr),\label{EDeq6}\\
\Lambda &= \frac{\mathrm{i}m}{3}\bigl(\phi_1\chi_0-\bar\phi_1\bar\chi_0-\phi_0\chi_1+\bar\phi_0\bar\chi_1\bigr). \label{EDeq7}
\end{align}
\end{subequations}

\smallskip
\noindent
\textbf{(3) The structure equations, Bianchi identities and the renormalised Weyl curvature} 

Once we have the expression of Ricci tensor shown in above, we can obtain the structure equations 
whose schematic are
\begin{align*}
\{\mthorn,\mthorn'\}\Gamma-\meth\Gamma=m\psi^2+\Upsilon\psi+\Gamma\Gamma+\Psi.
\end{align*}
The fully explicit expressions can be found in the appendix \ref{StructureEQ}.

In order to formulate a Hodge system (as defined for instance
in~\cite{ChrKla93}) :
\begin{align*}
\mthorn'\Psi_{j}-\meth \Psi_{j+1}=&P_0; \\
\mthorn \Psi_{j+1}-\meth' \Psi_j=&Q_0,
\end{align*}
for the Bianchi identity and apply the energy estimate \eqref{EnergyIdentity}, 
besides the equations of motion, 
one also needs to introduce the renormalised Weyl curvature which are defined by
\begin{subequations}
\begin{align}
\tilde\Psi_1\equiv\Psi_1-\Phi_{01}, \quad
\tilde\Psi_2\equiv\Psi_2+2\Lambda,  \quad
\tilde\Psi_3\equiv\Psi_3-\Phi_{21}.
\end{align}
\end{subequations}
With those quantities, 
one can absorb the trouble terms $\mthorn\{\zeta_0,\eta_0\}$ and $\mthorn'\{\zeta_5,\eta_5\}$ 
in the equations of $\{\mthorn,\mthorn'\}\Psi_{1,2,3}$. 
For the trouble terms $\mthorn\{\phi_0,\chi_0\}$ and $\mthorn'\{\phi_1,\chi_1\}$, 
one can make use of the definition equation shown in \ref{DefinitionDerDirac}. 
Here trouble terms means we do not have their equations.
Then one has the following schematic expression for Bianchi Identity:
\begin{align*}
\{\mthorn,\mthorn'\}\Psi_i-\{\meth,\meth'\}\Psi_j=&m\Upsilon\psi+m\psi^2\Gamma
+\Upsilon\psi\Gamma+\psi\meth\Upsilon+\Psi\psi^2+\Upsilon\psi^3+\Upsilon^2+\Gamma\Psi_k. 
\end{align*}
The fully explicit equations are shown in \ref{BianchiIdentity}. 
Since the right-hand side of the equation involves first-order spherical derivatives of $\Upsilon$, 
the curvature can be controlled only at one order less than $\Upsilon$. 

\subsection{The formulation of the characteristic initial value problem}
In this section we follow the standard procedure to construct the initial data for Einstein-Dirac system on 
$\mathcal{N}_{\star}\cup\mathcal{N}'_{\star}$ from freely specifiable data.   

\begin{lemma}[\textbf{\em freely specifiable data for the CIVP}]
  \label{Lemma:FreeDataCIVP} Working under the coordinate choice~\ref{CoordinateChoice},
  initial data for the Einstein-Dirac system
  on~$\mathcal{N}_{\star}\cup\mathcal{N}_{\star}^{\prime}$ can be
  computed (near~$\mathcal{S}_{\star}$) from a reduced data
  set~$\mathbf{r}_\star$ consisting of:
\begin{align*}
  &\Psi_0, \quad \phi_0,\quad \chi_0, \quad \epsilon+\bar\epsilon \quad \mbox{\textrm{on}}
  \quad \mathcal{N}_{\star},\nonumber\\
  &\Psi_4, \quad \phi_1,\quad \chi_1, \ \mbox{\textrm{on}}\ \ \mathcal{N}_{\star}^{\prime}, \nonumber\\
  &\lambda,\ \ \sigma,\ \ \mu,\ \ \rho, \ \ \pi,  \ \
  P^{\mathcal{A}}\ \ \mbox{\textrm{on}}\ \ \mathcal{S}_{\star}. \nonumber
\end{align*}
\end{lemma}
\begin{proof}
We follow the standard strategy by solving the ODE on the lightcone.

\smallskip
\noindent
\textbf{Data on $\mathcal{S}_{\star}$. } 
From $P^{\mathcal{A}}$ 
one can define the 2-metric and the connections $\alpha-\bar\beta$. 
This leads to the definition of operators $\delta$, $\bar\delta$ as well as $\meth$ and $\meth'$. 
Then \eqref{framecoefficient6} and $Q=1$ lead to $\tau=\bar\pi=\bar\alpha+\beta$ 
and hence we obtain $\alpha$ and $\beta$. 
With the standard NP operators and all connection coefficients, 
one can make use of the equations for the definition of $\zeta_{ABA'}$ and $\eta_{ABA'}$ 
shown in \ref{DefinitionDerDirac} to obtain all the value of $\zeta_{ABA'}$ and $\eta_{ABA'}$.
The value of $\TiPsi_1$ and $\TiPsi_3$ can be computed by \eqref{rhosigma} and \eqref{mulambda}.
$\TiPsi_2$ can be computed from \eqref{alphabeta}.

\smallskip
\noindent
\textbf{Data on $\mathcal{N}'_{\star}$. } 
$Q=1$ leads to $\Delta=\partial_u$ and $\tau=\bar\pi$. 
$\gamma=0$, \eqref{Defzeta5} and \eqref{Defeta5} let one compute $\zeta_5$ and $\eta_5$.
With the results above and solve \eqref{thornprimemu} and \eqref{thornprimelambda} together, 
one can obtain $\mu$ and $\lambda$. 
With the value of $\mu$ and $\lambda$ one can compute $P^{\mathcal{A}}$ from \eqref{framecoefficient2}. 
Hence one can define the 2-metric, the connections $\alpha-\bar\beta$. 
and operators $\delta$, $\bar\delta$ as well as $\meth$ and $\meth'$ on $\mathcal{N}'_{\star}$. 
Solve the $\bmn$-direction equations \eqref{thornprimepi}, \eqref{Deltaalpha}, \eqref{Deltabeta}, \eqref{thornprimePsi3}, \eqref{EOMDiracL2}, 
\eqref{EOMDiracR2}, \eqref{thornprimezeta2}, \eqref{thornprimezeta4}, 
\eqref{thornprimeeta2} and \eqref{thornprimeeta4} along $\mathcal{N}'_{\star}$ together, 
one can obtain the value of $\pi$, $\alpha$, $\beta$, $\TiPsi_3$, $\phi_0$, $\chi_0$, $\zeta_2$, $\zeta_4$, 
$\eta_2$ and $\eta_4$. 
Then from $\tau=\bar\pi$ one obtains $\tau$ and hence 
the equation \eqref{framecoefficient1} leads the value of $C^{\mathcal{A}}$.
Again solve \eqref{Deltaepsilon}, \eqref{thornprimerho}, \eqref{thornprimesigma}, \eqref{thornprimePsi2}, 
\eqref{thornprimezeta1}, \eqref{thornprimezeta3}, 
\eqref{thornprimeeta1} and \eqref{thornprimeeta3} together, 
one can obtain the value of $\epsilon$, $\rho$, $\sigma$, $\TiPsi_2$, $\zeta_1$, $\zeta_3$, 
$\eta_1$ and $\eta_3$. The value of $\omega$ can be obtained by its definition $\omega=\epsilon+\bar\epsilon$. 
Then one can obtain $\TiPsi_1$ from \eqref{thornprimePsi1}. 
The value of $\Psi_0$, $\eta_0$ and $\zeta_0$ can be obtained by 
\eqref{thornprimePsi0}, \eqref{thornprimeeta0} and \eqref{thornprimezeta0}.

\smallskip
\noindent
\textbf{Data on $\mathcal{N}_{\star}$. } $C^{\mathcal{A}}=0$ means $D=\partial_v$. 
The value of $\epsilon+\bar\epsilon$, i.e. $\omega$ and $\epsilon=\bar\epsilon$ leads to $\epsilon$. 
Then the value of $\zeta_0$ and $\eta_0$ can be calculated by \eqref{Defzeta0} and \eqref{Defeta0}.
The value of $Q$ can be computed by \eqref{framecoefficient4} with $\epsilon+\bar\epsilon$.
One can obtain $\rho$ and $\sigma$ by solving \eqref{thornrho} and \eqref{thornsigma} together. 
The value of $P^{\mathcal{A}}$ can be computed by \eqref{framecoefficient3} and hence 
one obtains $\delta$, $\bar\delta$, $\meth$ and $\meth'$.
Then solve \eqref{Dbeta}, \eqref{Dalpha}, \eqref{thornpi}, \eqref{thornPsi1}, \eqref{EOMDiracL1}, 
\eqref{EOMDiracR1}, \eqref{thornzeta1}, \eqref{thornzeta3}, \eqref{thorneta1} and \eqref{thorneta3} together 
one can obtain $\beta$, $\alpha$, $\pi$, $\TiPsi_1$, $\phi_1$, $\chi_1$, $\zeta_1$, 
$\zeta_3$, $\eta_1$ and $\eta_3$.
With these one can obtain $\tau$ by solving \eqref{thorntau}. 
Combine \eqref{thornmu}, \eqref{thornlambda}, \eqref{thornPsi2}, 
\eqref{thornzeta2}, \eqref{thornzeta4}, \eqref{thorneta2} and \eqref{thorneta4}, 
one can obtain $\mu$, $\lambda$, $\TiPsi_2$, $\zeta_2$, $\zeta_4$, $\eta_2$ and $\eta_4$. 
With these results, the value of $\TiPsi_3$ can be calculated by solving \eqref{thornPsi3}.
Finally, the value of $\Psi_4$, $\eta_5$ and $\zeta_5$ can be obtained by 
\eqref{thornPsi4}, \eqref{thorneta5} and \eqref{thornzeta5}.

\end{proof}

Next one can extract a symmetric hyperbolic system (SHS) from the Einstein-Dirac system and then obtain 
the local existence results:

\begin{theorem}
($\bm{Local\ existence\ and\ uniqueness\ to\ the\ standard\ characteristic\ initial}$\\ 
$\bm{\ value\ problem\ of \ Einstein-Dirac \ system}$) Given a smooth reduced initial data set $\bm{r}_\star$ for the Einstein-Dirac system on $\mathcal{N}_{\star}\cup\mathcal{N}_{\star}^{\prime}$, there exists a unique smooth solution of the Einstein-Dirac system in a neighbourhood of $\mathcal{D}_{u,v}$ on $J^+(\mathcal{S}_{\star})$ which induces the prescribed initial data on $\mathcal{N}_{\star}\cup\mathcal{N}_{\star}^{\prime}$.
\end{theorem}

The proof makes use of Rendall’s method \cite{Ren90} and Whitney’s theorem, similar discussion can be found in 
\cite{HilValZha20,PengXiaoning25}

\section{Main theorem and the strategy of proof}
\label{Maintheorem}

For convenience, we define a new quantity $\varrho$ by $\varrho\equiv\Delta\log Q$ to obtain a better estimate the 
frame coefficient $Q$. Quantity $\varrho$ is at the same level of connection coefficients. 
We can calculate its outgoing direction equation by the commutator relation and $\mthorn\omega$:
\begin{align}
\mthorn \varrho &= \TiPsi_2 + \bar\TiPsi_2 
  + 2\mathrm{i}\,\bar{\zeta}_4 \phi_0 
  - 2\mathrm{i}\,\zeta_4 \bar{\phi}_0 
  + 2\mathrm{i}\,\bar{\zeta}_1 \phi_1 
  - 2\mathrm{i}\,\zeta_1 \bar{\phi}_1 
  - 2\mathrm{i}\,\bar{\eta}_4 \chi_0 \nonumber\\
  &\quad 
  - \frac{2}{3}\mathrm{i}\,m \phi_1 \chi_0 
  + 2\mathrm{i}\,\eta_4 \bar{\chi}_0 
  + \frac{2}{3}\mathrm{i}\,m\overline{ \phi}_1\,\bar{\chi}_0 
  - 2\mathrm{i}\,\bar{\eta}_1 \chi_1 
  + \frac{2}{3}\mathrm{i}\,m \phi_0 \chi_1 
  + 2\mathrm{i}\,\eta_1 \bar{\chi}_1  \nonumber\\
  &\quad
   - \frac{2}{3}\mathrm{i}\,m\overline{ \phi}_0\,\bar{\chi}_1
  + 2\pi \tau 
  + 2\bar{\pi}\,\bar{\tau} 
  + 2\tau \bar{\tau} 
  - \varrho \omega. \label{thornvarrho}
\end{align}
The initial data of $\varrho$ is 0 on $\mathcal{N}'_{\star}$. 
Once we have controlled $\varrho$, one can then control the frame coefficient $Q$.
Because we do not need the estimate of top derivative, 
hence the curvature terms do not cause troubles, more details and discussions can be 
found in \cite{PengXiaoning25}.

\subsection{Integration and Norms }
Define the norm on $\mathcal{S}_{u,v}$:
\begin{align}
||f||_{L^2(\mathcal{S}_{u,v})}\equiv\left(\int_{\mathcal{S}_{u,v}}|f|^2\right)^{1/2},\quad
||f||_{L^p(\mathcal{S}_{u,v})}\equiv\left(\int_{\mathcal{S}_{u,v}}|f|^p\right)^{1/p},\quad
||f||_{L^{\infty}(\mathcal{S}_{u,v})}\equiv\sup_{\mathcal{S}_{u,v}}|f|,
\end{align}
where $1\leq p<\infty$.
Assume the T-weight of $f$ is 0, define integration over $\mathcal{D}_{u,v}$:
\begin{align}
\int_{\mathcal{D}_{u,v}}f&\equiv
\int_0^u\int_0^{v}\int_{\mathcal{S}_{u',v'}}f \bm{\varepsilon}_{\bm{g}}
=\int_0^u\int_0^{v}\int_{\mathcal{S}_{u',v'}}
Q^{-1}f\bm{\varepsilon}_{\bm{\sigma}}\mathrm{d}v'\mathrm{d}u'.
\end{align}
Here the bold letter $\bm{\varepsilon}_{\bm{g}}$ is the volume element with
spacetime metric $\bmg$,
bold letter $\bm{\varepsilon}_{\bm{\sigma}}$ is the volume element with the
induced metric $\bm{\sigma}$ on $\mathcal{S}_{u,v}$.
Define norms on the null hypersurfaces~$\mathcal{N}_u$
and~$\mathcal{N}'_v$:
\begin{align}
\int_{\mathcal{N}_u(0,v)}f \equiv \int_0^v\int_{S_{u,v'}}
f \bm{\varepsilon}_{\bm{\sigma}}\mathrm{d}v',\quad
\int_{\mathcal{N}'_v(0,u)}f \equiv
\int_0^u\int_{S_{u',v}}f \bm{\varepsilon}_{\bm{\sigma}}\mathrm{d}u'.
\end{align}
We will often use the notation 
\begin{align}
\int_{\mathcal{N}_u}f \equiv \int_{\mathcal{N}_u(0,I)}f , \qquad
\int_{\mathcal{N}'_v}f \equiv \int_{\mathcal{N}'_v(0,\epsilon)}f
\end{align}
to denote the norms on the full outgoing and incoming slices.

Then we introduce norms that will be used in the main bootstrap argument.

\smallskip
\noindent
\textbf{Norms in the spacetime.} 

(i) Supremum-type norm over the~$L^2$-norm of the connection coefficients at
 spheres of constant~$u,v$, given by,
\begin{align*}
\Delta_{\Gamma}(\mathcal{S})\equiv\sup_{u,v}
\sup_{\Gamma\in\{\rho,\mu,\sigma,\lambda,\tau,\pi,\varrho,\omega\}}\max
\{\sum_{i=0}^1||\mathcal{D}^i\Gamma||_{L^{\infty}(\mathcal{S}_{u,v})},
\sum_{i=0}^2||\mathcal{D}^i\Gamma||_{L^{4}(\mathcal{S}_{u,v})},
\sum_{i=0}^3||\mathcal{D}^i\Gamma||_{L^{2}(\mathcal{S}_{u,v})} \}.
\end{align*}

(ii) Norm for the components of the Weyl tensor at null
 hypersurfaces, given by,
\begin{align*}
  \Delta_{\Psi}\equiv  \sum_{i=0}^3  \left(\sup_{\Psi_L\in\{\Psi_0,\Psi_1,\Psi_2,\Psi_3\}}\sup_{u}
||\mathcal{D}^i\Psi_L||_{L^2(\mathcal{N}_{u})}
  +\sup_{\Psi_R\in\{\Psi_1,\Psi_2,\Psi_3,\Psi_4\}}\sup_{v}
  ||\mathcal{D}^i\Psi_R||_{L^2(\mathcal{N}_{v})}\right)
\end{align*}
where the supreme in~$u$ and~$v$ are taken
over~$\mathcal{D}_{u,v}$.

(iii) Supremum-type norm over the~$L^2$-norm of the components of the Weyl tensor at
  spheres of constant~$u,v$, given by,
\begin{align*}
  \Delta_{\Psi}(\mathcal{S})=\sum_{i=0}^2\sup_{u,v}
    ||\mathcal{D}^i\{\Psi_0,\Psi_1,\Psi_2,\Psi_3\}||_{L^2(\mathcal{S}_{u,v})},
\end{align*}
with the supremum taken over~$\mathcal{D}_{u,v}$, and in
which~$u$ will be taken sufficiently small to apply our estimates.

(iv) Norm for the components of the $\phi_A$ and $\chi_A$ at null
  hypersurfaces, given by,
\begin{align*}
  \Delta_{\psi}\equiv&\sum_{i=0}^4\left(\sup_{u}
  ||\mathcal{D}^i\{\phi_0,\chi_0\}||_{L^2(\mathcal{N}_{u})}
  +\sup_{v}
  ||\mathcal{D}^i\{\phi_1,\chi_1\}||_{L^2(\mathcal{N}_{v})}\right)
\end{align*}
where the suprema in~$u$ and~$v$ are taken
over~$\mathcal{D}_{u,v}$.

(v) Supremum-type norm over the~$L^2$-norm of the components of $\phi_A$ and $\chi_A$ at
  spheres of constant~$u,v$, given by,
\begin{align*}
  \Delta_{\psi}(\mathcal{S})=\sum_{i=0}^3\sup_{u,v}
    ||\mathcal{D}^i\{\phi_0,\phi_1,\chi_0,\chi_1\}||_{L^2(\mathcal{S}_{u,v})},
\end{align*}
with the supremum taken over~$\mathcal{D}_{u,v}$, and in
which~$u$ will be taken sufficiently small to apply our estimates.

(vi) Norm for the components of the $\zeta_{ABA'}$ and $\eta_{ABA'}$ at null
  hypersurfaces, given by,
\begin{align*}
  \Delta_{\Upsilon}\equiv&  \sum_{i=0}^4\left(\sup_{\Upsilon_L\in\{\zeta_0,\zeta_1,\zeta_3,\zeta_4,\eta_0,\eta_1,\eta_3,\eta_4\}}\sup_{u}
  ||\mathcal{D}^i\Upsilon_L||_{L^2(\mathcal{N}_{u})}
  +\sup_{\Upsilon_R\in\{\zeta_1,\zeta_2,\zeta_4,\zeta_5,\eta_1,\eta_2,\eta_4,\eta_5\}}\sup_{v}
  ||\mathcal{D}^i\Upsilon_R||_{L^2(\mathcal{N}_{v})}\right) 
\end{align*}
where the suprema in~$u$ and~$v$ are taken
over~$\mathcal{D}_{u,v}$.

(vii) Supremum-type norm over the~$L^2$-norm of the components of $\zeta_{ABA'}$ and $\eta_{ABA'}$ at
  spheres of constant~$u,v$, given by,
\begin{align*}
  \Delta_{\Upsilon}(\mathcal{S})=\sum_{i=0}^3\sup_{u,v}
    ||\mathcal{D}^i\{\zeta_i,\eta_j\}||_{L^2(\mathcal{S}_{u,v})},
\end{align*}
with the supremum taken over~$\mathcal{D}_{u,v}$ and $i$, $j$ from 0 to 5, and in
which~$u$ will be taken sufficiently small to apply our estimates.

\smallskip
\noindent
\textbf{Norms for the initial data.} 

(i) Norm for the initial data of frame is defined by:
\begin{align*}
\Delta_{e_{\star}}\equiv\sup_{\mathcal{N}_{\star},\mathcal{N}'_{\star}}\sup_{D_U}
\{|Q|, |Q^{-1}|, |C^{\mathcal{A}}|, |P^{\mathcal{A}}|, |\varphi| \}+I,
\end{align*}
where $D_U\equiv\cup_{0\leq u\leq\varepsilon,0\leq v\leq I}U_{u,v}$ and 
$U_{u,v}$ is the coordinate patch generated along $\bml$ and $\bmn$ from 
the coordinate patch $U$ on $\mathcal{S}_{\star}$.
We make use of $C(\Delta_{e_{\star}})$ to denote a constant which is only depend on $\Delta_{e_{\star}}$. 

(ii) Norm for the initial data of connection coefficients is defined by 
\begin{align*}
\Delta_{\Gamma_{\star}}\equiv\sup_{\mathcal{S}\in\mathcal{N}_{\star}\cup\mathcal{N}'_{\star}}
\sup_{\Gamma\in\{\rho,\mu,\sigma,\lambda,\tau,\pi,\varrho,\omega\}}\max
\{1,\sum_{i=0}^1||\mathcal{D}^i\Gamma||_{L^{\infty}(\mathcal{S})},
\sum_{i=0}^2||\mathcal{D}^i\Gamma||_{L^{4}(\mathcal{S})},
\sum_{i=0}^3||\mathcal{D}^i\Gamma||_{L^{2}(\mathcal{S})} \}.
\end{align*}

(iii) The norm for the initial data of curvature is defined by
\begin{align*}
\Delta_{\Psi_{\star}}\equiv&\sup_{\mathcal{S}\subset\mathcal{N}_{\star}\cup\mathcal{N}'_{\star}}
\sup_{\Psi\in\{\Psi_0,...\Psi_4\}}\max\{1,\sum_{i=0}^1||\mathcal{D}^i\Psi||_{L^4(\mathcal{S})},
\sum_{i=0}^2||\mathcal{D}^i\Psi||_{L^2(\mathcal{S})} \} \\
&+\sum_{i=0}^3\left(\sup_{\Psi_L\in\{\Psi_0,...,\Psi_3\}}||\mathcal{D}^i\Psi_L||_{L^2(\mathcal{N}_{\star})}
+\sup_{\Psi_R\in\{\Psi_1,...,\Psi_4\}}||\mathcal{D}^i\Psi_R||_{L^2(\mathcal{N}'_{\star})} \right).
\end{align*}

(iv) The norm for the initial data of $\phi_A$ and $\chi_A$ is defined by 
\begin{align*}
\Delta_{\psi_{\star}}\equiv&\sup_{\mathcal{S}\subset\mathcal{N}_{\star}\cup\mathcal{N}'_{\star}}
\sup_{\psi_j\in\{\phi_0,\phi_1,\chi_0,\chi_1\}}\max\{1,
\sum_{i=0}^1||\mathcal{D}^i\psi_j||_{L^{\infty}(\mathcal{S})},
\sum_{i=0}^2||\mathcal{D}^i\psi_j||_{L^4(\mathcal{S})},
\sum_{i=0}^3||\mathcal{D}^i\psi_j||_{L^2(\mathcal{S})} \} \\
&+\sum_{i=0}^4\left(||\mathcal{D}^i\{\phi_0,\chi_0\}||_{L^2(\mathcal{N}_{\star})}
+||\mathcal{D}^i\{\phi_1,\chi_1\}||_{L^2(\mathcal{N}'_{\star})} \right)
\end{align*}

(v) The norm for the initial data of $\zeta_{ABA'}$ and $\eta_{ABA'}$ is defined by 
\begin{align*}
\Delta_{\Upsilon_{\star}}\equiv&\sup_{\mathcal{S}\subset\mathcal{N}_{\star}\cup\mathcal{N}'_{\star}}
\sup_{\Upsilon_j\in\{\zeta_0,\zeta_1,...,\zeta_5,\eta_0,\eta_1,...,\eta_5\}}\max\{1,
\sum_{i=0}^1||\mathcal{D}^i\Upsilon_j||_{L^{\infty}(\mathcal{S})},
\sum_{i=0}^2||\mathcal{D}^i\Upsilon_j||_{L^4(\mathcal{S})},
\sum_{i=0}^3||\mathcal{D}^i\Upsilon_j||_{L^2(\mathcal{S})} \} \\
&+\sum_{i=0}^4\left(\sup_{\Upsilon_L\in\{\zeta_0,\zeta_1,\zeta_3,\zeta_4,\eta_0,\eta_1,\eta_3,\eta_4\}}||\mathcal{D}^i\Upsilon_L||_{L^2(\mathcal{N}_{\star})}
+\sup_{\Upsilon_R\in\{\zeta_1,\zeta_2,\zeta_4,\zeta_5,\eta_1,\eta_2,\eta_4,\eta_5\}}||\mathcal{D}^i\Upsilon_R||_{L^2(\mathcal{N}'_{\star})} \right).
\end{align*}

\subsection{Main theorem and strategy of proof}
In this section we present the main results and the strategy of proof.
\begin{theorem}[\textbf{\em Improved local existence for the CIVP for
    the Einstein-Dirac system}]
\label{MainTheoremED}
Given regular initial data for the Einstein-Dirac system as
constructed in Lemma~\ref{Lemma:FreeDataCIVP} on the null
hypersurfaces~$\mathcal{N}_\star\cup\mathcal{N}'_\star$
for~$\{0\leq v\leq I\}$, there exists~$\varepsilon>0$
such that a unique smooth solution to the Einstein-Dirac system 
exists in the region where~$\{0\leq v\leq I\}$
and~$0\leq u\leq \varepsilon$ defined by the null
coordinates~$(u,v)$. The number~$\varepsilon$ can be chosen to depend
only on the initial data 
\begin{align*}
\Delta_{e_{\star}},\quad \Delta_{\Gamma_{\star}},\quad \Delta_{\psi_{\star}},
\quad \Delta_{\Upsilon_{\star}},\quad 
\Delta_{\Psi_{\star}}.
\end{align*}
Moreover, in this spacetime, the following holds
\begin{align*}
\Delta_{\Gamma(\mathcal{S})}+\Delta_{\psi}+
\Delta_{\Upsilon}+\Delta_{\Psi}\leq
C(\Delta_{e_{\star}}, \Delta_{\Gamma_{\star}},\Delta_{\psi_{\star}},
 \Delta_{\Upsilon_{\star}}, \Delta_{\Psi_{\star}}).
\end{align*}

\end{theorem}

\smallskip 
\noindent
\textbf{Strategy of proof}: 
The energy-momentum tensor of Einstein-Dirac system depends on the Dirac spinor $\psi$ 
and its derivative $\Upsilon$. 
Consequently, when estimating to the Weyl curvature via the Bianchi identities, 
one must control higher-order derivatives of $\Upsilon$. 
To close the bootstrap argument, we require that 
the evolution equations for $\Upsilon$ do not involve the Weyl curvature. 
The equations for $\Upsilon$ are derived by commutating the covariant derivative to the Dirac spinor $\psi$, 
so the Weyl curvature appears a priori. 
However, by invoking the Dirac equation and reorganizing the resulting identities, 
one obtains a system for $\Upsilon$ in which the Weyl curvature disappears, 
see \ref{TweightzetaNoCurv} and \ref{TweightetaNoCurv}.
These Weyl–free equations are central to the estimates for $\Upsilon$. 

With this in mind, our proof strategy follows \cite{PengXiaoning25}. 
We begin by imposing bootstrap assumptions for connection coefficients $\Gamma$, 
curvature $\Psi$ and matter fields $\psi$, $\Upsilon$. 
Under this assumptions we derive the next-to-leading order estimates for 
$\Gamma$, $\psi$, $\Upsilon$ and $\Psi$ via Gr\"onwall type inequalities. 
These estimates are established in Section \ref{next-to-leading}. 
Building on them, we then obtain the elliptic estimates for $\Gamma$ 
required in the energy argument, see Section \ref{EllipticEst}. 

To close the bootstrap, 
we require highest–order energy estimates for both the matter fields and the curvature. 
The Dirac equation and the evolution equations for $\Upsilon$  
exhibit a favourable null structure, analogous to the Bianchi identities.
This enables us to cast the systems into Hodge form and to perform pairwise energy estimates.
We first treat the pairs $(\phi_0,\phi_1)$ and $(\chi_0,\chi_1)$. 
We then exploit the Weyl–free evolution systems to estimate  
$(\zeta_0,\zeta_1)$, $(\zeta_1,\zeta_2)$, $(\zeta_3,\zeta_4)$, $(\zeta_4,\zeta_5)$, 
$(\eta_0,\eta_1)$, $(\eta_1,\eta_2)$, $(\eta_3,\eta_4)$, $(\eta_4,\eta_5)$.
These bounds yield the requisite control of the Weyl curvature at top order and thus close the bootstrap argument, see Section \ref{EnergyEst}.

Having closed the bootstrap scheme, 
we establish existence via a standard last–slice argument \cite{Luk12,PengXiaoning25}. 
Assume, for contradiction, that there is a last spacelike slice of existence 
in the rectangular domain $\bm{\mathcal{D}}$. 
The bootstrap estimates furnish uniform control of $\Delta_{\Gamma(\mathcal{S})}$, $\Delta_{\psi}$, 
$\Delta_{\Upsilon}$ and $\Delta_{\Psi}$ up to this slice and, 
in particular, ensure solvability of the evolution / constraint system slightly to its future. 
Hence one can produce a future development from the purported last slice, 
contradicting its definition. It follows that the solution persists throughout $\bm{\mathcal{D}}$.

\section{Main analysis}
\label{Mainanalysis}
In this section we carry out the core analysis. 
The overall strategy closely follows that of Paper \cite{PengXiaoning25}, 
that is because the structure of matter fields terms is the product of two field $\psi\Upsilon$ 
which share the similar structure with that of Einstein-Maxwell-Complex Scalar system. 
Moreover, the Dirac equation and the equation for $\Upsilon$ have the same null structure and 
can also formulate a Hodge system. 
That is the basis for applying the energy estimate by Luk's strategy. 
Hence we omit most details in the proofs of the lemmas and propositions 
and instead concentrate on the places 
where our arguments deviate or require modification from those in Paper \cite{PengXiaoning25}.

\subsection{Preliminaries and estimates for the components of frame}
In this section we present the inequalities, conventions and the control of frame coefficient which 
are used in the analysis without proof. 
The details and discussions can be found in \cite{Luk12, HilValZha23, PengXiaoning25}. 

We begin with the following control for the components of frame
\begin{lemma}[\textbf{\em control on the metric coefficients}]
Under the following bootstrap assumption
\begin{align}
||\{\rho, \mu, \sigma, \lambda, \tau, \pi, \varrho\}||_{L^{\infty}(\mathcal{S}_{u,v})}\leq\mathcal{O},
\end{align}
then there exists a sufficiently small number $\varepsilon$, for example $\mathcal{O}\varepsilon\ll1$, 
such that 
\begin{align*}
&||Q,Q^{-1},||_{L^{\infty}(\mathcal{S}_{u,v})}\leq
   C(\Delta_{e_{\star}}), \\
&  |P^{\mathcal{A}},(P^{\mathcal{A}})^{-1},C^{\mathcal{A}}|
   \leq C(\Delta_{e_{\star}}), \\
& Area(\mathcal{S}_{u,v})\leq C(\Delta_{e_{\star}}),
\end{align*}
on~$\mathcal{D}_{u,v}$. 
\end{lemma}

Make use of the following integral identities:
\begin{subequations}
\begin{align}
\frac{\mathrm{d}}{\mathrm{d}v}\int_{\mathcal{S}_{u,v}}
f=&\int_{\mathcal{S}_{u,v}}\left (Df-(\rho+\bar\rho)f\right), \label{IntegralIDv}\\
\frac{\mathrm{d}}{\mathrm{d}u}\int_{\mathcal{S}_{u,v}}
f=&\int_{\mathcal{S}_{u,v}}Q^{-1}\left (\Delta f+(\mu+\bar\mu)f\right), \label{IntegralIDu}
\end{align}
\end{subequations}
where $f$ denote an arbitrary quantity with zero T-weight, one obtains the Gr\"onwall type inequality:

\begin{proposition}
\label{GronwalltypeEst}
Assume that
\begin{align*}
||\{\rho,\mu\}||_{L^{\infty}(\mathcal{S}_{u,v})}\leq4\Delta_{\Gamma_{\star}},
\end{align*}
then there exists $\varepsilon_{\star}=\varepsilon_{\star}(\Delta_{e_{\star}},\Delta_{\Gamma_{\star}})$, the following Gr\"onwall-type estimates hold
\begin{subequations}
\begin{align}
||f||_{L^p(\mathcal{S}_{u,v})}\leq&
C(\Delta_{e_{\star}},\Delta_{\Gamma_{\star}})\left(||f||_{L^p(\mathcal{S}_{u,0})}
+\int_0^v||\mthorn f||_{L^p(\mathcal{S}_{u,v'})}\right), \label{LpLongGronwallEst} \\
||f||_{L^p(\mathcal{S}_{u,v})}\leq& 2\left(||f||_{L^p(\mathcal{S}_{0,v})}
+\int_0^u||\mthorn' f||_{L^p(\mathcal{S}_{u',v})}\right). \label{LpShortGronwallEst}
\end{align}
\end{subequations}
where $1\leq p\leq\infty$.
Also we have
\begin{subequations}
\begin{align}
||f||_{L^{\infty}(\mathcal{S}_{u,v})}\leq&
C(\Delta_{e_{\star}},\Delta_{\Gamma_{\star}})\left(||f||_{L^{\infty}(\mathcal{S}_{u,0})}
+\int_0^v||\mthorn f||_{L^{\infty}(\mathcal{S}_{u,v'})}\right), \label{LinftyLongGronwallEst} \\
||f||_{L^{\infty}(\mathcal{S}_{u,v})}\leq& 2\left(||f||_{L^{\infty}(\mathcal{S}_{0,v})}
+\int_0^u||\mthorn' f||_{L^{\infty}(\mathcal{S}_{u',v})}\right). \label{LinftyLongGronwallEst}
\end{align}
\end{subequations}

\end{proposition}

Next we list the necessary results of Sobolev embedding inequality
\begin{proposition}[\textbf{\em Sobolev-type inequality. I}]
\label{Proposition:Sobolev} 
 Let~$f$ be a T-weight quantity
  on~$\mathcal{S}_{u,v}$ which is square-integrable with
  square-integrable first covariant derivatives. Then for
  each~$2<p<\infty$, $f\in L^p(\mathcal{S}_{u,v})$, there
  exists~$\varepsilon_\star=\varepsilon_\star(\Delta_{e_\star},\Delta_{\Gamma_\star})$
  such that as long as~$\varepsilon\leq\varepsilon_\star$, we have
  \begin{align*}
||f||_{L^p(\mathcal{S}_{u,v})}\leq
G_p(\bmsigma)\left(||f||_{L^2(\mathcal{S}_{u,v})}
+||\mathcal{D}f||_{L^2(\mathcal{S}_{u,v})}\right)
\end{align*}
where~$G_p(\bmsigma)$ is a constant also depends on the isoperimetric
constant~$\mathcal{I}(\mathcal{S}_{u,v})$ and~$p$, but is controlled
by some~$C(\Delta_{e_\star})$.
\end{proposition}

\begin{remark}
Note that in the T-weight formalism we have
$||\mathcal{D}f||_{L^2(\mathcal{S})}=||\nablasl T(f)||_{L^2(\mathcal{S})}$,
hence the results here and following in this subsection
are standard embedding results in \cite{Luk12}
and do not introduce extra estimate.
\end{remark}

\begin{proposition}[\textbf{\em Sobolev-type inequality. II}]
  \label{Proposition:EstimateInfinityNorm}There
  exists~$\varepsilon_\star=\varepsilon_\star(\Delta_{e_\star},\Delta_{\Gamma_\star})$
  such that as long as~$\varepsilon\leq\varepsilon_\star$, we have
\begin{align*}
  ||f||_{L^{\infty}(\mathcal{S}_{u,v})}\leq G_{p}(\bmsigma)
  \left(||f||_{L^p(\mathcal{S}_{u,v})}+||
  \mathcal{D}f||_{L^p(\mathcal{S}_{u,v})}\right),
\end{align*}
with~$2<p<\infty$ and~$G_{p}(\bmsigma)\leq C(\Delta_{e_\star})$ as above.
\end{proposition}

\begin{corollary}[\textbf{\em Sobolev-type inequality. III}]
\label{Corollary:SobolevEmbedding} There
  exists~$\varepsilon_\star=\varepsilon_\star(\Delta_{e_\star},\Delta_{\Gamma_\star})$
  such that as long as~$\varepsilon\leq\varepsilon_\star$, we have
\begin{align*}
  &||f||_{L^4(\mathcal{S}_{u,v})}\leq G(\bmsigma)
  \left(||f||_{L^2(\mathcal{S}_{u,v})}
  +||\mathcal{D}f||_{L^2(\mathcal{S}_{u,v})}\right), \\
  &||f||_{L^{\infty}(\mathcal{S}_{u,v})}\leq G(\bmsigma)
  \left(||f||_{L^2(\mathcal{S}_{u,v})}
  +||\mathcal{D}f||_{L^2(\mathcal{S}_{u,v})}
  +||\mathcal{D}^2f||_{L^2(\mathcal{S}_{u,v})}\right),
\end{align*}
again with~$G(\bmsigma)\leq C(\Delta_{e_\star})$.
\end{corollary}

In the end, we present the necessary commutator equations. 
Suppose that the T-weighted quantity~$f$ satisfies the transport
equation~$\mthorn' f=H_0$.  Then, under the coordinate choice one has
\begin{align*}
H_k=\sum_{i_1+i_2+i_3=k}\meth^{i_1}\Gamma(\pi,\tau)^{i_2}\meth^{i_3}H_0+
\sum_{i_1+i_2+i_3+i_4=k}\meth^{i_1}\Gamma(\tau,\pi)^{i_2}\meth^{i_3}\Gamma(\tau,\pi,\mu,\lambda)\meth^{i_4}f,
\end{align*}
where~$H_k\equiv\mthorn'\meth^kf$. 
Similarly, suppose~$f$
satisfies~$\mthorn f=G_0$, one has
\begin{align*}
G_k=\meth^{k}G_0+
\sum_{i=0}^k\meth^i\rho\meth^{k-i}f+
\sum_{i=0}^k\meth^i\sigma\meth^{k-i}f,
\end{align*}
where~$G_k\equiv\mthorn\meth^kf$.

\begin{remark}
In the estimates of the proof, we choose $\meth^kf$ as an example.
That is because the structure of transport equation of any other string
$\{\mthorn,\mthorn'\}\mathcal{D}^{k_i}f$ is the same to that of $\{\mthorn,\mthorn'\}\meth^kf$,
hence the results of $||\meth^kf||$ leads to the estimate for $||\mathcal{D}^kf||$.
\end{remark}

\begin{remark}
We denote~$\meth^{i_1}\Gamma^{i_2}$
as~$\meth^{j_1}\Gamma\meth^{j_2}\Gamma...\meth^{j_{i_2}}\Gamma$
where~$i_1\geq0$, $i_2\geq1$, $j_1$, $j_2$, ...,
$j_{i_2}\in\mathbb{N}$ and~$j_1+j_2+...+j_{i_2}=i_1$.
\end{remark}

\subsection{Estimates of next-to-leading order derivative}
\label{next-to-leading}
In this section we focus on the estimate of next-to-leading order derivative on $\mathcal{S}_{u,v}$.

\subsubsection{Estimate for the connection coefficients}
\begin{proposition}
\label{inftynormalconnection}
Assume the boundedness of the following
\begin{align*}
&\sum_{i=2}^3\sup_{u,v}||\mathcal{D}^i\tau||_{L^2(\mathcal{S}_{u,v})}, \quad
\sup_{v}||\mathcal{D}^4\tau||_{L^2(\mathcal{N}'_{v})}, \\
&\Delta_{\psi}(\mathcal{S}), \quad \Delta_{\Upsilon}(\mathcal{S}), 
 \quad \Delta_{\Psi}(\mathcal{S})
,\quad \Delta_{\psi},\quad \Delta_{\Upsilon},
\quad  \Delta_{\Psi},
\end{align*}
then there exists sufficiently small $\varepsilon_{\star}$ depends on 
\begin{align*}
&\Delta_{e_{\star}},\quad \Delta_{\Gamma_{\star}},\quad
\sum_{l=2}^3||\mathcal{D}^l\tau||_{L^{2}(\mathcal{S})},\quad ||\mathcal{D}^4\tau||_{L^2(\mathcal{N}'_{v})}, \\
&\Delta_{\psi}(\mathcal{S}), \quad \Delta_{\Upsilon}(\mathcal{S}), 
 \quad \Delta_{\Psi}(\mathcal{S}),
 \quad \Delta_{\Upsilon},
\quad  \Delta_{\Psi},
\end{align*}
such that when $\varepsilon\leq\varepsilon_{\star}$, for $i=0,1$, we have
\begin{align*}
&\sup_{u,v}||\mathcal{D}^i\{\tau,\varrho\}||_{L^{\infty}(\mathcal{S}_{u,v})}
\leq C(\Delta_{e_{\star}},\Delta_{\Gamma_{\star}},
\Delta_{\psi}(\mathcal{S}),\Delta_{\Upsilon}(\mathcal{S})
,\Delta_{\Psi}(\mathcal{S}),\Delta_{\Psi}),\\
&\sup_{u,v}||\mathcal{D}^i\{\rho,\sigma,\mu,\lambda,\omega,\pi\}||_{L^{\infty}(\mathcal{S}_{u,v})}
\leq3\Delta_{\Gamma_{\star}}.
\end{align*}
\end{proposition}
\begin{proof}
The schematic equation for connections is 
\begin{align*}
\{\mthorn,\mthorn'\}\Gamma-\meth\Gamma=m\psi^2+\Upsilon\psi+\Gamma\Gamma+\Psi.
\end{align*}
We focus on the terms contain matter field.
For $\tau$, we make use of its long direction equations \eqref{thorntau} and need to estimate
\begin{align*}
zero-deriv: \Upsilon_i\psi_j, \quad 
1st-deriv: \meth\Upsilon_i\psi_j+\Upsilon_i\meth\psi_j
\end{align*}
and have
\begin{align*}
||\meth^k\Upsilon_i\meth^{p-k}\psi_j||_{L^{\infty}(\mathcal{S})}\leq&
C(\Delta_{e_{\star}})\left(\sum_{l=0}^3||\mathcal{D}^l\Upsilon_i||_{L^{2}(\mathcal{S})}\right)
\left(\sum_{l=0}^3||\mathcal{D}^l\psi_j||_{L^{2}(\mathcal{S})}\right) \\
\leq&C(\Delta_{e_{\star}},\Delta_{\Upsilon}(\mathcal{S}),\Delta_{\psi}(\mathcal{S}))
\end{align*}
where $p\leq1$. Then we have
\begin{align*}
||\mathcal{D}^{i\leq1}\tau||_{L^{\infty}(\mathcal{S}_{u,v})}\leq
C(\Delta_{e_{\star}},\Delta_{\Gamma_{\star}},
\Delta_{\psi}(\mathcal{S}),\Delta_{\Upsilon}(\mathcal{S})
,\Delta_{\Psi}(\mathcal{S}),\Delta_{\Psi}).
\end{align*}
The analysis for $\varrho$ is similar.

For $\rho$, $\sigma$, $\mu$, $\lambda$, $\omega$ and $\pi$, we make use of 
their short direction equations, i.e. \eqref{thornprimerho},\eqref{thornprimesigma}, 
\eqref{thornprimemu}, \eqref{thornprimelambda}, \eqref{thornprimeomega} and \eqref{thornprimepi}. 
The analysis is similar.
Specifically, for terms $\zeta_5$ and $\eta_5$, we make use of their norms on the ingoing lightcone, 
then we obtain
\begin{align*}
||\mathcal{D}^{i\leq1}\{\rho,\sigma,\mu,\lambda,\omega,\pi\}||_{L^{\infty}(\mathcal{S}_{u,v})}\leq&
2\Delta_{\Gamma_{\star}}+
C(\Delta_{e_{\star}},\Delta_{\Psi},\Delta_{\Upsilon},
||\mathcal{D}^4\tau||_{L^2(\mathcal{N}'_v)},)\varepsilon^{\frac{1}{2}} \\
&+C(\Delta_{e_{\star}},\Delta_{\Gamma_{\star}},\sum_{l=2}^3||\mathcal{D}^l\tau||_{L^{2}(\mathcal{S})},
\Delta_{\Psi}(\mathcal{S}),\Delta_{\psi}(\mathcal{S}),\Delta_{\Upsilon}(\mathcal{S}))\varepsilon. 
\end{align*}

\end{proof}

\begin{remark}
We can always replace the derivative of $\psi$ with $\Upsilon$ and hence 
here we only need the norm of $\psi$ on sphere.
\end{remark}

Follow the same strategy one can obtain the estimates for $L^4$ and $L^2$ norm for connections, 
we show the results in the following two propositions.
\begin{proposition}
\label{L4normalconnection}
Make the same assumption as in Prop. \ref{inftynormalconnection}
then there exists sufficiently small $\varepsilon_{\star}$ depends on 
\begin{align*}
&\Delta_{e_{\star}},\quad \Delta_{\Gamma_{\star}},\quad
\sum_{l=2}^3||\mathcal{D}^l\tau||_{L^{2}(\mathcal{S})},\quad ||\mathcal{D}^4\tau||_{L^2(\mathcal{N}'_{v})}, \\
&\Delta_{\psi}(\mathcal{S}), \quad \Delta_{\Upsilon}(\mathcal{S}), 
 \quad \Delta_{\Psi}(\mathcal{S}),
 \quad \Delta_{\Upsilon},
\quad  \Delta_{\Psi},
\end{align*}
such that when $\varepsilon\leq\varepsilon_{\star}$, for $i=1,2$, we have
\begin{align*}
&\sup_{u,v}||\mathcal{D}^i\{\tau,\varrho\}||_{L^{4}(\mathcal{S}_{u,v})}
\leq C(\Delta_{e_{\star}},\Delta_{\Gamma_{\star}},
\Delta_{\psi}(\mathcal{S}),\Delta_{\Upsilon}(\mathcal{S})
,\Delta_{\Psi}(\mathcal{S}),\Delta_{\Psi}),\\
&\sup_{u,v}||\mathcal{D}^i\{\rho,\sigma,\mu,\lambda,\omega,\pi\}||_{L^{4}(\mathcal{S}_{u,v})}
\leq3\Delta_{\Gamma_{\star}}.
\end{align*}
\end{proposition}

\begin{proposition}
\label{L2normalconnection}
Assume the boundedness of the following
\begin{align*}
\sup_{v}||\mathcal{D}^4\tau||_{L^2(\mathcal{N}'_{v})}, \quad
\Delta_{\psi}(\mathcal{S}), \quad \Delta_{\Upsilon}(\mathcal{S}), 
 \quad \Delta_{\Psi}(\mathcal{S})
,\quad \Delta_{\psi},\quad \Delta_{\Upsilon},
\quad  \Delta_{\Psi},
\end{align*}
then there exists sufficiently small $\varepsilon_{\star}$ depends on 
\begin{align*}
&\Delta_{e_{\star}},\quad \Delta_{\Gamma_{\star}},
\quad ||\mathcal{D}^4\tau||_{L^2(\mathcal{N}'_{v})}, \\
&\Delta_{\psi}(\mathcal{S}), \quad \Delta_{\Upsilon}(\mathcal{S}), 
 \quad \Delta_{\Psi}(\mathcal{S}),
 \quad \Delta_{\Upsilon},
\quad  \Delta_{\Psi},
\end{align*}
such that when $\varepsilon\leq\varepsilon_{\star}$, for $i=2,3$, we have
\begin{align*}
&\sup_{u,v}||\mathcal{D}^i\{\tau,\varrho\}||_{L^{2}(\mathcal{S}_{u,v})}
\leq C(\Delta_{e_{\star}},\Delta_{\Gamma_{\star}},
\Delta_{\psi}(\mathcal{S}),\Delta_{\Upsilon}(\mathcal{S})
,\Delta_{\Psi}(\mathcal{S}),\Delta_{\Psi}),\\
&\sup_{u,v}||\mathcal{D}^i\{\rho,\sigma,\mu,\lambda,\omega,\pi\}||_{L^{2}(\mathcal{S}_{u,v})}
\leq3\Delta_{\Gamma_{\star}}.
\end{align*}
\end{proposition}

We gather the estimates of connection coefficients that we have obtained: 
\begin{proposition}
\label{NexttoLeadingConnection}
Assume the boundedness of the following
\begin{align*}
\sup_{v}||\mathcal{D}^4\tau||_{L^2(\mathcal{N}'_{v})}, \quad
\Delta_{\psi}(\mathcal{S}), \quad \Delta_{\Upsilon}(\mathcal{S}), 
\quad \Delta_{\Psi}(\mathcal{S})
,\quad \Delta_{\psi},\quad \Delta_{\Upsilon},
\quad  \Delta_{\Psi},
\end{align*}
then there exists sufficiently small $\varepsilon_{\star}$ depends on 
\begin{align*}
\Delta_{e_{\star}},\quad \Delta_{\Gamma_{\star}},
\quad ||\mathcal{D}^4\tau||_{L^2(\mathcal{N}'_{v})}, \quad
\Delta_{\psi}(\mathcal{S}), \quad \Delta_{\Upsilon}(\mathcal{S}), 
 \quad \Delta_{\Psi}(\mathcal{S}),
\quad \Delta_{\Upsilon},
\quad  \Delta_{\Psi},
\end{align*}
such that when $\varepsilon\leq\varepsilon_{\star}$, we have
\begin{align*}
&\sup_{u,v}\left(\sum_{i=0}^1||\mathcal{D}^i\{\tau,\varrho\}||_{L^{\infty}(\mathcal{S}_{u,v})}
+\sum_{i=1}^2||\mathcal{D}^i\{\tau,\varrho\}||_{L^{4}(\mathcal{S}_{u,v})}
+\sum_{i=2}^3||\mathcal{D}^i\{\tau,\varrho\}||_{L^{2}(\mathcal{S}_{u,v})}\right)\\
\leq& C(\Delta_{e_{\star}},\Delta_{\Gamma_{\star}}
,\Delta_{\psi}(\mathcal{S}),\Delta_{\Upsilon}(\mathcal{S}),
\Delta_{\Psi}(\mathcal{S}),\Delta_{\Psi}), \\
&\sup_{u,v}(\sup_{i=0,1}||\mathcal{D}^i\{\rho,\sigma,\mu,\lambda,\omega,\pi\}||_{L^{\infty}(\mathcal{S}_{u,v})},
\sup_{i=1,2}||\mathcal{D}^i\{\rho,\sigma,\mu,\lambda,\omega,\pi\}||_{L^{4}(\mathcal{S}_{u,v})}, \\
&\quad \quad 
\sup_{i=2,3}||\mathcal{D}^i\{\rho,\sigma,\mu,\lambda,\omega,\pi\}||_{L^{2}(\mathcal{S}_{u,v})} )
\leq3\Delta_{\Gamma_{\star}}.
\end{align*}
\end{proposition}

\subsubsection{$L^2(\mathcal{S})$ estimate for the matter fields}

\begin{proposition}
\label{L2estimatephiAchiA}
Assume the boundedness of the following
\begin{align*}
\sup_{v}||\mathcal{D}^4\tau||_{L^2(\mathcal{N}'_{v})}, 
 \quad \Delta_{\Upsilon}(\mathcal{S}), \quad \Delta_{\Psi}(\mathcal{S})
,\quad \Delta_{\psi},\quad \Delta_{\Upsilon},
\quad  \Delta_{\Psi},
\end{align*}
then there exists and $\varepsilon_{\star}$ depends on 
\begin{align*}
&\Delta_{e_{\star}},\quad \Delta_{\Gamma_{\star}},\quad \Delta_{\psi_{\star}}, \quad
 \Delta_{\Upsilon}(\mathcal{S}),\quad \Delta_{\Psi}(\mathcal{S}),
\quad \Delta_{\psi},\quad  \Delta_{\Psi},
\end{align*} 
such that when $\varepsilon\leq\varepsilon_{\star}$, we have
\begin{align*}
\Delta_{\psi}(\mathcal{S})\leq C(\Delta_{e_{\star}},\Delta_{\Gamma_{\star}},\Delta_{\psi_{\star}},\Delta_{\psi}).
\end{align*}

\end{proposition}

\begin{proof}
We begin with $\phi_0$ and $\phi_1$ by using
\begin{align*}
\mthorn \phi_1- \meth' \phi_0 = \frac{\phi_0 \pi}{2} + \phi_1 \rho - \frac{\phi_1 \omega}{2}-m\bar\chi_0 , \quad
\mthorn' \phi_0 -\meth \phi_1= -\phi_0 \mu + \frac{\phi_1 \bar{\pi}}{2} - \phi_1 \tau +m\bar\chi_1
\end{align*}
and for $i\leq3$ we obtain 
\begin{align*}
\mthorn'\meth^i\phi_0=&\sum_{i_1+i_2+i_3=i}\meth^{i_1}\Gamma^{i_2}(\meth^{i_3+1}\phi_1,m\meth^i\bar\chi_1)+
\sum_{i_1+...+i_4=i=i}\meth^{i_1}\Gamma^{i_2}\meth^{i_3}\Gamma\meth^{i_4}\phi_j, \\
\mthorn\meth^i\phi_1=&\meth^{i+1}\phi_0-m\meth^i\bar\chi_0+\sum_{i_1+i_2=i}\meth^{i_1}\Gamma\meth^{i_2}\phi_i
\end{align*}
Then we have
\begin{align*}
||\meth^i\phi_0||_{L^2(\mathcal{S}_{u,v})}\leq&2\Delta_{\phi_{\star}}+
||\meth^{i+1}\phi_1||_{\mathcal{N}'_{v}}\varepsilon^{1/2}+
C(\Delta_{e_{\star}},\Delta_{\Gamma_{\star}},\Delta_{\psi_{\star}},\Delta_{\psi}(\mathcal{S}),
\Delta_{\Upsilon}(\mathcal{S}),\Delta_{\Psi}(\mathcal{S}),\Delta_{\Psi})\varepsilon, \\
||\meth^i\phi_1||_{L^2(\mathcal{S}_{u,v})}\leq&C(\Delta_{e_{\star}},\Delta_{\Gamma_{\star}},\Delta_{\psi_{\star}},\Delta_{\psi})
\end{align*}
The analysis of $\chi_0$ and $\chi_1$ is the same and
hence we finish the proof.

\end{proof}

\begin{proposition}
\label{L2estimateUpsilon}
Assume the boundedness of the following
\begin{align*}
\sup_{v}||\mathcal{D}^4\tau||_{L^2(\mathcal{N}'_{v})}, 
  \quad \Delta_{\Psi}(\mathcal{S})
,\quad \Delta_{\psi},\quad \Delta_{\Upsilon},
\quad  \Delta_{\Psi},
\end{align*}
then there exists and $\varepsilon_{\star}$ depends on 
\begin{align*}
&\Delta_{e_{\star}},\quad \Delta_{\Gamma_{\star}},\quad \Delta_{\psi_{\star}},
\quad \Delta_{\Upsilon_{\star}}, \quad
 \Delta_{\Psi}(\mathcal{S}),
\quad \Delta_{\psi},\quad \Delta_{\Upsilon},\quad  \Delta_{\Psi},
\end{align*} 
such that when $\varepsilon\leq\varepsilon_{\star}$, we have
\begin{align*}
\Delta_{\Upsilon}(\mathcal{S})\leq C(\Delta_{e_{\star}},\Delta_{\Gamma_{\star}},\Delta_{\psi_{\star}},\Delta_{\Upsilon_{\star}},\Delta_{\psi},\Delta_{\Upsilon}).
\end{align*}

\end{proposition}

\begin{proof}
Take $\zeta_i$ as an example.
We make use of the short direction equations, \eqref{thornprimezeta0}, 
 \eqref{thornprimezeta1},  \eqref{thornprimezeta2},  \eqref{thornprimezeta3} and  \eqref{thornprimezeta4} 
 for $\zeta_{0,1,2,3,4}$, and the long direction equations \eqref{thornzeta5} for $\zeta_{5}$. 
 The schematic form of such equations are
\begin{align*}
\{\mthorn,\mthorn'\}\Upsilon-\meth\Upsilon=m\Upsilon+m^2\psi+m\psi^2+\Gamma\Upsilon+\Upsilon\psi^2,
\end{align*}
For $\zeta_{0,1,2,3,4}$ we have
\begin{align*}
\mthorn'\meth^{i}\zeta_j=&\sum_{i_i+...+i_3=i}\meth^{i_1}\Gamma^{i_2}\meth^{i_3+1}\zeta_k
+\sum_{i_i+...+i_5=i}\meth^{i_1}\Gamma^{i_2}\meth^{i_3}\zeta_{k_1}\meth^{i_4}\phi_{k_2}\meth^{i_5}\phi_{k_3} \\
&+\sum_{i_i+...+i_4=i}\meth^{i_1}\Gamma^{i_2}\meth^{i_3}\zeta_{k}\meth^{i_4}\Gamma
+\sum_{i_i+...+i_4=i}\meth^{i_1}\Gamma^{i_2}\meth^{i_3}\Psi_{k_1}\meth^{i_4}\phi_{k_2} \\
&+\sum_{i_i+...+i_3=i}\meth^{i_1}\Gamma^{i_2}\meth^{i_3}(m\eta_k,m^2\phi_k)
+m\sum_{i_i+...+i_5=i}\meth^{i_1}\Gamma^{i_2}\meth^{i_3}\chi_{k_1}\meth^{i_4}\phi_{k_2}\meth^{i_5}\phi_{k_3} \\
&+\sum_{i_i+...+i_4=i}\meth^{i_1}\Gamma^{i_2}\meth^{i_3}\phi_j\meth^{i_4}\Psi_l
\end{align*}
For $\zeta_5$, we have
\begin{align*}
\mthorn\meth^i\zeta_5=&\meth^{i}\meth'\zeta_4+
\sum_{i_1+i_2+i_3=i}\meth^{i_1}\zeta_{j_1}\meth^{i_2}\phi_{j_2}\meth^{i_3}\phi_{j_3}
+\sum_{i_1+i_2=i}\meth^{i_1}\zeta_{j_1}\meth^{i_2}\Gamma \\
&+\meth^{i}(m\eta_k,m^2\phi_k)
+\sum_{i_1+i_2+i_3=i}\meth^{i_1}\phi_{j_1}\meth^{i_2}\phi_{j_2}\meth^{i_3}\chi_{j_3}
\end{align*}
Note that although there are $\Psi_4$ and $\TiPsi_3$ in equation $\mthorn'\zeta_2$, 
we estimate the next-to-leading derivative 
and hence the requirement for curvature is up to 3. 
One can translate the norm 
of curvature to the ingoing cone. 
There is no $\tau$ in $\mthorn\zeta_5$.
We make use of the norm for $\meth^{i+1}\zeta_{(1,2,4,5)}$ on the ingoing lightcone, 
and norm for $\meth^{i+1}\zeta_4$ on the outgoing lightcone
then with the results in previous propositions, we obtain 
\begin{align*}
||\meth^i\zeta_{(0,1,2,3,4)}||_{L^2(\mathcal{S}_{u,v})}\leq&2\Delta_{\Upsilon_{\star}}+
C(\Delta_{\psi_{\star}})
||\meth^{i+1}\zeta_{(1,2,4,5)},\meth^i\Psi_k||_{L^2(\mathcal{N}'_{v})}\varepsilon^{1/2}\\
&+C(\Delta_{e_{\star}},\Delta_{\Gamma_{\star}},\Delta_{\psi_{\star}},\Delta_{\Upsilon_{\star}},\Delta_{\psi},
\Delta_{\Psi}(\mathcal{S}),\Delta_{\Psi})\varepsilon, \\
||\meth^i\zeta_5||_{L^2(\mathcal{S}_{u,v})}\leq&C(\Delta_{e_{\star}},\Delta_{\Gamma_{\star}},\Delta_{\psi_{\star}},\Delta_{\Upsilon_{\star}},\Delta_{\psi},\Delta_{\Upsilon}).
\end{align*}
The analysis for $\eta_k$ is the same. Hence we finish the proof.

\end{proof}

\subsubsection{$L^2(\mathcal{S})$ Estimate for the Weyl curvature}

\begin{proposition}
\label{L2estimateWeyl}
Assume the boundedness of the following
\begin{align*}
\sup_{v}||\mathcal{D}^4\tau||_{L^2(\mathcal{N}'_{v})}
,\quad \Delta_{\psi},\quad \Delta_{\Upsilon},
\quad  \Delta_{\Psi},
\end{align*}
then there exists and $\varepsilon_{\star}$ depends on 
\begin{align*}
\Delta_{e_{\star}},\quad \Delta_{\Gamma_{\star}},\quad \Delta_{\psi_{\star}},
\quad \Delta_{\Upsilon_{\star}},\quad
\Delta_{\Psi_{\star}}, \quad \Delta_{\psi},\quad \Delta_{\Upsilon}, \quad
\Delta_{\Psi},
\end{align*}
such that when $\varepsilon\leq\varepsilon_{\star}$, we have
\begin{align*}
\sup_{u,v}\sup_{i=0,1,2}||\mathcal{D}^i\{\Psi_0,\TiPsi_1,\TiPsi_2,\TiPsi_3\}||_{L^2(\mathcal{S}_{u,v})}\leq3\Delta_{\Psi_{\star}}.
\end{align*}

\end{proposition}

\begin{proof}
The schematic form of Bianchi identities for $\Psi_0,...,\TiPsi_3$ is
\begin{align*}
\mthorn'\Psi_i-\meth\Psi_j=&m\Upsilon\psi+m\psi^2\Gamma
+\Upsilon\psi\Gamma+\psi\meth\Upsilon+\Psi\psi^2+\Upsilon\psi^3+\Upsilon^2+\Gamma\Psi_k, 
\end{align*}
Follow the similar method, for $i\leq2$ we have
\begin{align*}
||\meth^i\Psi_j||_{L^2(\mathcal{S}_{u,v})}\leq&2\Delta_{\Psi_{\star}}+
||\meth^{i+1}\{\TiPsi_{1,2,3},\Psi_4\}||_{L^2(\mathcal{N}'_{v})}\varepsilon^{1/2}\\
&+C(\Delta_{e_{\star}},\Delta_{\Gamma_{\star}},\Delta_{\psi_{\star}},\Delta_{\Upsilon_{\star}},\Delta_{\Psi_{\star}},\Delta_{\psi},\Delta_{\Upsilon},\Delta_{\Psi})(\varepsilon+\varepsilon^{1/2}).
\end{align*}
Here we make use of the estimates in previous propositions.

\end{proof}

\subsection{Elliptic estimates}
\label{EllipticEst}
In this section we estimate the top-dervative of connection coefficients.
We first list the necessary results for elliptic estimate.
\begin{proposition}
\label{EllipticTweightnonzero}

Let~$f$ denote a nonzero Tweight quantity and suppose that
\begin{align*}
\sum_{i=0}^{k-2}||\mathcal{D}^iK||_{L^2(\mathcal{S})}\leq\infty,
\end{align*}
then make use of the results in \ref{next-to-leading}, for~$0\leq k\leq 4$,
one has that
\begin{align*}
||\mathcal{D}^kf||_{L^2(\mathcal{S})}\leq
C(\sum_{i=0}^{k-2}||\mathcal{D}^iK||_{L^2(\mathcal{S})},\Delta_{e_{\star}})
\sum_{j=0}^{k-1}\left(||\mathcal{D}^j\mathscr{D}_f||_{L^2(\mathcal{S})}+
||\mathcal{D}^jf||_{L^2(\mathcal{S})}\right).
\end{align*}

\end{proposition}

\begin{proposition}
  \label{EllipticPureScalar}
  Let~$f$ denote a quantity with zero
  T-weight. Then make use of the results in \ref{next-to-leading} and
  for~$0\leq k\leq 4$, one has that
\begin{align*}
||\mathcal{D}^{k}f||_{L^2(\mathcal{S})}\leq C(\sum_{i=0}^{k-2}||\mathcal{D}^iK||_{L^2(\mathcal{S})},\Delta_{e_{\star}})
\left(||\mathcal{D}^{k-2}(\bmDelta f)||_{L^2(\mathcal{S})}
+\sum_{i=0}^{k-1}||\mathcal{D}^if||_{L^2(\mathcal{S})}\right),
\end{align*}
where~$\bmDelta f\equiv 2\meth\meth'f$.
\end{proposition}

\begin{proposition}
\label{L2estimateGausscurv}
Assume the boundedness of the following
\begin{align*}
\sup_{v}||\mathcal{D}^4\tau||_{L^2(\mathcal{N}'_{v})}
,\quad \Delta_{\psi},\quad \Delta_{\Upsilon},
\quad  \Delta_{\Psi},
\end{align*}
then there exists and $\varepsilon_{\star}$ depends on
\begin{align*}
&\Delta_{e_{\star}},\quad \Delta_{\Gamma_{\star}},\quad \Delta_{\psi_{\star}},
\quad \Delta_{\Upsilon_{\star}},\quad 
\Delta_{\Psi_{\star}}, \\
& \Delta_{\Psi},\quad \Delta_{\psi},\quad \Delta_{\Upsilon},
\end{align*}
such that when $\varepsilon\leq\varepsilon_{\star}$, we have
\begin{align*}
\sum_{i=0}^2\sup_{u,v}||\mathcal{D}^iK||_{L^2(\mathcal{S}_{u,v})}\leq
C(\Delta_{e_{\star}},\Delta_{\Gamma_{\star}},\Delta_{\psi_{\star}},\Delta_{\Upsilon_{\star}},\Delta_{\Psi_{\star}}).
\end{align*}
and
\begin{align*}
&\sum_{i=0}^3\sup_{u,v}||\mathcal{D}^iK||_{L^2(\mathcal{N}_u)}\leq
C(\Delta_{e_{\star}},\Delta_{\Gamma_{\star}},\Delta_{\psi_{\star}},\Delta_{\Upsilon_{\star}},\Delta_{\Psi_{\star}},\Delta_{\Psi}), \\
&\sum_{i=0}^3\sup_{u,v}||\mathcal{D}^iK||_{L^2(\mathcal{N}'_v)}\leq
C(\Delta_{\Psi}).
\end{align*}
\end{proposition}
\begin{proof}
Make use of the expression of the Gaussian curvature:
\begin{align}
K&= 2 \mathrm{i}(\bar\zeta_4\phi_0-\zeta_4\bar\phi_0+\bar\zeta_1\phi_1-\zeta_1\bar\phi_1-\bar\eta_4\chi_0+\eta_4\bar\chi_0-\bar\eta_1\chi_1+\eta_1\bar\chi_1) \nonumber\\
&\quad+2\mathrm{i}(-m\phi_0\chi_1+m\phi_1\chi_0+m\bar\phi_0\chi_1-\bar\phi_1\bar\chi_0) - \TiPsi_2 - \bar{\TiPsi}_2 
+ 2\mu \rho - \lambda \sigma - \bar{\lambda} \bar{\sigma}
\end{align}
and the estimate results in last section.
\end{proof}

With the elliptic inequality one can then estimate the top-derivative of connections in the following propositions
\begin{proposition}
\label{4derivativerpi}
Assume the boundedness of the following
\begin{align*}
 \Delta_{\psi},\quad \Delta_{\Upsilon},
\quad  \Delta_{\Psi},
\end{align*}
 then there exists a sufficiently small $\varepsilon_{\star}$ 
depending on 
\begin{align*}
&\Delta_{e_{\star}},\quad \Delta_{\Gamma_{\star}},\quad \Delta_{\psi_{\star}},
\quad \Delta_{\Upsilon_{\star}},\quad 
\Delta_{\Psi_{\star}}, \quad
 \Delta_{\Psi},\quad \Delta_{\psi},\quad \Delta_{\Upsilon}, 
\end{align*}
such that when $\varepsilon\leq\varepsilon_{\star}$, the following hold
\begin{align*}
||\mathcal{D}^4\pi||_{L^2(\mathcal{N}_u)},||\mathcal{D}^4\pi||_{L^2(\mathcal{N}'_v)}\leq 
C(\Delta_{e_{\star}},\Delta_{\Gamma_{\star}},\Delta_{\psi_{\star}},\Delta_{\Upsilon_{\star}},\Delta_{\Psi_{\star}},\Delta_{\Psi}).
\end{align*}
\end{proposition}

\begin{proof}
Define 
\begin{align*}
\Tipi\equiv\TiPsi_2+\mathscr{D}_{\pi}=\TiPsi_2+\meth\pi.
\end{align*}
\begin{align*}
\mthorn'\Tipi=2\mathrm{i}(\bar\phi_1\meth'\zeta_4-\bar\chi_1\meth'\eta_4+\chi_1\meth\bar\eta_4
-\phi_1\meth\bar\zeta_4)+m\Upsilon_j\psi_k+m\Gamma\psi_j^2
+ \Upsilon_i\psi_j^3+\Upsilon_j^2+
\Upsilon_j\psi_k\Gamma+V
\end{align*}
Here $V$ means the vacuum case, see
We have 
\begin{align*}
||\mathcal{D}^i\Tipi||_{L^2(\mathcal{S}_{u,v})}\leq&C(\Delta_{\Gamma_{\star}},\Delta_{\Psi_{\star}})
+C(\Delta_{e_{\star}})\int_0^u||\mathcal{D}^i\Tipi||_{L^2(\mathcal{S}_{u',v})} \\
&+C(\Delta_{e_{\star}},\Delta_{\Gamma_{\star}},\Delta_{\psi_{\star}},\Delta_{\Upsilon_{\star}},
\Delta_{\Psi_{\star}},\Delta_{\Psi},\Delta_{\psi},\Delta_{\Upsilon}, \mathcal{O}_{4,2})(\varepsilon^{1/2}+\varepsilon) 
\leq C(\Delta_{\Gamma_{\star}},\Delta_{\Psi_{\star}}).
\end{align*}
Now we can make use of Prop. \ref{EllipticTweightnonzero} and obtain 
\begin{align*}
||\mathcal{D}^4\pi||_{L^2(\mathcal{S})}\leq&
C(\sum_{i=0}^{2}||\mathcal{D}^iK||_{L^2(\mathcal{S})},\Delta_{e_{\star}})
\sum_{j=0}^{3}\left(||\mathcal{D}^j\mathscr{D}_{\pi}||_{L^2(\mathcal{S})}+
||\mathcal{D}^j\pi||_{L^2(\mathcal{S})}\right) \\
\leq&C(\Delta_{e_{\star}},\Delta_{\Gamma_{\star}},\Delta_{\psi_{\star}},\Delta_{\Upsilon_{\star}},\Delta_{\Psi_{\star}})
(||\mathcal{D}^3\TiPsi_2||_{L^2(\mathcal{S}_{u,v})}+1).
\end{align*}
Then integral along the light cone we obtain
\begin{align*}
||\mathcal{D}^4\pi||_{L^2(\mathcal{N}_u)},||\mathcal{D}^4\pi||_{L^2(\mathcal{N}'_v)}\leq 
C(\Delta_{e_{\star}},\Delta_{\Gamma_{\star}},\Delta_{\psi_{\star}},\Delta_{\Upsilon_{\star}},\Delta_{\Psi_{\star}},\Delta_{\Psi}).
\end{align*}

\end{proof}

\begin{proposition}
\label{4derivativeomega}
Assume the boundedness of the following
\begin{align*}
 \Delta_{\psi},\quad \Delta_{\Upsilon},
\quad  \Delta_{\Psi},
\end{align*}
 then there exists a sufficiently small $\varepsilon_{\star}$ 
depending on 
\begin{align*}
&\Delta_{e_{\star}},\quad \Delta_{\Gamma_{\star}},\quad \Delta_{\psi_{\star}},
\quad \Delta_{\Upsilon_{\star}},\quad 
\Delta_{\Psi_{\star}}, \quad
 \Delta_{\Psi},\quad \Delta_{\psi},\quad \Delta_{\Upsilon}, 
\end{align*}
such that when $\varepsilon\leq\varepsilon_{\star}$, the following hold
\begin{align*}
\sup_{u}||\mathcal{D}^4\omega||_{L^2(\mathcal{N}_u)},\sup_{v}||\mathcal{D}^4\omega||_{L^2(\mathcal{N}'_v)}\leq 
C(\Delta_{e_{\star}},\Delta_{\Gamma_{\star}},\Delta_{\psi_{\star}},\Delta_{\Upsilon_{\star}},\Delta_{\Psi_{\star}},\Delta_{\Psi}).
\end{align*}
\end{proposition}

\begin{proof}

First we construct an auxiliary function $\omegadg$ with zero T-weight through the relation
\begin{align*}
\mthorn'\omegadg=\mathrm{i}(\TiPsi_2-\bar\TiPsi_2)
\end{align*}
with trivial initial data on $\mathcal{N}_{\star}$. Note here $\omegadg$ is real. 
Then define another function $\Tiomega$ by 
\begin{align*}
\Tiomega\equiv\meth\omega+\mathrm{i}\meth\omegadg+2\TiPsi_1.
\end{align*}
and we have
\begin{align*}
\mthorn'\Tiomega=\phi_j\meth\Upsilon_k+m\Upsilon_j\psi_k+m\Gamma\psi_j^2
+ \Upsilon_i\psi_j^3+\Upsilon_j^2+
\Upsilon_j\psi_k\Gamma+V
\end{align*}
similarly we obtain
\begin{align*}
||\mathcal{D}^i\Tiomega||_{L^2(\mathcal{S}_{u,v})}\leq&C(\Delta_{\Gamma_{\star}},\Delta_{\Psi_{\star}})
+C(\Delta_{e_{\star}})\int_0^u||\mathcal{D}^i\Tiomega||_{L^2(\mathcal{S}_{u',v})} \\
&+C(\Delta_{e_{\star}},\Delta_{\Gamma_{\star}},\Delta_{\psi_{\star}},\Delta_{\Upsilon_{\star}},
\Delta_{\Psi_{\star}},\Delta_{\Psi},\Delta_{\psi},\Delta_{\Upsilon}, \mathcal{O}_{4,2})(\varepsilon^{1/2}+\varepsilon) 
\leq C(\Delta_{\Gamma_{\star}},\Delta_{\Psi_{\star}}).
\end{align*}
Then making use of the elliptic results Prop. \ref{EllipticPureScalar} we obtain
\begin{align*}
||\mathcal{D}^{4}\omega||_{L^2(\mathcal{S})}\leq&C(\sum_{i=0}^{k-2}||\mathcal{D}^iK||_{L^2(\mathcal{S})},\Delta_{e_{\star}})
\left(||\mathcal{D}^{2}(\bmDelta\omega)||_{L^2(\mathcal{S})}
+\sum_{i=0}^{3}||\mathcal{D}^i\omega||_{L^2(\mathcal{S})}\right) \\
\leq&C(\Delta_{e_{\star}},\Delta_{\Gamma_{\star}},\Delta_{\psi_{\star}},\Delta_{\Upsilon_{\star}},\Delta_{\Psi_{\star}})
(||\mathcal{D}^{2}(\meth'\meth\omega+\mathrm{i}\meth'\meth\omegadg)||_{L^2(\mathcal{S})}+C(\Delta_{e_{\star}})) \\
\leq&C(\Delta_{e_{\star}},\Delta_{\Gamma_{\star}},\Delta_{\psi_{\star}},\Delta_{\Upsilon_{\star}},\Delta_{\Psi_{\star}})
\left(||\mathcal{D}^3\Tiomega||_{L^2(\mathcal{S}_{u,v})}+||\mathcal{D}^3\TiPsi_1||_{L^2(\mathcal{S}_{u,v})}\right) \\
\leq&C(\Delta_{e_{\star}},\Delta_{\Gamma_{\star}},\Delta_{\psi_{\star}},\Delta_{\Upsilon_{\star}},\Delta_{\Psi_{\star}})
\left(||\mathcal{D}^3\TiPsi_1||_{L^2(\mathcal{S}_{u,v})}+1\right) .
\end{align*}
Then we can integral along the light cone and obtain 
\begin{align*}
||\mathcal{D}^4\omega||_{L^2(\mathcal{N}_u)},||\mathcal{D}^4\omega||_{L^2(\mathcal{N}'_v)}\leq 
C(\Delta_{e_{\star}},\Delta_{\Gamma_{\star}},\Delta_{\psi_{\star}},\Delta_{\Upsilon_{\star}},\Delta_{\Psi_{\star}},\Delta_{\Psi}).
\end{align*}

\end{proof}

\begin{proposition}
\label{4derivativemulambda}
Assume the boundedness of the following
\begin{align*}
 \Delta_{\psi},\quad \Delta_{\Upsilon},
\quad  \Delta_{\Psi},
\end{align*}
 then there exists a sufficiently small $\varepsilon_{\star}$ 
depending on 
\begin{align*}
&\Delta_{e_{\star}},\quad \Delta_{\Gamma_{\star}},\quad \Delta_{\psi_{\star}},
\quad \Delta_{\Upsilon_{\star}},\quad 
\Delta_{\Psi_{\star}}, \quad
 \Delta_{\Psi},\quad \Delta_{\psi},\quad \Delta_{\Upsilon}, 
\end{align*}
such that when $\varepsilon\leq\varepsilon_{\star}$, the following hold
\begin{align*}
\sup_{u,v}||\mathcal{D}^4\mu||_{L^2(\mathcal{S}_{u,v})}\leq&
C(\Delta_{e_{\star}},\Delta_{\Gamma_{\star}},\Delta_{\psi_{\star}},\Delta_{\Upsilon_{\star}},\Delta_{\Psi_{\star}},\Delta_{\Psi}),\\
\sup_{u}||\mathcal{D}^4\lambda||_{L^2(\mathcal{N}_u)},\sup_{v}||\mathcal{D}^4\lambda||_{L^2(\mathcal{N}'_v)}\leq& 
C(\Delta_{e_{\star}},\Delta_{\Gamma_{\star}},\Delta_{\psi_{\star}},\Delta_{\Upsilon_{\star}},\Delta_{\Psi_{\star}},\Delta_{\Psi}).
\end{align*}
\end{proposition}

\begin{proof}
\begin{align*}
\mthorn'\mu&=-\mu^2-\lambda\bar\lambda
-2 \, \mathrm{i} \left( \bar{\zeta}_5 \phi_1 - \zeta_5 \bar{\phi}_1-\bar{\eta}_5 \phi_1 + \eta_5 \bar{\phi}_1 \right)\\
\meth\lambda-\meth'\mu&=\pi\mu-\bar\pi\lambda
-\TiPsi_3,
\end{align*}
Start with $\mthorn'\mu$, make use of the norm of $\zeta_5$ and $\phi_1$ on the ingoing lightcone we have 
\begin{align*}
||\mathcal{D}^i\mu||_{L^2(\mathcal{S}_{u,v})}\leq&C(\Delta_{\Gamma_{\star}})
+C(\Delta_{e_{\star}})\int_0^u||\mathcal{D}^i\mu||_{L^2(\mathcal{S}_{u',v})}
+C(\Delta_{e_{\star}})\int_0^u||\mathcal{D}^i\lambda||_{L^2(\mathcal{S}_{u',v})} \\
&+C(\Delta_{e_{\star}},\Delta_{\Gamma_{\star}},\Delta_{\psi_{\star}},\Delta_{\Upsilon_{\star}},
\Delta_{\Psi_{\star}},\Delta_{\psi},\Delta_{\Upsilon}, \Delta_{\Psi},\mathcal{O}_{4,2})
(\varepsilon^{1/2}+\varepsilon) \\
\leq&C(\Delta_{e_{\star}},\Delta_{\Gamma_{\star}})
+C(\Delta_{e_{\star}})\int_0^u||\mathcal{D}^i\lambda||_{L^2(\mathcal{S}_{u',v})}
\end{align*}
Then from the Codizzi eq we have
\begin{align*}
||\mathcal{D}^4\lambda||_{L^2(\mathcal{S})}\leq&
C(\sum_{i=0}^{2}||\mathcal{D}^iK||_{L^2(\mathcal{S})},\Delta_{e_{\star}})
\sum_{j=0}^{3}\left(||\mathcal{D}^j\mathscr{D}_{\lambda}||_{L^2(\mathcal{S})}+
||\mathcal{D}^j\lambda||_{L^2(\mathcal{S})}\right) \\
\leq&C(\Delta_{e_{\star}},\Delta_{\Gamma_{\star}},\Delta_{\psi_{\star}},\Delta_{\Upsilon_{\star}},\Delta_{\Psi_{\star}})
(||\mathcal{D}^4\mu||_{L^2(\mathcal{S}_{u,v})}+||\mathcal{D}^3\TiPsi_3||_{L^2(\mathcal{S}_{u,v})}+1).
\end{align*}
Combine we have
\begin{align*}
||\mathcal{D}^4\mu||_{L^2(\mathcal{S}_{u,v})}
\leq C(\Delta_{e_{\star}},\Delta_{\Gamma_{\star}},\Delta_{\psi_{\star}},\Delta_{\Upsilon_{\star}},\Delta_{\Psi_{\star}},\Delta_{\Psi})
\end{align*}
\begin{align*}
||\mathcal{D}^4\lambda||_{L^2(\mathcal{N}_u)},||\mathcal{D}^4\lambda||_{L^2(\mathcal{N}'_v)}\leq 
C(\Delta_{e_{\star}},\Delta_{\Gamma_{\star}},\Delta_{\psi_{\star}},\Delta_{\Upsilon_{\star}},\Delta_{\Psi_{\star}},\Delta_{\Psi}).
\end{align*}

\end{proof}

\begin{proposition}
\label{4derivativerhosigma}
Assume the boundedness of the following
\begin{align*}
 \Delta_{\psi},\quad \Delta_{\Upsilon},
\quad  \Delta_{\Psi},
\end{align*}
 then there exists a sufficiently small $\varepsilon_{\star}$ 
depending on 
\begin{align*}
&\Delta_{e_{\star}},\quad \Delta_{\Gamma_{\star}},\quad \Delta_{\psi_{\star}},
\quad \Delta_{\Upsilon_{\star}},\quad 
\Delta_{\Psi_{\star}}, \quad
 \Delta_{\Psi},\quad \Delta_{\psi},\quad \Delta_{\Upsilon}, 
\end{align*}
such that when $\varepsilon\leq\varepsilon_{\star}$, the following hold
\begin{align*}
\sup_{u,v}||\mathcal{D}^4\rho||_{L^{2}(\mathcal{S}_{u,v})}\leq&
C(\Delta_{e_{\star}},\Delta_{\Gamma_{\star}},\Delta_{\psi_{\star}},\Delta_{\Upsilon_{\star}},\Delta_{\Psi_{\star}},\Delta_{\Psi},\Delta_{\psi},\Delta_{\Upsilon}), \\
\sup_{u}||\mathcal{D}^4\sigma||_{L^2(\mathcal{N}_u)},\sup_{v}||\mathcal{D}^4\sigma||_{L^2(\mathcal{N}'_v)}\leq&
C(\Delta_{e_{\star}},\Delta_{\Gamma_{\star}},\Delta_{\psi_{\star}},\Delta_{\Upsilon_{\star}},\Delta_{\Psi_{\star}},\Delta_{\Psi},\Delta_{\psi},\Delta_{\Upsilon}).
\end{align*}
\end{proposition}

\begin{proof}
\begin{align*}
\mthorn\rho&=\rho^2+\sigma\bar\sigma+\omega\rho+
2 \, \mathrm{i} \left( \bar{\zeta}_0 \phi_0 - \zeta_0 \bar{\phi}_0-\bar{\eta}_0 \chi_0 + \eta_0 \bar{\chi}_0 \right), \\
\meth\rho-\meth'\sigma&=\bar\pi\rho-\pi\sigma
-\TiPsi_1,
\end{align*}
We have
\begin{align*}
||\mathcal{D}^4\rho||_{L^{2}(\mathcal{S}_{u,v})}\leq&
C(\Delta_{e_{\star}},\Delta_{\Gamma_{\star}})
+C(\Delta_{e_{\star}},\Delta_{\Gamma_{\star}})\left(\int_0^v||\mathcal{D}^4\rho||_{L^{2}(\mathcal{S}_{u,v'})}
+\int_0^v||\mathcal{D}^4\sigma||_{L^{2}(\mathcal{S}_{u,v'})}\right)\\
&+C(\Delta_{e_{\star}},\Delta_{\Gamma_{\star}},\Delta_{\psi_{\star}},\Delta_{\Upsilon_{\star}},\Delta_{\Psi_{\star}},\Delta_{\Psi},\Delta_{\psi},\Delta_{\Upsilon})\\
\leq&C(\Delta_{e_{\star}},\Delta_{\Gamma_{\star}},\Delta_{\psi_{\star}},\Delta_{\Upsilon_{\star}},\Delta_{\Psi_{\star}},\Delta_{\Psi},\Delta_{\psi},\Delta_{\Upsilon})\\
&+C(\Delta_{e_{\star}},\Delta_{\Gamma_{\star}})\left(\int_0^v||\mathcal{D}^4\sigma||_{L^{2}(\mathcal{S}_{u,v'})}\right).
\end{align*}
and
\begin{align*}
||\mathcal{D}^4\sigma||_{L^2(\mathcal{S})}\leq&
C(\sum_{i=0}^{2}||\mathcal{D}^iK||_{L^2(\mathcal{S})},\Delta_{e_{\star}})
\sum_{j=0}^{3}\left(||\mathcal{D}^j\mathscr{D}_{\sigma}||_{L^2(\mathcal{S})}+
||\mathcal{D}^j\sigma||_{L^2(\mathcal{S})}\right) \\
\leq&C(\Delta_{e_{\star}},\Delta_{\Gamma_{\star}},\Delta_{\psi_{\star}},\Delta_{\Upsilon_{\star}},\Delta_{\Psi_{\star}})
(||\mathcal{D}^4\rho||_{L^2(\mathcal{S}_{u,v})}+||\mathcal{D}^3\TiPsi_1||_{L^2(\mathcal{S}_{u,v})}+1).
\end{align*}
Combine we obtain the results
\begin{align*}
||\mathcal{D}^4\rho||_{L^{2}(\mathcal{S}_{u,v})}\leq&
C(\Delta_{e_{\star}},\Delta_{\Gamma_{\star}},\Delta_{\psi_{\star}},\Delta_{\Upsilon_{\star}},\Delta_{\Psi_{\star}},\Delta_{\Psi},\Delta_{\psi},\Delta_{\Upsilon})\\
&+C(\Delta_{e_{\star}},\Delta_{\Gamma_{\star}},\Delta_{\psi_{\star}},\Delta_{\Upsilon_{\star}},\Delta_{\Psi_{\star}})\left(\int_0^v||\mathcal{D}^4\rho,\mathcal{D}^3\TiPsi||_{L^{2}(\mathcal{S}_{u,v'})}\right) \\
\leq&C(\Delta_{e_{\star}},\Delta_{\Gamma_{\star}},\Delta_{\psi_{\star}},\Delta_{\Upsilon_{\star}},\Delta_{\Psi_{\star}},\Delta_{\Psi},\Delta_{\psi},\Delta_{\Upsilon}).
\end{align*}
and 
\begin{align*}
||\mathcal{D}^4\sigma||_{L^2(\mathcal{N}_u)},\sup_{v}||\mathcal{D}^4\sigma||_{L^2(\mathcal{N}'_v)}\leq&
C(\Delta_{e_{\star}},\Delta_{\Gamma_{\star}},\Delta_{\psi_{\star}},\Delta_{\Upsilon_{\star}},\Delta_{\Psi_{\star}},\Delta_{\Psi},\Delta_{\psi},\Delta_{\Upsilon}).
\end{align*}

\end{proof}

\begin{proposition}
\label{4derivativetau}
Assume the boundedness of the following
\begin{align*}
 \Delta_{\psi},\quad \Delta_{\Upsilon},
\quad  \Delta_{\Psi},
\end{align*}
 then there exists a sufficiently small $\varepsilon_{\star}$ 
depending on 
\begin{align*}
&\Delta_{e_{\star}},\quad \Delta_{\Gamma_{\star}},\quad \Delta_{\psi_{\star}},
\quad \Delta_{\Upsilon_{\star}},\quad 
\Delta_{\Psi_{\star}}, \quad
 \Delta_{\Psi},\quad \Delta_{\psi},\quad \Delta_{\Upsilon}, 
\end{align*}
such that when $\varepsilon\leq\varepsilon_{\star}$, the following hold
\begin{align*}
\sup_u||\mathcal{D}^4\tau||_{L^2(\mathcal{N}_{u})},\sup_v||\mathcal{D}^4\tau||_{L^2(\mathcal{N}'_{v})}\leq
C(\Delta_{e_{\star}},\Delta_{\Gamma_{\star}},\Delta_{\psi_{\star}},\Delta_{\Upsilon_{\star}},
\Delta_{\Psi_{\star}},\Delta_{\psi},\Delta_{\Upsilon},\Delta_{\Psi}).
\end{align*}
\end{proposition}

\begin{proof}
We define the following auxiliary field 
\begin{align*}
\Titau\equiv\meth'\tau-\TiPsi_2.
\end{align*} 
\begin{align*}
\mthorn\Titau=2\mathrm{i}(\bar\chi_0\meth\eta_1-\bar\phi_0\meth\zeta_1-\chi_0\meth'\bar\eta_1
+\phi_0\meth'\bar\zeta_1)+m\Upsilon_j\psi_k+m\Gamma\psi_j^2
+ \Upsilon_i\psi_j^3+\Upsilon_j^2+
\Upsilon_j\psi_k\Gamma+V
\end{align*}
Then we have
\begin{align*}
||\mathcal{D}^3\Titau||_{L^{2}(\mathcal{S}_{u,v})}
\leq&C(\Delta_{e_{\star}},\Delta_{\Gamma_{\star}},\Delta_{\psi_{\star}},\Delta_{\Upsilon_{\star}},
\Delta_{\Psi_{\star}},\Delta_{\psi},\Delta_{\Upsilon},\Delta_{\Psi}) \\
&+C(\Delta_{e_{\star}},\Delta_{\Gamma_{\star}})\left(\int_0^v||\mathcal{D}^4\tau||_{L^{2}(\mathcal{S}_{u,v'})}+
\int_0^v||\mathcal{D}^3\Titau||_{L^{2}(\mathcal{S}_{u,v'})}\right) \\
\leq&C(\Delta_{e_{\star}},\Delta_{\Gamma_{\star}},\Delta_{\psi_{\star}},\Delta_{\Upsilon_{\star}},
\Delta_{\Psi_{\star}},\Delta_{\psi},\Delta_{\Upsilon},\Delta_{\Psi})\\
&+C(\Delta_{e_{\star}},\Delta_{\Gamma_{\star}})\int_0^v||\mathcal{D}^4\tau||_{L^{2}(\mathcal{S}_{u,v'})}.
\end{align*}

Then we make use of the definition of $\Titau$ and obtain
\begin{align*}
||\mathcal{D}^3\mathscr{D}_{\tau}||_{L^2(\mathcal{S}_{u,v})}\leq&||\mathcal{D}^3\TiPsi_2||_{L^2(\mathcal{S}_{u,v})}
+C(\Delta_{e_{\star}},\Delta_{\Gamma_{\star}})\int_0^v||\mathcal{D}^4\tau||_{L^{2}(\mathcal{S}_{u,v'})}\\
&+C(\Delta_{e_{\star}},\Delta_{\Gamma_{\star}},\Delta_{\psi_{\star}},\Delta_{\Upsilon_{\star}},
\Delta_{\Psi_{\star}},\Delta_{\psi},\Delta_{\Upsilon},\Delta_{\Psi}).
\end{align*}
Now we can make use of Prop. \ref{EllipticTweightnonzero} and obtain
\begin{align*}
||\mathcal{D}^4\tau||_{L^2(\mathcal{S}_{u,v})}\leq&
C(\sum_{i=0}^{2}||\mathcal{D}^iK||_{L^2(\mathcal{S})},\Delta_{e_{\star}})
\sum_{j=0}^{3}\left(||\mathcal{D}^j\mathscr{D}_{\tau}||_{L^2(\mathcal{S})}+
||\mathcal{D}^j\tau||_{L^2(\mathcal{S})}\right) \\
\leq&C(\Delta_{e_{\star}},\Delta_{\Gamma_{\star}},\Delta_{\psi_{\star}},\Delta_{\Upsilon_{\star}},\Delta_{\Psi_{\star}})
||\mathcal{D}^3\TiPsi_2||_{L^2(\mathcal{S}_{u,v})} \\
&+C(\Delta_{e_{\star}},\Delta_{\Gamma_{\star}},\Delta_{\psi_{\star}},\Delta_{\Upsilon_{\star}},\Delta_{\Psi_{\star}})
\int_0^v||\mathcal{D}^4\tau||_{L^{2}(\mathcal{S}_{u,v'})} \\
&+C(\Delta_{e_{\star}},\Delta_{\Gamma_{\star}},\Delta_{\psi_{\star}},\Delta_{\Upsilon_{\star}},
\Delta_{\Psi_{\star}},\Delta_{\psi},\Delta_{\Upsilon},\Delta_{\Psi}) \\
\leq&C(\Delta_{e_{\star}},\Delta_{\Gamma_{\star}},\Delta_{\psi_{\star}},\Delta_{\Upsilon_{\star}},\Delta_{\Psi_{\star}})
||\mathcal{D}^3\TiPsi_2||_{L^2(\mathcal{S}_{u,v})} \\
&+C(\Delta_{e_{\star}},\Delta_{\Gamma_{\star}},\Delta_{\psi_{\star}},\Delta_{\Upsilon_{\star}},
\Delta_{\Psi_{\star}},\Delta_{\psi},\Delta_{\Upsilon},\Delta_{\Psi})
\end{align*}

Integral along the light cone we obtain
\begin{align*}
||\mathcal{D}^4\tau||_{L^2(\mathcal{N}_{u})},||\mathcal{D}^4\tau||_{L^2(\mathcal{N}'_{v})}\leq
C(\Delta_{e_{\star}},\Delta_{\Gamma_{\star}},\Delta_{\psi_{\star}},\Delta_{\Upsilon_{\star}},
\Delta_{\Psi_{\star}},\Delta_{\psi},\Delta_{\Upsilon},\Delta_{\Psi}).
\end{align*}

\end{proof}

\subsection{Energy estimates}
\label{EnergyEst}
In this section, we make energy estimate for $\psi$, $\Upsilon$ and $\Psi$. 
We begin with the energy equality for the Hodge system:
\begin{lemma}
\label{corollaryIntegration}
For the pair $(f_1,f_2)$ satisfying system
\begin{align*}
\mthorn'f_1-\meth f_2=&P_0; \\
\mthorn f_2-\meth' f_1=&Q_0,
\end{align*}
one has the following energy equality
\begin{align}
&\int_{\mathcal{N}_u(0,v)}|f_1|^2+\int_{\mathcal{N}'_v(0,u)}Q^{-1}|f_2|^2
=\int_{\mathcal{N}_0(0,v)}|f_1|^2+\int_{\mathcal{N}'_0(0,u)}Q^{-1}|f_2|^2 \nonumber \\
&+\int_{\mathcal{D}_{u,v}}\left(2\mu|f_1|^2-(\omega+2\rho)|f_2|^2\right)
+\int_{\mathcal{D}_{u,v}}\left(\langle f_1,P_0\rangle+\langle f_2,Q_0\rangle
+\langle (\bar\tau-\pi)f_1, f_2\rangle\right), \label{EnergyIdentity}
\end{align}
where $\langle x,y \rangle\equiv\bar{x}y+x\bar{y}$. 
\end{lemma}

\begin{proposition}
\label{EnergyEstpsi}
Assume the boundedness of $\Delta_{\Psi}$ and $\Delta_{\Upsilon}$, 
then there exists a sufficiently small $\varepsilon_{\star}$ 
depending on 
\begin{align*}
\Delta_{e_{\star}},\quad \Delta_{\Gamma_{\star}},\quad \Delta_{\psi_{\star}},
\quad \Delta_{\Upsilon_{\star}},\quad 
\Delta_{\Psi_{\star}}, \quad \Delta_{\Upsilon},\quad \Delta_{\Psi},
\end{align*}
such that when $\varepsilon\leq\varepsilon_{\star}$, the following holds
\begin{align*}
\Delta_{\psi}\leq C(\Delta_{e_{\star}},\Delta_{\Gamma_{\star}},\Delta_{\psi_{\star}}).
\end{align*}
\end{proposition}

\begin{proof}
We start by pair $(\phi_0,\phi_1)$ by using
\begin{align*}
\mthorn \phi_1- \meth' \phi_0 =-m\bar\chi_0+ \frac{\phi_0 \pi}{2} + \phi_1 \rho - \frac{\phi_1 \omega}{2} , \quad
\mthorn' \phi_0 -\meth \phi_1= m\bar\chi_1-\phi_0 \mu + \frac{\phi_1 \bar{\pi}}{2} - \phi_1 \tau 
\end{align*}
and have
\begin{align*}
&\int_{\mathcal{N}_u(0,v)}|\meth^{i}\phi_0|^2+\int_{\mathcal{N}'_v(0,u)}Q^{-1}|\meth^{i}\phi_1|^2
=\int_{\mathcal{N}_0(0,v)}|\meth^{i}\phi_0|^2+\int_{\mathcal{N}'_0(0,u)}Q^{-1}|\meth^{i}\phi_1|^2  \\
&+\int_{\mathcal{D}_{u,v}}\left(2\mu|\meth^{i}\phi_0|^2-(\omega+2\rho)|\meth^{i}\phi_1|^2\right)
+\int_{\mathcal{D}_{u,v}}\left(\langle \meth^{i}\phi_0,P_i\rangle+\langle \meth^{i}\phi_1,Q_i\rangle
+\langle (\bar\tau-\pi)\meth^i\phi_1, \meth^{i}\phi_0\rangle\right) 
\end{align*}
where $i\leq4$ and
\begin{align*}
P_i=&\sum_{i_1+i_2+i_3+i_4=i}\meth^{i_1}\Gamma^{i_2}\meth^{i_3}\Gamma\meth^{i_4}\phi_k
+\sum_{i_1+i_2+i_3=i}m\meth^{i_1}\Gamma^{i_2}\meth^{i_3}\bar\chi_1, \\
Q_i=&m\meth^i\bar\chi_0+\sum_{i_1+i_2=i}\meth^{i_1}\Gamma\meth^{i_2}\phi_k
+\sum_{i_1+i_2=i-1}\meth^{i_1}K\meth^{i_2}\phi_0.
\end{align*}
Then we can estimate 
\begin{align*}
\sum_{i=0}^4\int_{\mathcal{D}_{u,v}}2\mu|\meth^{i}\phi_0|^2\leq 
C(\Delta_{e_{\star}},\Delta_{\Gamma_{\star}})\sum_{i=0}^4\int_0^u||\meth^{i}\phi_0||^2_{L^2(\mathcal{N}_u')}
\leq C(\Delta_{e_{\star}},\Delta_{\Gamma_{\star}},\Delta_{\phi})\varepsilon,
\end{align*}

\begin{align*}
\sum_{i=0}^4\int_{\mathcal{D}_{u,v}}(\omega+2\rho)|\meth^{i}\phi_1|^2\leq 
C(\Delta_{e_{\star}},\Delta_{\Gamma_{\star}})\sum_{i=0}^4\int_0^v||\meth^{i}\phi_1||^2_{L^2(\mathcal{N}'_{v'})},
\end{align*}

\begin{align*}
\sum_{i=0}^4\int_{\mathcal{D}_{u,v}}\langle (\bar\tau-\pi)\meth^i\phi_1, \meth^{i}\phi_0\rangle
\leq& C\sum_{i=0}^4\left(\int_{\mathcal{D}_{u,v}}|\meth^i\phi_0|^2 \right)^{1/2}\left(\int_{\mathcal{D}_{u,v}}|\meth^{i}\phi_1|^2 \right)^{1/2} \\
\leq& C(\Delta_{e_{\star}},\Delta_{\Gamma_{\star}},\Delta_{\psi_{\star}},\Delta_{\Upsilon_{\star}},
\Delta_{\Psi_{\star}},\Delta_{\psi},\Delta_{\Upsilon},\Delta_{\Psi})\varepsilon^{1/2}.
\end{align*}

Note here there is couple term $\bar\chi_1$ in $P_i$. 
But one can obtain an $\varepsilon^{\frac{1}{2}}$ from the integral of $\phi_0$ over $\mathcal{D}_{u,v}$ 
and then we have
\begin{align*}
\sum_{i=0}^4\int_{\mathcal{D}_{u,v}}\langle \meth^{i}\phi_0,P_i\rangle\leq&
\sum_{i=0}^4||\meth^{i}\phi_0||_{L^2(\mathcal{D}_{u,v})}||P_i||_{L^2(\mathcal{D}_{u,v})}
\leq C(\Delta_{e_{\star}},\Delta_{\phi})\varepsilon^{1/2}\sum_{i=0}^4||P_i||_{L^2(\mathcal{D}_{u,v})} \\
\leq&C(\Delta_{e_{\star}},\Delta_{\Gamma_{\star}},\Delta_{\psi_{\star}},\Delta_{\Upsilon_{\star}},
\Delta_{\Psi_{\star}},\Delta_{\psi},\Delta_{\Upsilon},\Delta_{\Psi})(\varepsilon+\varepsilon^{1/2}).
\end{align*}

Again there is couple term $\bar\chi_0$ in $Q_i$. One can make use of the norm on the 
outgoing cone and then integral along the ingoing short direction and hence obtain an $\varepsilon^{1/2}$. 
Then one has
\begin{align*}
\sum_{i=0}^4\int_{\mathcal{D}_{u,v}}\langle \meth^{i}\phi_1,Q_i\rangle\leq&
\sum_{i=0}^4||\meth^{i}\phi_1||_{L^2(\mathcal{D}_{u,v})}||Q_i||_{L^2(\mathcal{D}_{u,v})} \\
\leq& C(\Delta_{e_{\star}},\Delta_{\Gamma_{\star}})\sum_{i=0}^4\int_0^v||\meth^{i}\phi_1||^2_{L^2(\mathcal{N}'_{v'})}+
\sum_{i=0}^4||Q_i||^2_{L^2(\mathcal{D}_{u,v})} \\
\leq&C(\Delta_{e_{\star}},\Delta_{\Gamma_{\star}})\sum_{i=0}^4\int_0^v||\meth^{i}\phi_1||^2_{L^2(\mathcal{N}'_{v'})} \\
&+C(\Delta_{e_{\star}},\Delta_{\Gamma_{\star}},\Delta_{\psi_{\star}},\Delta_{\Upsilon_{\star}},
\Delta_{\Psi_{\star}},\Delta_{\psi},\Delta_{\Upsilon},\Delta_{\Psi})\varepsilon^{1/2}
\end{align*}
Collect the results above we have
\begin{align*}
\sum_{i=0}^4\left(\sup_{u}
  ||\mathcal{D}^i\phi_0||_{L^2(\mathcal{N}_{u})}
  +\sup_{v}
  ||\mathcal{D}^i\phi_1||_{L^2(\mathcal{N}'_{v})}\right)
  \leq C(\Delta_{e_{\star}},\Delta_{\Gamma_{\star}},\Delta_{\psi_{\star}}) .
\end{align*}
The analysis pf pair $(\chi_0,\chi_1)$ is the same.

\end{proof}

\begin{proposition}
\label{EnergyEstUpsilon}
Assume the boundedness of $\Delta_{\Psi}$, 
then there exists a sufficiently small $\varepsilon_{\star}$ 
depending on 
\begin{align*}
\Delta_{e_{\star}},\quad \Delta_{\Gamma_{\star}},\quad \Delta_{\psi_{\star}},
\quad \Delta_{\Upsilon_{\star}},\quad 
\Delta_{\Psi_{\star}}, \quad \Delta_{\Psi},
\end{align*}
such that when $\varepsilon\leq\varepsilon_{\star}$, the following holds
\begin{align*}
\Delta_{\Upsilon}\leq C(\Delta_{e_{\star}},\Delta_{\Gamma_{\star}},\Delta_{\psi_{\star}}, \Delta_{\Upsilon_{\star}}).
\end{align*}
\end{proposition}

\begin{proof}
We analyze the pair $(\zeta_0,\zeta_1)$, $(\zeta_1,\zeta_2)$, 
$(\zeta_3,\zeta_4)$, $(\zeta_4,\zeta_5)$, $(\eta_0,\eta_1)$, $(\eta_1,\eta_2)$, 
$(\eta_3,\eta_4)$ and $(\eta_4,\eta_5)$ by analysing equation systems 
(\eqref{thornprimezeta0},\eqref{thornzeta1}), (\eqref{thornprimezeta1},\eqref{thornzeta2}), 
(\eqref{thornprimezeta3},\eqref{thornzeta4}), (\eqref{thornprimezeta4},\eqref{thornzeta5}), 
(\eqref{thornprimeeta0},\eqref{thorneta1}), (\eqref{thornprimeeta1},\eqref{thorneta2}), 
(\eqref{thornprimeeta3},\eqref{thorneta4}) and (\eqref{thornprimeeta4},\eqref{thorneta5}) 
respectively.

Denote $\Upsilon_L\in\{\zeta_0,\zeta_1,\zeta_3,\zeta_4,\eta_0,\eta_1,\eta_3,\eta_4\}$ and 
$\Upsilon_R\in\{\zeta_1,\zeta_2,\zeta_4,\zeta_5,\eta_1,\eta_2,\eta_4,\eta_5\}$, they satisfy the following equations:
\begin{align*}
\mthorn'\Upsilon_L-\meth\Upsilon_R=&m\Upsilon+m^2\psi+m\psi^2+\Gamma\Upsilon+\Upsilon\psi^2,\\
\mthorn\Upsilon_R-\meth'\Upsilon_L=&m\Upsilon+m^2\psi+m\psi^2+\Gamma\Upsilon+\Upsilon\psi^2.
\end{align*}
Note that there is no curvature terms in these equations and hence we do not need the 
control of 4-derivative of curvature. 
Such good feature guarantees the closeness of the bootstrap arguments. 

We have
\begin{align*}
&\int_{\mathcal{N}_u(0,v)}|\meth^{i}\Upsilon_L|^2+\int_{\mathcal{N}'_v(0,u)}Q^{-1}|\meth^{i}\Upsilon_R|^2
=\int_{\mathcal{N}_0(0,v)}|\meth^{i}\Upsilon_L|^2+\int_{\mathcal{N}'_0(0,u)}Q^{-1}|\meth^{i}\Upsilon_R|^2  \\
&+\int_{\mathcal{D}_{u,v}}\left(2\mu|\meth^{i}\Upsilon_L|^2-(\omega+2\rho)|\meth^{i}\Upsilon_R|^2\right)
+\int_{\mathcal{D}_{u,v}}\left(\langle \meth^{i}\Upsilon_L,P_i\rangle+\langle \meth^{i}\Upsilon_R,Q_i\rangle
+\langle (\bar\tau-\pi)\meth^i\Upsilon_R, \meth^{i}\Upsilon_L\rangle\right) 
\end{align*}
where
\begin{align*}
P_i=&\sum_{i_1+i_2+i_3+i_4=i}\meth^{i_1}\Gamma^{i_2}\meth^{i_3}\Gamma\meth^{i_4}\Upsilon_j
 +\sum_{i_1+i_2+i_3=i}\meth^{i_1}\Gamma^{i_2}\meth^{i_3}\{m\Upsilon,m^2\psi\}\\
&+\sum_{i_1+i_2+i_3+i_4=i}m\meth^{i_1}\Gamma^{i_2}\meth^{i_3}\psi_j\meth^{i_4}\psi_k
+\sum_{i_1+i_2+i_3+i_4+i_5=i}\meth^{i_1}\Gamma^{i_2}\meth^{i_3}\phi_j
\meth^{i_4}\phi_k\meth^{i_5}\Upsilon_l, \\
Q_i=&\meth^i\{m\Upsilon,m^2\psi\}
+\sum_{i_1+i_2=i}\meth^{i_1}\Upsilon_j\meth^{i_2}\Gamma+
\sum_{i_1+i_2+i_3=i}\meth^{i_1}\phi_j\meth^{i_2}\phi_k\meth^{i_3}\Upsilon_l
+\sum_{i_1+i_2=i-1}\meth^{i_1}K\meth^{i_2}\varphi_L.
\end{align*}

For the pair $(\zeta_4,\zeta_5)$, there are terms $\psi^2(\eta_{2},\eta_5,\zeta_2)$ in $\mthorn\zeta_5$. 
Similarly, there are terms $\psi^2(\zeta_{2},\zeta_5,\eta_2)$ in $\mthorn\eta_5$, 
terms $\psi^2\eta_{2}$ in $\mthorn\zeta_2$, terms $\psi^2\zeta_{2}$ in $\mthorn\eta_2$.
Note that for $\zeta_{2,5}$ and $\eta_{2,5}$, we only have their norm on the ingoing cone. 
For such coupled trouble terms, one can separate by Cauchy inequality. 
Take $\psi^2\eta_5$ as an example and we have
\begin{align*}
\int_{\mathcal{D}_{u,v}}\langle \meth^{k}\zeta_5,\psi^2\meth^{k}\eta_5\rangle\ ,
\leq
C(\Delta_{e_{\star}},\Delta_{\psi_{\star}})
\left(\sum_{i=0}^4\int_0^v||\meth^i\zeta_5||^2_{L^2(\mathcal{N}'_{v'})} 
+\sum_{i=0}^4\int_0^v||\meth^i\eta_5||^2_{L^2(\mathcal{N}'_{v'})} \right).
\end{align*}

For the rest terms, the analysis are similar. 
Then follow the strategy shown in Prop. \ref{EnergyEstpsi}, 
one can then have the following control
\begin{align*}
&\sum_{i=0}^4\left(\int_{\mathcal{N}_u(0,v)}|\meth^{i}\Upsilon_L|^2
+\int_{\mathcal{N}'_v(0,u)}Q^{-1}|\meth^{i}\Upsilon_R|^2\right)
\leq C(\Delta_{e_{\star}},\Delta_{\Upsilon_{\star}}) \\
&+C(\Delta_{e_{\star}},\Delta_{\Gamma_{\star}},\Delta_{\psi_{\star}})
\left(\sum_{i=0}^4\int_0^v||\meth^i\zeta_{2,5}||^2_{L^2(\mathcal{N}'_{v'})} 
+\sum_{i=0}^4\int_0^v||\meth^i\eta_{2,5}||^2_{L^2(\mathcal{N}'_{v'})} \right)
+C\varepsilon^{1/2} \\
\leq &C(\Delta_{e_{\star}},\Delta_{\Gamma_{\star}},\Delta_{\psi_{\star}},\Delta_{\Psi_{\star}}),
\end{align*}
Hence we have
\begin{align*}
\sum_{i=0}^4\left(\sup_{u}
  ||\mathcal{D}^i\Upsilon_L||_{L^2(\mathcal{N}_{u})}
  +\sup_{v}
  ||\mathcal{D}^i\Upsilon_R||_{L^2(\mathcal{N}'_{v})}\right)
  \leq C(\Delta_{e_{\star}},\Delta_{\Gamma_{\star}},\Delta_{\psi_{\star}},\Delta_{\Upsilon_{\star}}) .
\end{align*}
Here the dependence of $\Delta_{\psi_{\star}}$ results from the term $\Upsilon\psi^2$.

\end{proof}

\begin{remark}
\label{equiv3142}
Moreover, make use of the constraint equations \eqref{zeta3zeta1},\eqref{zeta4zeta2}, \eqref{eta3eta1} and \eqref{eta4eta2} and the elliptic inequality, for the top derivative $k=4$, one has the following 
\begin{align*}
||\mathcal{D}^k\zeta_3||_{\mathcal{S}_{u,v}}\leq& 
C(\Delta_{e_{\star}},\Delta_{\Gamma_{\star}},\Delta_{\psi_{\star}},\Delta_{\Upsilon_{\star}},
\Delta_{\Psi_{\star}})||\mathcal{D}^k\zeta_1||_{L^2(\mathcal{S}_{u,v})} \\
&+C(\Delta_{e_{\star}},\Delta_{\Gamma_{\star}},\Delta_{\psi_{\star}},\Delta_{\Upsilon_{\star}},
\Delta_{\Psi_{\star}})||\mathcal{D}^{k-1}\TiPsi_{1,2}||_{L^2(\mathcal{S}_{u,v})}, \\
||\mathcal{D}^k\zeta_2||_{\mathcal{S}_{u,v}}\leq& 
C(\Delta_{e_{\star}},\Delta_{\Gamma_{\star}},\Delta_{\psi_{\star}},\Delta_{\Upsilon_{\star}},
\Delta_{\Psi_{\star}})||\mathcal{D}^k\zeta_4||_{L^2(\mathcal{S}_{u,v})} \\
&+C(\Delta_{e_{\star}},\Delta_{\Gamma_{\star}},\Delta_{\psi_{\star}},\Delta_{\Upsilon_{\star}},
\Delta_{\Psi_{\star}})||\mathcal{D}^{k-1}\TiPsi_{2,3}||_{L^2(\mathcal{S}_{u,v})}
\end{align*}
and
\begin{align*}
||\mathcal{D}^k\eta_3||_{\mathcal{S}_{u,v}}\leq& 
C(\Delta_{e_{\star}},\Delta_{\Gamma_{\star}},\Delta_{\psi_{\star}},\Delta_{\Upsilon_{\star}},
\Delta_{\Psi_{\star}})||\mathcal{D}^k\eta_1||_{L^2(\mathcal{S}_{u,v})} \\
&+C(\Delta_{e_{\star}},\Delta_{\Gamma_{\star}},\Delta_{\psi_{\star}},\Delta_{\Upsilon_{\star}},
\Delta_{\Psi_{\star}})||\mathcal{D}^{k-1}\TiPsi_{1,2}||_{L^2(\mathcal{S}_{u,v})}, \\
||\mathcal{D}^k\eta_2||_{\mathcal{S}_{u,v}}\leq& 
C(\Delta_{e_{\star}},\Delta_{\Gamma_{\star}},\Delta_{\psi_{\star}},\Delta_{\Upsilon_{\star}},
\Delta_{\Psi_{\star}})||\mathcal{D}^k\eta_4||_{L^2(\mathcal{S}_{u,v})} \\
&+C(\Delta_{e_{\star}},\Delta_{\Gamma_{\star}},\Delta_{\psi_{\star}},\Delta_{\Upsilon_{\star}},
\Delta_{\Psi_{\star}})||\mathcal{D}^{k-1}\TiPsi_{2,3}||_{L^2(\mathcal{S}_{u,v})}.
\end{align*}
Then integral along the lightcone one has 
\begin{align*}
||\mathcal{D}^k\zeta_3||_{L^2(\mathcal{N}'_{v})}\leq&
C(\Delta_{e_{\star}},\Delta_{\Gamma_{\star}},\Delta_{\psi_{\star}},\Delta_{\Upsilon_{\star}},
\Delta_{\Psi_{\star}})(||\mathcal{D}^k\zeta_1||_{L^2(\mathcal{N}'_{v})}+
||\mathcal{D}^{k-1}\TiPsi_{1,2}||_{L^2(\mathcal{N}'_{v})}) \\
\leq&C(\Delta_{e_{\star}},\Delta_{\Gamma_{\star}},\Delta_{\psi_{\star}},\Delta_{\Upsilon_{\star}},
\Delta_{\Psi_{\star}})||\mathcal{D}^{k-1}\TiPsi_{1,2}||_{L^2(\mathcal{N}'_{v})},
\end{align*}

\begin{align*}
||\mathcal{D}^k\zeta_2||_{L^2(\mathcal{N}_{u})}
\leq&C(\Delta_{e_{\star}},\Delta_{\Gamma_{\star}},\Delta_{\psi_{\star}},\Delta_{\Upsilon_{\star}},
\Delta_{\Psi_{\star}})||\mathcal{D}^{k-1}\TiPsi_{2,3}||_{L^2(\mathcal{N}'_{v})},
\end{align*}

\begin{align*}
||\mathcal{D}^k\eta_3||_{L^2(\mathcal{N}'_{v})}\leq 
C(\Delta_{e_{\star}},\Delta_{\Gamma_{\star}},\Delta_{\psi_{\star}},\Delta_{\Upsilon_{\star}},
\Delta_{\Psi_{\star}})||\mathcal{D}^{k-1}\TiPsi_{1,2}||_{L^2(\mathcal{N}'_{v})}
\end{align*}
and
\begin{align*}
||\mathcal{D}^k\eta_2||_{L^2(\mathcal{N}_{u})}
\leq&C(\Delta_{e_{\star}},\Delta_{\Gamma_{\star}},\Delta_{\psi_{\star}},\Delta_{\Upsilon_{\star}},
\Delta_{\Psi_{\star}})||\mathcal{D}^{k-1}\TiPsi_{2,3}||_{L^2(\mathcal{N}_{u})}.
\end{align*}

\end{remark}

With these additional results one can then have

\begin{proposition}
\label{EnergyEstWeyl}
There exists a sufficiently small $\varepsilon_{\star}$ 
depending on 
\begin{align*}
\Delta_{e_{\star}},\quad \Delta_{\Gamma_{\star}},\quad \Delta_{\psi_{\star}},
\quad \Delta_{\Upsilon_{\star}}, \quad
\Delta_{\Psi_{\star}},
\end{align*}
such that when $\varepsilon\leq\varepsilon_{\star}$, the following hold
\begin{align*}
\Delta_{\Psi}\leq C(\Delta_{e_{\star}},\Delta_{\Gamma_{\star}},\Delta_{\psi_{\star}},\Delta_{\Psi_{\star}}).
\end{align*}
\end{proposition}
\begin{proof}
For the Weyl components $(\Psi_0,\TiPsi_1)$, $(\TiPsi_1,\TiPsi_2)$,$(\TiPsi_2,\TiPsi_3)$ and 
$(\TiPsi_3,\Psi_4)$ satisfy the Bianchi identities:
\begin{align*}
\mthorn'\Psi_L-\meth\Psi_R=&m\Upsilon\psi+m\psi^2\Gamma
+\Upsilon\psi\Gamma+\psi\meth\Upsilon+\Psi\psi^2+\Upsilon\psi^3+\Upsilon^2+\Gamma\Psi_j , \\
\mthorn\Psi_R-\meth'\Psi_L=&m\Upsilon\psi+m\psi^2\Gamma
+\Upsilon\psi\Gamma+\psi\meth\Upsilon+\Psi\psi^2+\Upsilon\psi^3+\Upsilon^2+\Gamma\Psi_j.
\end{align*}
where $\Psi_L\in\{\Psi_0,\TiPsi_{1,2,3}\}$, $\Psi_R\in\{\TiPsi_{1,2,3},\Psi_4\}$, $j_1=0,1$, $j_2=2,3,4$.

We have
\begin{align*}
&\int_{\mathcal{N}_u(0,v)}|\meth^{i}\Psi_L|^2+\int_{\mathcal{N}'_v(0,u)}Q^{-1}|\meth^{i}\Psi_R|^2
=\int_{\mathcal{N}_0(0,v)}|\meth^{i}\Psi_L|^2+\int_{\mathcal{N}'_0(0,u)}Q^{-1}|\meth^{i}\Psi_R|^2  \\
&+\int_{\mathcal{D}_{u,v}}\left(2\mu|\meth^{i}\Psi_L|^2-(\omega+2\rho)|\meth^{i}\Psi_R|^2\right)
+\int_{\mathcal{D}_{u,v}}\left(\langle \meth^{i}\Psi_L,P_i\rangle+\langle \meth^{i}\Psi_R,Q_i\rangle
+\langle (\bar\tau-\pi)\meth^i\Psi_R, \meth^{i}\Psi_L\rangle\right) 
\end{align*}
for $i\leq3$ and 
\begin{align*}
P_i=&\sum_{i_1+...+i_4=i}\meth^{i_1}\Gamma^{i_2}\meth^{i_3}\Gamma\meth^{i_4}\Psi_j
+\sum_{i_1+...+i_4=i}\meth^{i_1}\Gamma^{i_2}\meth^{i_3}\psi_j\meth^{i_4+1}\Upsilon_k\\
&+\sum_{i_1+...+i_5=i}\meth^{i_1}\Gamma^{i_2}\meth^{i_3}\psi_j\meth^{i_4}\psi_k\meth^{i_5}\Psi_l
+\sum_{i_1+...+i_4}\meth^{i_1}\Gamma^{i_2}\meth^{i_3}(\Upsilon^2,\Upsilon\psi^{j_1})\meth^{i_4}\Gamma^{j_2}\\
Q_i=&\sum_{i_1+i_2=i}\meth^{i_1}\Gamma\meth^{i_2}\Psi_j
+\sum_{i_1+i_2=i}\meth^{i_1}\psi_j\meth^{i_2+1}\Upsilon_k
+\sum_{i_1+i_2+i_3=i}\meth^{i_1}\psi\meth^{i_2}\psi_j\meth^{i_3}\Psi_k \\
&+\sum_{i_1+i_2+i_3=i}\meth^{i_1}\Gamma^{j_2}\meth^{i_2}\psi^{j_2}\meth^{i_3}\Upsilon^{j_3}.
\end{align*}
where $j_1=1,3$, $j_2=0,1$, $j_3=0,1,2$.
The key point is the analysis of $\meth\Upsilon$ one needs estimate the following:
\begin{align*}
I\equiv\int_{\mathcal{D}_{u,v}}\langle \meth^{k}\Psi_j,\psi\meth^{k+1}\Upsilon\rangle\ .
\end{align*}
For pair $(\Psi_0,\TiPsi_1)$, $(\TiPsi_1,\TiPsi_2)$ and $(\TiPsi_2,\TiPsi_3)$ 
one can first integral $\Psi_j$ along the outgoing lightcone and then integral along the ingoing short direction, 
then one has 
\begin{align*}
I_1\leq||\meth^{k}\Psi_j||_{L^2(\mathcal{D}_{u,v})}||\psi\meth^{k+1}\Upsilon||_{L^2(\mathcal{D}_{u,v})}
\leq C\Delta_{\Psi}\varepsilon^{\frac{1}{2}}.
\end{align*}
For the pair $(\TiPsi_3,\Psi_4)$, for the equation $\mthorn'\TiPsi_3$, 
terms $\psi\meth\Upsilon$ can still be treated in the above strategy 
and have control by $\varepsilon^{\frac{1}{2}}$. 
For $\Psi_4$, we only has its norm on the ingoing lightcone, then the above strategy failed. 
But one still should make sure that term $\psi\meth\Upsilon$ do not cause trouble.
Actually there are terms $\bar\chi_1\bigl(\meth'\eta_2\bigr)$, 
$\chi_1\meth'\bar\eta_4$, $\bar\phi_1\meth'\zeta_2$ and $\phi_1\meth'\bar\zeta_4$ 
in the equation $\mthorn\Psi_4$. 
For terms contain $\zeta_4$ and $\eta_4$, one can first integral them along the outgoing lightcone 
and then we have
\begin{align*}
I_2\leq C||\mathcal{D}^{k}\Psi_4||_{L^2(\mathcal{N}'_{v})}
||\mathcal{D}^{k+1}\{\zeta_4,\eta_4\}||_{L^2(\mathcal{N}_{u})}\varepsilon^{\frac{1}{2}}.
\end{align*}

For terms contain $\zeta_2$ and $\eta_2$, make use of the additional results in the Remark \ref{equiv3142}, 
i.e. one can control the norm of $\zeta_2$ and $\eta_2$ along the the outgoing cone via $\zeta_4$ and $\eta_4$. 
Hence one has
\begin{align*}
I_3\leq C||\mathcal{D}^{k}\Psi_4||_{L^2(\mathcal{N}'_{v})}
\left(||\mathcal{D}^{k+1}\{\zeta_4,\eta_4\},\mathcal{D}^k\TiPsi_{2,3}||_{L^2(\mathcal{N}_{u})}\right)\varepsilon^{\frac{1}{2}}.
\end{align*}

For term $\psi^2\Psi$, as we have already obtained the control of the next-to-leading derivative of $\psi$, 
such terms contribute 
\begin{align*}
C(\Delta_{e_{\star}},\Delta_{\Gamma_{\star}},\Delta_{\psi_{\star}})
\int_0^v||\mathcal{D}^{k}\Psi_4||^2_{L^2(\mathcal{N}'_{v'})}+C\varepsilon^{\frac{1}{2}}.
\end{align*}
The rest terms are also the next-to-leading terms, 
one can make use of the results in \ref{next-to-leading} to control.

Then one obtains
\begin{align*}
&\sum_{i=0}^3\left(\int_{\mathcal{N}_u(0,v)}|\meth^{i}\Psi_L|^2
+\int_{\mathcal{N}'_v(0,u)}Q^{-1}|\meth^{i}\Psi_R|^2\right)
\leq C(\Delta_{e_{\star}},\Delta_{\Psi_{\star}}) \\
&+C(\Delta_{e_{\star}},\Delta_{\Gamma_{\star}},\Delta_{\psi_{\star}})
\sum_{i=0}^3\int_0^v||\meth^i\Psi_4||^2_{L^2(\mathcal{N}'_{v'})} 
+C\varepsilon^{1/2} 
\leq C(\Delta_{e_{\star}},\Delta_{\Gamma_{\star}},\Delta_{\psi_{\star}},\Delta_{\Psi_{\star}}),
\end{align*}
this implies 
\begin{align*}
&\sum_{i=0}^3\left(\sup_{\Psi_L\in\{\Psi_0,\TiPsi_{1,2,3}\}}\sup_{u}
  ||\mathcal{D}^i\Psi_L||_{L^2(\mathcal{N}_{u})}
  +\sup_{\Psi_R\in\{\TiPsi_{1,2,3},\Psi_4\}}\sup_{v}
  ||\mathcal{D}^i\Psi_R||_{L^2(\mathcal{N}^{'}_{v})}\right) \\
  \leq& C(\Delta_{e_{\star}},\Delta_{\Gamma_{\star}},\Delta_{\psi_{\star}},\Delta_{\Psi_{\star}}) .
\end{align*}

\end{proof}

\appendix

\section{Equations}

\subsection{Definition of the derivative of Dirac field}
\label{DefinitionDerDirac}
\begin{subequations}
\begin{align}
\mthorn\,\phi_{0} &= \zeta_{0} + \frac{\phi_{0}\,\omega}{2}, \label{Defzeta0}\\
\meth'\,\phi_{0} &= \tfrac12\bigl(2\,\zeta_{1} + m\,\overline{\chi}_{0} 
+ \phi_{0}\,\pi - 2\,\phi_{1}\,\rho\bigr),\label{Defzeta1}\\
\meth'\,\phi_{1} &= \zeta_{2} + \phi_{0}\,\lambda - \frac{\phi_{1}\,\pi}{2},\label{Defzeta2}\\
\meth\,\phi_{0} &= \zeta_{3} + \frac{\phi_{0}\,\bar\pi}{2} - \phi_{1}\,\sigma,\label{Defzeta3}\\
\meth\,\phi_{1} &= \zeta_{4} - \frac{m\,\overline{\chi}_{1}}{2} 
+ \phi_{0}\,\mu - \frac{\phi_{1}\,\bar\pi}{2},\label{Defzeta4}\\
\mthorn'\,\phi_{1} &= \zeta_{5},\label{Defzeta5}\\
\mthorn\,\chi_{0} &= \eta_{0} + \frac{\chi_{0}\,\omega}{2},\label{Defeta0}\\
\meth'\,\chi_{0} &= \tfrac12\bigl(2\,\eta_{1} + m\,\overline{\phi}_{0} 
+ \chi_{0}\,\pi - 2\,\chi_{1}\,\rho\bigr),\label{Defeta1}\\
\meth'\,\chi_{1} &= \eta_{2} + \chi_{0}\,\lambda - \frac{\chi_{1}\,\pi}{2}.\label{Defeta2}\\
\meth\,\chi_{0} &= \eta_{3} + \frac{\chi_{0}\,\bar\pi}{2} - \chi_{1}\,\sigma,\label{Defeta3}\\
\meth\,\chi_{1} &= \eta_{4} - \frac{m\,\overline{\phi}_{1}}{2} 
+ \chi_{0}\,\mu - \frac{\chi_{1}\,\bar\pi}{2},\label{Defeta4}\\
\mthorn'\,\chi_{1} &= \eta_{5},\label{Defeta5}
\end{align}
\end{subequations}

\subsection{Equations for $\zeta_{ABA'}$}
\label{Tweightzeta}

\subsubsection{Transport equations of $\zeta_{ABA'}$ without curvature}
\label{TweightzetaNoCurv}
\begin{align}
\mthorn' \zeta_0 &= \frac{m \overline{\eta}_1}{2}
- \frac{3 m^2 \phi_0}{4}
+ \mathrm{i} \overline{\zeta}_4 \phi_0^2
- \mathrm{i} \overline{\zeta}_4 \phi_0 \overline{\phi}_0
- \mathrm{i} \overline{\zeta}_1 \phi_0 \phi_1
+ \mathrm{i} \zeta_3 \overline{\phi}_0 \phi_1
- \mathrm{i} \zeta_1 \phi_0 \overline{\phi}_1
+ \mathrm{i} \zeta_0 \phi_1 \overline{\phi}_1
- \mathrm{i} m \phi_0^2 \chi_1 \nonumber\\
&\quad
- \mathrm{i} \overline{\eta}_4 \phi_0 \chi_0
+ 2 \mathrm{i} \overline{\eta}_1 \phi_1 \chi_0
+ \mathrm{i} m \phi_0 \phi_1 \chi_0
+ \mathrm{i} \eta_4 \phi_0 \overline{\chi}_0
- \mathrm{i} \eta_3 \phi_1 \overline{\chi}_0
- \mathrm{i} m \phi_0 \overline{\phi}_1 \overline{\chi}_0
- \mathrm{i} \overline{\eta}_1 \phi_0 \chi_1 \nonumber\\
&\quad
+ \mathrm{i} \eta_1 \phi_0 \overline{\chi}_1
+ \mathrm{i} m \phi_0 \overline{\phi}_0 \overline{\chi}_1
- \mathrm{i} \eta_0 \phi_1 \overline{\chi}_1
- \zeta_0 \mu
- \frac{\zeta_1 \bar\pi}{2}
+ \zeta_4 \rho
+ \zeta_2 \sigma
- 2 \zeta_1 \tau
- \zeta_3 \overline{\tau}
+ \meth \zeta_1, \label{thornprimezeta0}
\end{align}

\begin{align}
\mthorn' \zeta_1 &= \frac{m \overline{\eta}_4}{2}
- \mathrm{i} \zeta_5 \phi_0 \overline{\phi}_0
- \frac{3 m^2 \phi_1}{4}
+ \mathrm{i} \overline{\zeta}_4 \phi_0 \phi_1
+ \mathrm{i} \zeta_4 \overline{\phi}_0 \phi_1
- \mathrm{i} \overline{\zeta}_1 \phi_1^2
- \mathrm{i} \zeta_2 \phi_0 \overline{\phi}_1
+ \mathrm{i} \zeta_1 \phi_1 \overline{\phi}_1
+ \mathrm{i} \overline{\eta}_4 \phi_1 \chi_0 \nonumber\\
&\quad
+ \mathrm{i} m \phi_1^2 \chi_0
+ \mathrm{i} \eta_5 \phi_0 \overline{\chi}_0
- \mathrm{i} \eta_4 \phi_1 \overline{\chi}_0
- \mathrm{i} m \phi_1 \overline{\phi}_1 \overline{\chi}_0
- 2 \mathrm{i} \overline{\eta}_4 \phi_0 \chi_1
+ \mathrm{i} \overline{\eta}_1 \phi_1 \chi_1
- \mathrm{i} m \phi_0 \phi_1 \chi_1 \nonumber\\
&\quad
+\mathrm{i} \eta_2 \phi_0 \overline{\chi}_1
- \mathrm{i} \eta_1 \phi_1 \overline{\chi}_1
+ \mathrm{i} m \overline{\phi}_0 \phi_1 \overline{\chi}_1
- 2 \zeta_1 \mu
+ \frac{\zeta_2 \bar\pi}{2}
+ \zeta_5 \rho
- \zeta_2 \tau
- \zeta_4 \overline{\tau}
+ \meth \zeta_2,  \label{thornprimezeta1}
\end{align}

\begin{align}
\mthorn' \zeta_3 &= \frac{m \overline{\eta}_2}{2}
+ \mathrm{i} \overline{\zeta}_5 \phi_0^2
- \mathrm{i} \overline{\zeta}_2 \phi_0 \phi_1
- 2 \mathrm{i} \zeta_4 \phi_0 \overline{\phi}_1
+ 2 \mathrm{i} \zeta_3 \phi_1 \overline{\phi}_1
- \mathrm{i} \overline{\eta}_5 \phi_0 \chi_0
+ 2 \mathrm{i} \overline{\eta}_2 \phi_1 \chi_0
- \mathrm{i} \overline{\eta}_2 \phi_0 \chi_1
+ 2 \mathrm{i} \eta_4 \phi_0 \overline{\chi}_1 \nonumber\\
&\quad
- 2 \mathrm{i} \eta_3 \phi_1 \overline{\chi}_1
- \zeta_1 \overline{\lambda}
- \zeta_3 \mu
+ \frac{\zeta_4 \bar\pi}{2}
+ \zeta_5 \sigma
- 2 \zeta_4 \tau
+ \meth \zeta_4, \label{thornprimezeta3}
\end{align}

\begin{align}
\mthorn' \zeta_4 &= \frac{m \overline{\eta}_5}{2}
+ \mathrm{i} \overline{\zeta}_5 \phi_0 \phi_1
- \mathrm{i} \overline{\zeta}_2 \phi_1^2
- 2 \mathrm{i} \zeta_5 \phi_0 \overline{\phi}_1
+ 2 \mathrm{i} \zeta_4 \phi_1 \overline{\phi}_1
+ \mathrm{i} \overline{\eta}_5 \phi_1 \chi_0
- 2 \mathrm{i} \overline{\eta}_5 \phi_0 \chi_1
+ \mathrm{i} \overline{\eta}_2 \phi_1 \chi_1
- 2 \mathrm{i} \eta_4 \phi_1 \overline{\chi}_1 \nonumber\\
&\quad
+2 \mathrm{i} \eta_5 \phi_0 \bar\chi_1
- \zeta_2 \overline{\lambda}
- 2 \zeta_4 \mu
+ \frac{3 \zeta_5 \bar\pi}{2}
- \zeta_5 \tau
+ \meth \zeta_5, \label{thornprimezeta4}
\end{align}

\begin{align}
\mthorn \zeta_1 &= - \frac{m \overline{\eta}_0}{2}
- \mathrm{i} \overline{\zeta}_3 \phi_0^2
+ 2 \mathrm{i} \zeta_1 \phi_0 \overline{\phi}_0
+ \mathrm{i} \overline{\zeta}_0 \phi_0 \phi_1
- 2 \mathrm{i} \zeta_0 \overline{\phi}_0 \phi_1
+ \mathrm{i} \overline{\eta}_3 \phi_0 \chi_0
- 2 \mathrm{i} \overline{\eta}_0 \phi_1 \chi_0
- 2 \mathrm{i} \eta_1 \phi_0 \overline{\chi}_0
+ 2 \mathrm{i} \eta_0 \phi_1 \overline{\chi}_0 \nonumber\\
&\quad
+ \mathrm{i} \overline{\eta}_0 \phi_0 \chi_1
- \frac{\zeta_0 \pi}{2}
+ 2 \zeta_1 \rho
+ \zeta_3 \bar\sigma
+ \frac{\zeta_1 \omega}{2}
+ \meth' \zeta_0, \label{thornzeta1}
\end{align}

\begin{align}
\mthorn \zeta_2 &= - \frac{m \overline{\eta}_3}{2}
+ 2 \mathrm{i} \zeta_2 \phi_0 \overline{\phi}_0
- \mathrm{i} \overline{\zeta}_3 \phi_0 \phi_1
- 2 \mathrm{i} \zeta_1 \overline{\phi}_0 \phi_1
+ \mathrm{i} \overline{\zeta}_0 \phi_1^2
- \mathrm{i} \overline{\eta}_3 \phi_1 \chi_0
- 2 \mathrm{i} \eta_2 \phi_0 \overline{\chi}_0
+ 2 \mathrm{i} \eta_1 \phi_1 \overline{\chi}_0 \nonumber\\
&\quad
+ 2 \mathrm{i} \bar\eta_3 \phi_0 \chi_1
- \mathrm{i} \overline{\eta}_0 \phi_1 \chi_1
- \zeta_0 \lambda
+ \frac{3 \zeta_1 \pi}{2}
+ \zeta_2 \rho
+ \zeta_4 \bar\sigma
- \frac{\zeta_2 \omega}{2}
+ \meth' \zeta_1, \label{thornzeta2}
\end{align}

\begin{align}
\mthorn \zeta_4 &= - \frac{m \overline{\eta}_1}{2}
- \frac{3 m^2 \phi_0}{4}
- \mathrm{i} \overline{\zeta}_4 \phi_0^2
+ \mathrm{i} \zeta_4 \phi_0 \overline{\phi}_0
+ \mathrm{i} \overline{\zeta}_1 \phi_0 \phi_1
- \mathrm{i} \zeta_3 \overline{\phi}_0 \phi_1
+ \mathrm{i} \zeta_1 \phi_0 \overline{\phi}_1
- \mathrm{i} \zeta_0 \phi_1 \overline{\phi}_1
+ \mathrm{i} \overline{\eta}_4 \phi_0 \chi_0 \nonumber\\
&\quad
- 2 \mathrm{i} \overline{\eta}_1 \phi_1 \chi_0
+ \mathrm{i} m \phi_0 \phi_1 \chi_0
- \mathrm{i} \eta_4 \phi_0 \overline{\chi}_0
+ \mathrm{i} \eta_3 \phi_1 \overline{\chi}_0
- \mathrm{i} m \phi_0 \overline{\phi}_1 \overline{\chi}_0
+ \mathrm{i} \overline{\eta}_1 \phi_0 \chi_1
- \mathrm{i} m \phi_0^2 \chi_1 \nonumber\\
&\quad
- \mathrm{i} \eta_1 \phi_0 \overline{\chi}_1
+ \mathrm{i} m \phi_0 \overline{\phi}_0 \overline{\chi}_1
+ \mathrm{i} \eta_0 \phi_1 \overline{\chi}_1
- \zeta_0 \mu
+ \frac{3 \zeta_3 \pi}{2}
+ \zeta_1 \bar\pi
+ 2 \zeta_4 \rho
- \frac{\zeta_4 \omega}{2}
+ \meth' \zeta_3, \label{thornzeta4}
\end{align}

\begin{align}
\mthorn \zeta_5 &= - \frac{m \overline{\eta}_4}{2}
+ \mathrm{i} \zeta_5 \phi_0 \overline{\phi}_0
- \frac{3 m^2 \phi_1}{4}
- \mathrm{i} \overline{\zeta}_4 \phi_0 \phi_1
- \mathrm{i} \zeta_4 \overline{\phi}_0 \phi_1
+ \mathrm{i} \overline{\zeta}_1 \phi_1^2
+ \mathrm{i} \zeta_2 \phi_0 \overline{\phi}_1
- \mathrm{i} \zeta_1 \phi_1 \overline{\phi}_1 \nonumber\\
&\quad
- \mathrm{i} \overline{\eta}_4 \phi_1 \chi_0
+ \mathrm{i} m \phi_1^2 \chi_0
- \mathrm{i} \eta_5 \phi_0 \overline{\chi}_0
+ \mathrm{i} \eta_4 \phi_1 \overline{\chi}_0
- \mathrm{i} m \phi_1 \overline{\phi}_1 \overline{\chi}_0
+ 2 \mathrm{i} \overline{\eta}_4 \phi_0 \chi_1
- \mathrm{i} \overline{\eta}_1 \phi_1 \chi_1
-\mathrm{i} \eta_2 \phi_0 \bar\chi_1 \nonumber\\
&\quad
+\mathrm{i} \eta_1 \phi_1 \bar\chi_1
- \mathrm{i} m \phi_0 \phi_1 \chi_1
+ \mathrm{i} m \overline{\phi}_0 \phi_1 \overline{\chi}_1
- \zeta_3 \overline{\lambda}
- \zeta_1 \mu
+ \frac{5 \zeta_4 \pi}{2}
+ \zeta_2 \bar\pi
+ \zeta_5 \rho
- \frac{3 \zeta_5 \omega}{2}
+ \meth' \zeta_4, \label{thornzeta5}
\end{align}

\subsubsection{Equations with curvature}
\label{TweightzetaWithCurv}
\begin{align}
\mthorn'\,\zeta_{2} &= \Psi_{4}\,\phi_{0}
  - \phi_{1}\Bigl(
      \TiPsi_{3}
      - \mathrm{i}\,\zeta_{5}\,\overline{\phi}_{0}
      + 2\,\mathrm{i}\,\overline{\zeta}_{4}\,\phi_{1}
      - \mathrm{i}\,\zeta_{2}\,\overline{\phi}_{1}
      + \mathrm{i}\,\eta_{5}\,\overline{\chi}_{0}
      - 2\,\mathrm{i}\,\overline{\eta}_{4}\,\chi_{1}
      + \mathrm{i}\,\eta_{2}\,\overline{\chi}_{1}
    \Bigr) \nonumber\\
&\quad - 2\,\zeta_{4}\,\lambda
  - \zeta_{2}\,\mu
  + \frac{3\,\zeta_{5}\,\pi}{2}
  - \zeta_{5}\,\bar\tau
  + \meth'\,\zeta_{5}, \label{thornprimezeta2}
\end{align}

\begin{align}
\mthorn \zeta_3 &= \Psi_0 \phi_1 
  - \phi_0\bigl(\TiPsi_1 
    + 2\mathrm{i}\,\bar{\zeta}_1 \phi_0 
    - \mathrm{i}\,\zeta_3 \bar{\phi}_0 
    - \mathrm{i}\,\zeta_0 \bar{\phi}_1 
    - 2\mathrm{i}\,\bar{\eta}_1 \chi_0 
    + \mathrm{i}\,\eta_3 \bar{\chi}_0 
    + \mathrm{i}\,\eta_0 \bar{\chi}_1\bigr) \nonumber\\
  &\quad
  - \frac{\zeta_0 \bar\pi}{2} 
  + \zeta_3 \rho 
  + 2 \zeta_1 \sigma 
  + \frac{\zeta_3 \omega}{2} 
  + \meth \zeta_0, \label{thornzeta3}
\end{align}

\begin{align}
\mthorn \zeta_5 &= - \TiPsi_3 \phi_0
+ \mathrm{i} \zeta_5 \phi_0 \overline{\phi}_0
- m^2 \phi_1
+ \TiPsi_2 \phi_1
- 2 \mathrm{i} \overline{\zeta}_4 \phi_0 \phi_1
+ \mathrm{i} \zeta_2 \phi_0 \overline{\phi}_1
+ \frac{2}{3} \mathrm{i} m \phi_1^2 \chi_0
- \mathrm{i} \eta_5 \phi_0 \overline{\chi}_0
 \nonumber\\
&\quad
- \frac{2}{3} \mathrm{i} m \phi_0 \phi_1 \chi_1
+ \mathrm{i} \bar\eta_4 \phi_0 \chi_1
- \frac{2}{3} \mathrm{i} m \overline{\phi}_1 \phi_1 \overline{\chi}_0
- \mathrm{i} \eta_2 \phi_0 \overline{\chi}_1
+ \frac{2}{3} \mathrm{i} m \overline{\phi}_0 \phi_1 \overline{\chi}_1
- 2 \zeta_1 \mu
+ 2 \zeta_4 \pi
\nonumber \\
& \quad
+ \frac{3 \zeta_2 \bar\pi}{2}
+ \zeta_5 \rho
- \frac{3 \zeta_5 \omega}{2}
+ \meth \zeta_2, \label{thornzeta5alt}
\end{align}

\begin{align}
\meth' \zeta_3 &= \frac{m \overline{\eta}_1}{2}
+ \frac{m^2 \phi_0}{4}
- \TiPsi_2 \phi_0
+ \mathrm{i} \overline{\zeta}_4 \phi_0^2
- \mathrm{i} \zeta_4 \phi_0 \overline{\phi}_0
+ \TiPsi_1 \phi_1
+ \mathrm{i} \zeta_1 \phi_0 \phi_1
- \mathrm{i} \zeta_1 \phi_0 \overline{\phi}_1 \nonumber\\
&\quad
- \mathrm{i} \overline{\eta}_4 \phi_0 \chi_0
+ \frac{1}{3} \mathrm{i} m \phi_0 \phi_1 \chi_0
+ \mathrm{i} \eta_4 \phi_0 \overline{\chi}_0
- \frac{1}{3} \mathrm{i} m \phi_0 \overline{\phi}_1 \overline{\chi}_0
- \mathrm{i} \overline{\eta}_1 \phi_0 \chi_1
+ \frac{1}{3} \mathrm{i} m \phi_0^2 \chi_1
 \nonumber\\
&\quad
+ \mathrm{i} \eta_1 \phi_0 \overline{\chi}_1
- \frac{1}{3} \mathrm{i} m \phi_0 \overline{\phi}_0 \overline{\chi}_1
+ \frac{\zeta_3 \pi}{2}
- \frac{\zeta_1 \bar\pi}{2}
- \zeta_4 \rho
+ \zeta_2 \sigma
+ \meth \zeta_1, \label{zeta3zeta1}
\end{align}

\begin{align}
\meth' \zeta_4 &= \frac{m \overline{\eta}_4}{2}
- \TiPsi_3 \phi_0
- \frac{m^2 \phi_1}{4}
+ \TiPsi_2 \phi_1
- \mathrm{i} \overline{\zeta}_4 \phi_0 \phi_1
+ \mathrm{i} \zeta_4 \overline{\phi}_0 \phi_1
- \mathrm{i} \overline{\zeta}_1 \phi_1^2
+ \mathrm{i} \zeta_1 \phi_1 \overline{\phi}_1 \nonumber\\
&\quad
+ \mathrm{i} \overline{\eta}_4 \phi_1 \chi_0
- \frac{1}{3} \mathrm{i} m \phi_1^2 \chi_0
- \mathrm{i} \eta_4 \phi_1 \overline{\chi}_0
+ \frac{1}{3} \mathrm{i} m \overline{\phi}_1 \phi_1 \overline{\chi}_0
+\mathrm{i} \bar\eta_1 \phi_1 \chi_1
+ \frac{1}{3} \mathrm{i} m \phi_0 \phi_1 \chi_1
- \mathrm{i} \eta_1 \phi_1 \overline{\chi}_1 \nonumber\\
&\quad
- \frac{1}{3} \mathrm{i} m \overline{\phi}_0 \phi_1 \overline{\chi}_1
+ \zeta_3 \lambda
- \zeta_1 \mu
- \frac{\zeta_4 \pi}{2}
+ \frac{\zeta_2 \bar\pi}{2}
+ \meth \zeta_2, \label{zeta4zeta2}
\end{align}

\subsection{Equations for the $\eta_{ABA'}$}
\label{Tweighteta}

\subsubsection{Transport equations of $\eta_{ABA'}$ without curvature}
\label{TweightetaNoCurv}
\begin{align}
\mthorn' \eta_0 &= \frac{m \overline{\zeta}_1}{2}
- \frac{3 m^2 \chi_0}{4}
+ \mathrm{i} \overline{\zeta}_4 \phi_0 \chi_0
- \mathrm{i} \zeta_4 \overline{\phi}_0 \chi_0
+ \mathrm{i} \overline{\zeta}_1 \phi_1 \chi_0
- \mathrm{i} \zeta_1 \overline{\phi}_1 \chi_0
- \mathrm{i} \overline{\eta}_4 \chi_0^2
+ \mathrm{i} m \phi_1 \chi_0^2 
+ \mathrm{i} \eta_4 \chi_0 \overline{\chi}_0 \nonumber\\
&\quad
- \mathrm{i} m \overline{\phi}_1 \chi_0 \overline{\chi}_0
- 2 \mathrm{i} \overline{\zeta}_1 \phi_0 \chi_1
+ \mathrm{i} \zeta_3 \overline{\phi}_0 \chi_1
+ \mathrm{i} \zeta_0 \overline{\phi}_1 \chi_1
+ \mathrm{i} \overline{\eta}_1 \chi_0 \chi_1
- \mathrm{i} m \phi_0 \chi_0 \chi_1
- \mathrm{i} \eta_3 \overline{\chi}_0 \chi_1 \nonumber\\
&\quad
+ \mathrm{i} \eta_1 \chi_0 \overline{\chi}_0
+ \mathrm{i} m \overline{\phi}_0 \chi_0 \overline{\chi}_1
- \mathrm{i} \eta_0 \chi_1 \overline{\chi}_1
- \eta_0 \mu
- \frac{\eta_1 \bar\pi}{2}
+ \eta_4 \rho
+ \eta_2 \sigma
- 2 \eta_1 \tau
- \eta_3 \overline{\tau}
+ \meth \eta_1, \label{thornprimeeta0}
\end{align}

\begin{align}
\mthorn' \eta_1 &= - \frac{m \overline{\zeta}_4}{2}
- \mathrm{i} \zeta_5 \overline{\phi}_0 \chi_0
+ 2 \mathrm{i} \overline{\zeta}_4 \phi_1 \chi_0
- \mathrm{i} \zeta_2 \overline{\phi}_1 \chi_0
+ \mathrm{i} \eta_5 \chi_0 \overline{\chi}_0
- \frac{3 m^2 \chi_1}{4}
- \mathrm{i} \overline{\zeta}_4 \phi_0 \chi_1
+ \mathrm{i} \zeta_4 \overline{\phi}_0 \chi_1 \nonumber\\
&\quad
- \mathrm{i} \overline{\zeta}_1 \phi_1 \chi_1
+ \mathrm{i} \zeta_1 \overline{\phi}_1 \chi_1
- \mathrm{i} \overline{\eta}_4 \chi_0 \chi_1
+ \mathrm{i} m \phi_1 \chi_0 \chi_1
+ \mathrm{i} \eta_2 \chi_0 \overline{\chi}_1
+ \mathrm{i} m \overline{\phi}_0 \chi_1 \overline{\chi}_1
-\mathrm{i} \eta_4 \chi_1 \overline{\chi}_0 \nonumber\\
&\quad
-\mathrm{i}m\bar\phi_1\bar\chi_0\chi_1
+ \mathrm{i} \overline{\eta}_1 \chi_1^2
- \mathrm{i} m \phi_0 \chi_1^2
-\mathrm{i} \eta_1\chi_1\bar\chi_1
- 2 \eta_1 \mu
+ \frac{\eta_2 \pi}{2}
+ \eta_5 \rho
- \eta_2 \tau
- \eta_4 \overline{\tau}
+ \meth \eta_2, \label{thornprimeeta1}
\end{align}

\begin{align}
\mthorn' \eta_3 &= \frac{m \overline{\zeta}_2}{2}
+ \mathrm{i} \overline{\zeta}_5 \phi_0 \chi_0
+ \mathrm{i} \overline{\zeta}_2 \phi_1 \chi_0
- 2 \mathrm{i} \zeta_4 \overline{\phi}_1 \chi_0
- \mathrm{i} \overline{\eta}_5 \chi_0^2
- 2 \mathrm{i} \overline{\zeta}_2 \phi_0 \chi_1
+ 2 \mathrm{i} \zeta_3 \overline{\phi}_1 \chi_1
+ \mathrm{i} \overline{\eta}_2 \chi_0 \chi_1 \nonumber\\
&\quad
+ 2 \mathrm{i} \eta_4 \chi_0 \overline{\chi}_1
- 2 \mathrm{i} \eta_3 \chi_1 \overline{\chi}_1
- \eta_1 \overline{\lambda}
- \eta_3 \mu
+ \frac{\eta_4 \bar\pi}{2}
+ \eta_5 \sigma
- 2 \eta_4 \tau
+ \meth \eta_4, \label{thornprimeeta3}
\end{align}

\begin{align}
\mthorn' \eta_4 &= \frac{m \overline{\zeta}_5}{2}
+ 2 \mathrm{i} \overline{\zeta}_5 \phi_1 \chi_0
- 2 \mathrm{i} \zeta_5 \overline{\phi}_1 \chi_0
- \mathrm{i} \overline{\zeta}_5 \phi_0 \chi_1
- \mathrm{i} \overline{\zeta}_2 \phi_1 \chi_1
+ 2 \mathrm{i} \zeta_4 \overline{\phi}_1 \chi_1
- \mathrm{i} \overline{\eta}_5 \chi_0 \chi_1
+ \mathrm{i} \overline{\eta}_2 \chi_1^2 \nonumber\\
&\quad
+ 2 \mathrm{i} \eta_5 \chi_0 \overline{\chi}_1
- 2 \mathrm{i} \eta_4 \chi_1 \overline{\chi}_1
- \eta_2 \overline{\lambda}
- 2 \eta_4 \mu
+ \frac{3 \eta_5 \bar\pi}{2}
- \eta_5 \tau
+ \meth \eta_5, \label{thornprimeeta4}
\end{align}

\begin{align}
\mthorn \eta_1 &= - \frac{m \overline{\zeta}_0}{2}
- \mathrm{i} \overline{\zeta}_3 \phi_0 \chi_0
+ 2 \mathrm{i} \zeta_1 \overline{\phi}_0 \chi_0
- \mathrm{i} \overline{\zeta}_0 \phi_1 \chi_0
+ \mathrm{i} \overline{\eta}_3 \chi_0^2
- 2 \mathrm{i} \eta_1 \chi_0 \overline{\chi}_0
+ 2 \mathrm{i} \overline{\zeta}_0 \phi_0 \chi_1
- \mathrm{i} \overline{\eta}_0 \chi_0 \chi_1 \nonumber\\
&\quad
-2 \mathrm{i} \bar\eta_0 \chi_0 \chi_1
+ 2 \mathrm{i} \eta_0 \overline{\chi}_0 \chi_1
- \frac{\eta_0 \pi}{2}
+ 2 \eta_1 \rho
+ \eta_3 \overline{\sigma}
+ \frac{\eta_1 \omega}{2}
+ \meth' \eta_0, \label{thorneta1}
\end{align}

\begin{align}
\mthorn \eta_2 &= - \frac{m \overline{\zeta}_3}{2}
+ 2 \mathrm{i} \zeta_2 \overline{\phi}_0 \chi_0
- 2 \mathrm{i} \overline{\zeta}_3 \phi_1 \chi_0
- 2 \mathrm{i} \eta_2 \chi_0 \overline{\chi}_0
+ \mathrm{i} \overline{\zeta}_3 \phi_0 \chi_1
- 2 \mathrm{i} \zeta_1 \overline{\phi}_0 \chi_1
+ 2\mathrm{i} \eta_1 \overline{\chi}_0 \chi_1
- \mathrm{i} \overline{\eta}_0 \chi_1^2 \nonumber\\
&\quad
+ \mathrm{i} \bar\zeta_0 \phi_0 \chi_1
+ \mathrm{i} \bar\eta_3 \chi_0 \chi_1
- \eta_0 \lambda
+ \frac{3 \eta_1 \pi}{2}
+ \eta_2 \rho
+ \eta_4 \overline{\sigma}
- \frac{\eta_2 \omega}{2}
+ \meth' \eta_1, \label{thorneta2}
\end{align}

\begin{align}
\mthorn \eta_4 &= - \frac{m \overline{\zeta}_1}{2}
- \frac{3 m^2 \chi_0}{4}
- \mathrm{i} \overline{\zeta}_4 \phi_0 \chi_0
+ \mathrm{i} \zeta_4 \overline{\phi}_0 \chi_0
- \mathrm{i} \overline{\zeta}_1 \phi_1 \chi_0
- \mathrm{i} \zeta_1 \overline{\phi}_1 \chi_0
+ \mathrm{i} \overline{\eta}_4 \chi_0^2
+ \mathrm{i} m \phi_1 \chi_0^2
- \mathrm{i} \eta_4 \chi_0 \overline{\chi}_0 \nonumber\\
&\quad
- \mathrm{i} m \overline{\phi}_1 \chi_0 \overline{\chi}_0
+ 2 \mathrm{i} \overline{\zeta}_1 \phi_0 \chi_1
- \mathrm{i} \zeta_3 \overline{\phi}_0 \chi_1
- \mathrm{i} \zeta_0 \overline{\phi}_1 \chi_1
- \mathrm{i} \overline{\eta}_1 \chi_0 \chi_1
- \mathrm{i} m \phi_0 \chi_0 \chi_1
+ \mathrm{i} \eta_3 \overline{\chi}_0 \chi_1
- \mathrm{i} \eta_1 \chi_0 \overline{\chi}_1 \nonumber\\
&\quad
+ \mathrm{i} m \overline{\phi}_0 \chi_0 \overline{\chi}_1
+\mathrm{i} \eta_0 \chi_1 \overline{\chi}_1
- \eta_0 \mu
+ \frac{\eta_3 \pi}{2}
+ \eta_1 \pi
+ 2 \eta_4 \rho
- \frac{\eta_4 \omega}{2}
+ \meth' \eta_3 \label{thorneta4}
\end{align}

\begin{align}
\mthorn \eta_5 &= - \frac{m \overline{\zeta}_4}{2}
+ \mathrm{i} \zeta_5 \overline{\phi}_0 \chi_0
- 2 \mathrm{i} \overline{\zeta}_4 \phi_1 \chi_0
+ \mathrm{i} \zeta_2 \overline{\phi}_1 \chi_0
- \mathrm{i} \eta_5 \chi_0 \overline{\chi}_0
- \frac{3 m^2 \chi_1}{4}
+ \mathrm{i} \overline{\zeta}_4 \phi_0 \chi_1
- \mathrm{i} \zeta_4 \overline{\phi}_0 \chi_1
+\mathrm{i} \bar\zeta_1 \phi_1 \chi_1 \nonumber\\
&\quad
-\mathrm{i} \zeta_1 \bar\phi_1 \chi_1
+ \mathrm{i} \overline{\eta}_4 \chi_0 \chi_1
+ \mathrm{i} m \phi_1 \chi_0 \chi_1
+ \mathrm{i} \eta_4 \overline{\chi}_0 \chi_1
- \mathrm{i} m \overline{\phi}_1 \chi_0 \chi_1
- \mathrm{i} \overline{\eta}_1 \chi_1^2
- \mathrm{i} m \phi_0 \chi_1^2
+ \mathrm{i} m \overline{\phi}_0 \chi_1 \overline{\chi}_1 \nonumber\\
&\quad
-\mathrm{i} \eta_2 \chi_0 \bar\chi_1
+\mathrm{i} \eta_1 \chi_1 \bar\chi_1
+ \mathrm{i} m \overline{\phi}_0 \chi_1 \overline{\chi}_1
- \eta_3 \lambda
- \eta_1 \mu
+ \frac{5 \eta_4 \pi}{2}
+ \eta_2 \pi
+ \eta_5 \rho
- \frac{3 \eta_5 \omega}{2}
+ \meth' \eta_4, \label{thorneta5}
\end{align}

\subsubsection{Equations of $\eta_{ABA'}$ with curvature}
\label{TweightetawithCurv}
\begin{align}
\mthorn'\,\eta_{2} &= \Psi_{4}\,\chi_{0}
  - \chi_{1}\bigl(
      \TiPsi_{3}
    - \mathrm{i}\,\zeta_{5}\,\overline{\phi}_{0}
    + 2\,\mathrm{i}\,\overline{\zeta}_{4}\,\phi_{1}
    - \mathrm{i}\,\zeta_{2}\,\overline{\phi}_{1}
    + \mathrm{i}\,\eta_{5}\,\overline{\chi}_{0}
    - 2\,\mathrm{i}\,\overline{\eta}_{4}\,\chi_{1}
    + \mathrm{i}\,\eta_{2}\,\overline{\chi}_{1}
  \bigr)\nonumber\\
&\quad - 2\,\eta_{4}\,\lambda
  - \eta_{2}\,\mu
  + \frac{3\,\eta_{5}\,\pi}{2}
  - \eta_{5}\,\overline{\tau}
  + \meth'\,\eta_{5},\label{thornprimeeta2}
 \end{align} 

\begin{align}
\mthorn \eta_5 &= -\TiPsi_3 \chi_0
+ \mathrm{i} \zeta_5 \overline{\phi}_0 \chi_0
-2 \mathrm{i} \overline{\zeta}_4 \phi_1 \chi_0
+ \mathrm{i} \zeta_2 \overline{\phi}_1 \chi_0
- \mathrm{i} \eta_5 \chi_0 \overline{\chi}_0
- m^2 \chi_1
+ \TiPsi_2 \chi_1
+ 2 \mathrm{i} \overline{\eta}_4 \chi_0 \chi_1 \nonumber\\
&\quad
+ \tfrac{2}{3} \mathrm{i} m \phi_1 \chi_0 \chi_1
- \tfrac{2}{3} \mathrm{i} m \bar\phi_1 \bar\chi_0 \chi_1
- \tfrac{2}{3} \mathrm{i} m \phi_0 \chi_1^2
- \mathrm{i} \eta_2 \chi_0 \overline{\chi}_1
+ \tfrac{2}{3} \mathrm{i} m \overline{\phi}_0 \chi_1 \overline{\chi}_1 \nonumber\\
&\quad- 2 \eta_1 \mu
+ 2 \eta_4 \pi
+ \frac{3 \eta_2 \bar\pi}{2}
+ \eta_5 \rho
- \frac{3 \eta_5 \omega}{2}
+ \meth' \eta_2, \label{thorneta5alt}
\end{align}

\begin{align}
\mthorn \eta_3 &= \Psi_0 \chi_1 
  - \chi_0\bigl(\TiPsi_1 
    + 2\mathrm{i}\,\bar{\zeta}_1 \phi_0 
    - \mathrm{i}\,\zeta_3 \bar{\phi}_0 
    - \mathrm{i}\,\zeta_0 \bar{\phi}_1 
    - 2\mathrm{i}\,\bar{\eta}_1 \chi_0 
    + \mathrm{i}\,\eta_3 \bar{\chi}_0 
    + \mathrm{i}\,\eta_0 \bar{\chi}_1\bigr) \nonumber\\
  &\quad
  - \frac{\eta_0 \bar\pi}{2} 
  + \eta_3 \rho 
  + 2 \eta_1 \sigma 
  + \frac{\eta_3 \omega}{2} 
  + \meth \eta_0, \label{thorneta3}
\end{align}

\begin{align}
\meth' \eta_3 &= \frac{m \overline{\zeta}_1}{2}
+ \frac{m^2 \chi_0}{4}
- \TiPsi_2 \chi_0
+ \mathrm{i} \overline{\zeta}_4 \phi_0 \chi_0
- \mathrm{i} \zeta_4 \overline{\phi}_0 \chi_0
+ \mathrm{i} \overline{\zeta}_1 \phi_1 \chi_0
- \mathrm{i} \zeta_1 \overline{\phi}_1 \chi_0
- \mathrm{i} \overline{\eta}_4 \chi_0^2 \nonumber\\
&\quad
+ \tfrac{1}{3} \mathrm{i} m \phi_1 \chi_0^2
+ \mathrm{i} \eta_4 \chi_0 \overline{\chi}_0
- \tfrac{1}{3} \mathrm{i} m \overline{\phi}_1 \chi_0 \overline{\chi}_0
+\TiPsi_1\chi_1
- \mathrm{i} \bar\eta_1 \chi_0 \chi_0
+ \tfrac{1}{3} \mathrm{i} m \phi_0 \chi_0 \chi_1
+ \mathrm{i} \eta_1 \chi_0 \overline{\chi}_1 \nonumber \\
&\quad
+ \tfrac{1}{3} \mathrm{i} m \overline{\phi}_0 \chi_0 \overline{\chi}_1
+ \frac{\eta_3 \pi}{2}
- \frac{\overline{\eta}_1 \pi}{2}
- \eta_4 \rho
+ \eta_2 \sigma
+ \meth \eta_1, \label{eta3eta1}
\end{align}

\begin{align}
\meth' \eta_4 &= \frac{m \overline{\zeta}_4}{2}
- \TiPsi_3 \chi_0
- \frac{m^2 \chi_1}{4}
+ \TiPsi_2 \chi_1
- \mathrm{i} \overline{\zeta}_4 \phi_0 \chi_1
+ \mathrm{i} \zeta_4 \overline{\phi}_0 \chi_1
- \mathrm{i} \overline{\zeta}_1 \phi_1 \chi_1
+ \mathrm{i} \zeta_1 \overline{\phi}_1 \chi_1
+ \mathrm{i} \overline{\eta}_4 \chi_0 \chi_1 \nonumber\\
&\quad
- \tfrac{1}{3} \mathrm{i} m \phi_1 \chi_0 \chi_1
- \mathrm{i} \eta_4 \overline{\chi}_0 \chi_1
+ \tfrac{1}{3} \mathrm{i} m \overline{\phi}_1 \overline{\chi}_0 \chi_1
+ \mathrm{i}\bar\eta_1\chi_1^2
+ \tfrac{1}{3} \mathrm{i} m \phi_0 \chi_1^2
- \mathrm{i} \eta_1 \chi_1 \overline{\chi}_1
- \tfrac{1}{3} \mathrm{i} m \overline{\phi}_0 \chi_1 \overline{\chi}_1 \nonumber\\
&\quad
+ \eta_3 \lambda
- \eta_1 \mu
- \frac{\eta_4 \pi}{2}
+ \frac{\eta_2 \bar\pi}{2}
+ \meth \eta_2. \label{eta4eta2}
\end{align}

\subsection{The structure equations}
\label{StructureEQ}
\begin{subequations}
\begin{align}
\mthorn\tau &= \TiPsi_1+4\mathrm{i}\bar\zeta_1\phi_0-2\mathrm{i}\zeta_3\bar\phi_0-2\mathrm{i}\zeta_0\bar\phi_1
-4\mathrm{i}\bar\eta_1\chi_0+2\mathrm{i}\eta_3\bar\chi_0+2\mathrm{i}\eta_0\bar\chi_1
+\bar\pi\rho+\pi\sigma+\rho\tau+\sigma\bar\tau,\label{thorntau}\\
\mthorn'\pi &= -\TiPsi_3+2\mathrm{i}\zeta_5\bar\phi_0-4\mathrm{i}\bar\zeta_4\phi_1
+2\mathrm{i}\zeta_2\bar\phi_1-2\mathrm{i}\eta_5\bar\chi_0+4\mathrm{i}\bar\eta_4\chi_1
-2\mathrm{i}\eta_2\bar\chi_1-\mu\pi-\lambda\bar\pi-\lambda\tau-\mu\bar\tau,\label{thornprimepi}\\
\mthorn'\omega &= -\TiPsi_2-\overline{\TiPsi}_2
-2\mathrm{i}\bar\zeta_4\phi_0+2\mathrm{i}\zeta_4\bar\phi_0-2\mathrm{i}\bar\zeta_1\phi_1
+2\mathrm{i}\zeta_1\bar\phi_1+2\mathrm{i}\bar\eta_4\chi_0
+\frac{2\mathrm{i}}{3}m\phi_1\chi_0
-2\mathrm{i}\eta_4\bar\chi_0 \nonumber\\
&\quad
-\frac{2\mathrm{i}}{3}m\bar\phi_1\bar\chi_0
+2\mathrm{i}\bar\eta_1\chi_1
-\frac{2\mathrm{i}}{3}m\phi_0\chi_1
-2\mathrm{i}\eta_1\bar\chi_1
+\frac{2\mathrm{i}}{3}m\bar\phi_0\bar\chi_1 
-2\pi\bar\pi-2\pi\tau-2\pi\bar\tau, \label{thornprimeomega} \\
\mthorn'\mu &= -2\mathrm{i}\bar\zeta_5\phi_1+2\mathrm{i}\zeta_5\bar\phi_1+2\mathrm{i}\bar\eta_5\chi_1-2\mathrm{i}\eta_5\bar\chi_1
-\lambda\bar\lambda-\mu^2,\label{thornprimemu}\\
\mthorn\mu &= \TiPsi_2
-\frac{4\mathrm{i}}{3}m\phi_1\chi_0
+\frac{4\mathrm{i}}{3}m\bar\phi_1\bar\chi_0
+\frac{4\mathrm{i}}{3}m\phi_0\chi_1
-\frac{4\mathrm{i}}{3}m\bar\phi_0\bar\chi_1
+\pi\bar\pi+\mu\rho+\lambda\sigma-\mu\omega+\meth\pi,\label{thornmu}\\
\mthorn'\rho &= -\TiPsi_2
+\frac{4\mathrm{i}}{3}m\phi_1\chi_0
-\frac{4\mathrm{i}}{3}m\bar\phi_1\bar\chi_0
-\frac{4\mathrm{i}}{3}m\phi_0\chi_1
+\frac{4\mathrm{i}}{3}m\bar\phi_0\bar\chi_1
-\mu\rho-\lambda\sigma-\tau\bar\tau+\meth'\tau,\label{thornprimerho}\\
\mthorn\rho &= 2\mathrm{i}\bar\zeta_0\phi_0-2\mathrm{i}\zeta_0\bar\phi_0
-2\mathrm{i}\bar\eta_0\chi_0+2\mathrm{i}\eta_0\bar\chi_0
+\rho^2+\sigma\bar\sigma+\rho\omega, \label{thornrho} \\
\mthorn'\sigma &= -2\mathrm{i}\bar\zeta_2\phi_0+2\mathrm{i}\zeta_3\bar\phi_1+2\mathrm{i}\bar\eta_2\chi_0-2\mathrm{i}\eta_3\bar\chi_1
-\bar\lambda\rho-\mu\sigma-\tau^{2}+\meth\tau,\label{thornprimesigma}\\
\mthorn\sigma &= \Psi_0+2\rho\sigma+\sigma\omega,\label{thornsigma}\\
\mthorn'\lambda &= -\Psi_4-2\lambda\mu,\label{thornprimelambda}\\
\mthorn\lambda &= -2\mathrm{i}\zeta_2\bar\phi_0+2\mathrm{i}\bar\zeta_3\phi_1+2\mathrm{i}\eta_2\bar\chi_0-2\mathrm{i}\bar\eta_3\chi_1
+\pi^{2}+\lambda\rho+\mu\bar\sigma-\lambda\omega+\meth'\pi,\label{thornlambda}\\
\mthorn\bar\pi &= \TiPsi_1+4\mathrm{i}\bar\zeta_1\phi_0-2\mathrm{i}\zeta_3\bar\phi_0-2\mathrm{i}\zeta_0\bar\phi_1
-4\mathrm{i}\bar\eta_1\chi_0+2\mathrm{i}\eta_3\bar\chi_0+2\mathrm{i}\eta_0\bar\chi_1
+2\bar\pi\rho+2\pi\sigma+\meth\omega,\label{thornpi}\\
\meth'\mu &= \TiPsi_3-\mu\pi+\lambda\bar\pi+\meth\lambda,\label{mulambda}\\
\meth'\sigma &= \TiPsi_1-\bar\pi\rho+\pi\sigma+\meth\rho.\label{rhosigma}
\end{align}
\end{subequations}

\subsection{Necessary NP structure equations}

\begin{subequations}
\begin{align}
\Delta \beta &= -\mathrm{i}\,\bar{\zeta}_5 \phi_0 
  - \mathrm{i}\,\bar{\zeta}_2 \phi_1 
  + 2\mathrm{i}\,\zeta_4 \bar{\phi}_1 
  + \mathrm{i}\,\bar{\eta}_5 \chi_0 
  + \mathrm{i}\,\bar{\eta}_2 \chi_1 
  - 2\mathrm{i}\,\eta_4 \bar{\chi}_1 
  - \alpha \bar{\lambda} 
  - \beta \mu 
  - \mu \tau, \label{Deltabeta}\\
D \beta &= \TiPsi_1 
  + 2\mathrm{i}\,\bar{\zeta}_1 \phi_0 
  - \mathrm{i}\,\zeta_3 \bar{\phi}_0 
  - \mathrm{i}\,\zeta_0 \bar{\phi}_1 
  - 2\mathrm{i}\,\bar{\eta}_1 \chi_0 
  + \mathrm{i}\,\eta_3 \bar{\chi}_0 
  + \mathrm{i}\,\eta_0 \bar{\chi}_1 \nonumber\\
 & - \bar{\alpha} \epsilon 
  - \beta \bar{\epsilon} 
  + \epsilon\bar\pi 
  + \beta \rho 
  + \alpha \sigma 
  + \pi \sigma 
  + \delta \varepsilon, \label{Dbeta}\\
\Delta \alpha &= -\TiPsi_3 
  + \mathrm{i}\,\zeta_5 \bar{\phi}_0 
  - 2\mathrm{i}\,\bar{\zeta}_4 \phi_1 
  + \mathrm{i}\,\zeta_2 \bar{\phi}_1 
  - \mathrm{i}\,\eta_5 \bar\chi_0 
  + 2\mathrm{i}\,\bar{\eta}_4 \chi_1 
  - \mathrm{i}\,\eta_2 \bar{\chi}_1 
  - \beta \lambda 
  - \alpha \mu 
  - \lambda \tau, \label{Deltaalpha}\\
D \alpha &= \mathrm{i}\,\bar{\zeta}_3 \phi_0 
  - 2\mathrm{i}\,\zeta_1 \bar{\phi}_0 
  + \mathrm{i}\,\bar{\zeta}_0 \phi_1 
  - \mathrm{i}\,\bar{\eta}_3 \chi_0 
  + 2\mathrm{i}\,\eta_1 \bar{\chi}_0 
  - \mathrm{i}\,\bar{\eta}_0 \chi_1 \nonumber\\
 & - 2\alpha \epsilon 
  - \bar{\beta} \epsilon 
  + \alpha \bar{\epsilon} 
  + \varepsilon \pi 
  + \alpha \rho 
  + \pi \rho 
  + \beta \bar{\sigma} 
  + \bar{\delta} \epsilon, \label{Dalpha}\\
\bar{\delta}\beta &= \TiPsi_2 
  - \mathrm{i}\,\bar{\zeta}_4 \phi_0 
  + \mathrm{i}\,\zeta_4 \bar{\phi}_0 
  - \mathrm{i}\,\bar{\zeta}_1 \phi_1 
  + \mathrm{i}\,\zeta_1 \bar{\phi}_1 
  + \mathrm{i}\,\bar{\eta}_4 \chi_0 \nonumber\\
  &- \tfrac{1}{3}\mathrm{i}\,m \phi_1 \chi_0 
  - \mathrm{i}\,\eta_4 \bar{\chi}_0 
  + \tfrac{1}{3}\mathrm{i}m\,\overline{ \phi_1} \,\bar{\chi}_0 
  + \mathrm{i}\,\bar{\eta}_1 \chi_1 
  + \tfrac{1}{3}\mathrm{i}\,m \phi_0 \chi_1 
  - \mathrm{i}\,\eta_1 \bar{\chi}_1 
  - \tfrac{1}{3}\mathrm{i}\,m\overline{ \phi_0}\,\bar{\chi}_1 \nonumber\\
 & - \alpha \bar{\alpha} 
  + 2\alpha \beta 
  - \beta \bar{\beta} 
  - \mu \rho 
  + \lambda \sigma 
  + \delta \alpha, \label{alphabeta}\\
\Delta \epsilon &= -\TiPsi_2 
  - \mathrm{i}\,\bar{\zeta}_4 \phi_0 
  + \mathrm{i}\,\zeta_4 \bar{\phi}_0 
  - \mathrm{i}\,\bar{\zeta}_1 \phi_1 
  + \mathrm{i}\,\zeta_1 \bar{\phi}_1 
  + \mathrm{i}\,\bar{\eta}_4 \chi_0 \nonumber\\
&  + \tfrac{1}{3}\mathrm{i}\,m \phi_1 \chi_0 
  - \mathrm{i}\,\eta_4 \bar{\chi}_0 
  - \tfrac{1}{3}\mathrm{i}\,m\overline{ \phi_1}\,\bar{\chi}_0 
  + \mathrm{i}\,\bar{\eta}_1 \chi_1 
  - \tfrac{1}{3}\mathrm{i}\,m \phi_0 \chi_1 
  - \mathrm{i}\,\eta_1 \bar{\chi}_1 
  + \tfrac{1}{3}\mathrm{i}\,m\overline{ \phi_0}\,\bar{\chi}_1 \nonumber\\
&  - \beta \pi 
  - \alpha \bar\pi 
  - \alpha \tau 
  - \pi \tau 
  - \beta \bar{\tau}. \label{Deltaepsilon}
\end{align}
\end{subequations}

\subsection{The Bianchi identity}
\label{BianchiIdentity}
\begin{align}
\mthorn'\,\Psi_{0} &= 
  4\,\mathrm{i}\,\eta_{0}\,\overline{\eta}_{2}
  -4\,\mathrm{i}\,\overline{\eta}_{1}\,\eta_{3}
  -4\,\mathrm{i}\,\zeta_{0}\,\overline{\zeta}_{2}
  +4\,\mathrm{i}\,\overline{\zeta}_{1}\,\zeta_{3}
  +3\,\mathrm{i}\,m\,\eta_{3}\,\phi_{0}
  -4\,\overline{\zeta}_{2}\,\phi_{0}^{2}\,\overline{\phi}_{0}
  -2\,\mathrm{i}\,\TiPsi_{1}\,\phi_{0}\,\overline{\phi}_{1}
  +8\,\overline{\zeta}_{1}\,\phi_{0}^{2}\,\overline{\phi}_{1} \nonumber\\
&\quad
  +2\,\mathrm{i}\,\Psi_{0}\,\phi_{1}\,\overline{\phi}_{1}
  -4\,\zeta_{0}\,\phi_{0}\,\overline{\phi}_{1}^{2}
  -3\,\mathrm{i}\,m\,\zeta_{3}\,\chi_{0}
  +4\,\overline{\eta}_{2}\,\phi_{0}\,\overline{\phi}_{0}\,\chi_{0}
  -8\,\overline{\eta}_{1}\,\phi_{0}\,\overline{\phi}_{1}\,\chi_{0} \nonumber\\
&\quad
  +4\,\eta_{3}\,\phi_{0}\,\overline{\phi}_{1}\,\overline{\chi}_{0}
  +4\,\overline{\zeta}_{2}\,\phi_{0}\,\chi_{0}\,\overline{\chi}_{0}
  -4\,\zeta_{3}\,\overline{\phi}_{1}\,\chi_{0}\,\overline{\chi}_{0}
  -4\,\overline{\eta}_{2}\,\chi_{0}^{2}\,\overline{\chi}_{0}
  -4\,\eta_{3}\,\phi_{0}\,\overline{\phi}_{0}\,\bar\chi_{1} \nonumber\\
&\quad
  +4\,\eta_{0}\,\phi_{0}\,\overline{\phi}_{1}\,\bar\chi_{1}
  +2\,\mathrm{i}\,\TiPsi_{1}\,\chi_{0}\,\overline{\chi}_{1}
  -8\,\bar\zeta_{1}\,\phi_{0}\,\chi_{0}\,\overline{\chi}_{1}
  +4\,\zeta_{3}\,\overline{\phi}_{0}\,\chi_{0}\,\overline{\chi}_{1}
  +4\,\zeta_{0}\,\overline{\phi}_{1}\,\chi_{0}\,\overline{\chi}_{1} \nonumber\\
&\quad
  +8\,\bar\eta_{1}\,\chi_{0}^{2}\,\bar\chi_{1}
  -2\,\mathrm{i}\,\Psi_{0}\,\chi_{1}\,\overline{\chi}_{1}
  -4\,\eta_{0}\,\chi_{0}\,\overline{\chi}_{1}^{2}
  -\Psi_{0}\,\mu
  -\TiPsi_{1}\,\bar\pi
  -\mathrm{i}\,\overline{\zeta}_{1}\,\phi_{0}\,\bar\pi \nonumber\\
&\quad
  +\mathrm{i}\,\zeta_{3}\,\overline{\phi}_{0}\,\bar\pi
  +\mathrm{i}\,\overline{\eta}_{1}\,\chi_{0}\,\bar\pi
  -\mathrm{i}\,\eta_{3}\,\overline{\chi}_{0}\,\bar\pi
  +2\,\mathrm{i}\,\zeta_{3}\,\overline{\phi}_{1}\,\rho
  -2\,\mathrm{i}\,\eta_{3}\,\overline{\chi}_{1}\,\rho
  +3\,\TiPsi_{2}\,\sigma \nonumber\\
&\quad
  -2\,\mathrm{i}\,\zeta_{4}\,\overline{\phi}_{0}\,\sigma
  -2\,\mathrm{i}\,\overline{\zeta}_{1}\,\phi_{1}\,\sigma
  +2\,\mathrm{i}\,\zeta_{1}\,\overline{\phi}_{1}\,\sigma
  -2\,\mathrm{i}\,m\,\phi_{1}\,\chi_{0}\,\sigma
  +2\,\mathrm{i}\,\eta_{4}\,\overline{\chi}_{0}\,\sigma \nonumber\\
&\quad
  +2\,\mathrm{i}\,m\,\overline{\phi}_{1}\,\bar\chi_{0}\,\sigma
  +2\,\mathrm{i}\,\overline{\eta}_{1}\,\chi_{1}\,\sigma
  +2\,\mathrm{i}\,m\,\phi_{0}\,\chi_{1}\,\sigma
  -2\,\mathrm{i}\,\eta_{1}\,\overline{\chi}_{1}\,\sigma \nonumber\\
&\quad
  -2\,\mathrm{i}\,m\,\overline{\phi}_{0}\,\bar\chi_{1}\,\sigma
  -4\,\TiPsi_{1}\,\tau
  -8\,\bar\zeta_{1}\,\phi_{0}\,\tau
  +4\,\zeta_{3}\,\overline{\phi}_{0}\,\tau
  +4\,\zeta_{0}\,\overline{\phi}_{1}\,\tau \nonumber\\
&\quad
  +8\,\overline{\eta}_{1}\,\chi_{0}\,\tau
  -4\,\eta_{3}\,\overline{\chi}_{0}\,\tau
  -4\,\eta_{0}\,\overline{\chi}_{1}\,\tau
  -2\,\mathrm{i}\,\chi_{0}\,\bigl(\meth\,\overline{\eta}_{1}\bigr)
  +2\,\mathrm{i}\,\overline{\chi}_{0}\,\bigl(\meth\,\eta_{3}\bigr) \nonumber\\
&\quad
  +2\,\mathrm{i}\,\phi_{0}\,\bigl(\meth\,\bar{\zeta}_{1}\bigr)
  -2\,\mathrm{i}\,\overline{\phi}_{0}\,\bigl(\meth\,\zeta_{3}\bigr)
  +\meth\,\TiPsi_{1}, \label{thornprimePsi0}
\end{align}

\begin{align}
\mthorn'\,\TiPsi_{1} &= 
  2\,\mathrm{i}\,\eta_{1}\,\overline{\eta}_{2}
  -2\,\mathrm{i}\,\eta_{3}\,\overline{\eta}_{4}
  -2\,\mathrm{i}\,\zeta_{1}\,\overline{\zeta}_{2}
  +2\,\mathrm{i}\,\zeta_{3}\,\overline{\zeta}_{4} 
  +\tfrac{4\,\mathrm{i}}3\,m\,\eta_{4}\,\phi_{0}
  -\tfrac{\mathrm{i}}3\,m\,\overline{\eta}_{2}\,\overline{\phi}_{0}
  -\tfrac{\mathrm{i}}3\,m\,\eta_{3}\,\phi_{1}
  +\tfrac{4\,\mathrm{i}}3\,m\,\overline{\eta}_{1}\,\overline{\phi}_{1} \nonumber\\
&\quad
 -\tfrac{4\,\mathrm{i}}3\,m\,\zeta_{4}\,\chi_{0}
  +\tfrac{\mathrm{i}}3\,m\,\overline{\zeta}_{2}\,\overline{\chi}_{0}
  +\tfrac{\mathrm{i}}3\,m\,\zeta_{3}\,\chi_{1} 
  -\tfrac{4\,\mathrm{i}}3\,m\,\overline{\zeta}_{1}\,\overline{\chi}_{1}
  -2\,\TiPsi_{1}\,\mu
  +2\,\mathrm{i}\,\zeta_{3}\,\overline{\phi}_{0}\,\mu
  -2\,\mathrm{i}\,\eta_{3}\,\overline{\chi}_{0}\,\mu \nonumber\\
&\quad
  -\mathrm{i}\,\overline{\zeta}_{2}\,\phi_{0}\,\pi
  -\mathrm{i}\,\zeta_{3}\,\overline{\phi}_{1}\,\pi
  +\mathrm{i}\,\overline{\eta}_{2}\,\chi_{0}\,\pi
  +\mathrm{i}\,\eta_{3}\,\overline{\chi}_{1}\,\pi
  -2\,\mathrm{i}\,\overline{\zeta}_{5}\,\phi_{0}\,\rho 
  +4\,\mathrm{i}\,\zeta_{4}\,\overline{\phi}_{1}\,\rho
  +2\,\mathrm{i}\,\overline{\eta}_{5}\,\chi_{0}\,\rho\nonumber\\
&\quad
  -4\,\mathrm{i}\,\eta_{4}\,\overline{\chi}_{1}\,\rho
  +2\,\TiPsi_{3}\,\sigma
  -2\,\mathrm{i}\,\zeta_{5}\,\overline{\phi}_{0}\,\sigma 
  +4\,\mathrm{i}\,\bar\zeta_{4}\,\phi_{1}\,\sigma
  -2\,\mathrm{i}\,\zeta_{2}\,\overline{\phi}_{1}\,\sigma
  +2\,\mathrm{i}\,\eta_{5}\,\overline{\chi}_{0}\,\sigma
  -4\,\mathrm{i}\,\bar\eta_{4}\,\chi_{1}\,\sigma \nonumber\\
&\quad
  -3\,\TiPsi_{2}\,\tau
  +2\,\mathrm{i}\,\bar\zeta_{4}\,\phi_{0}\,\tau
  -2\,\mathrm{i}\,\zeta_{4}\,\overline{\phi}_{0}\,\tau
  +2\,\mathrm{i}\,\bar\zeta_{1}\,\phi_{1}\,\tau
  -2\,\mathrm{i}\,\zeta_{1}\,\overline{\phi}_{1}\,\tau 
  -2\,\mathrm{i}\,\bar\eta_{4}\,\chi_{0}\,\tau
  +2\,\mathrm{i}\,m\,\phi_{1}\,\chi_{0}\,\tau\nonumber\\
&\quad
  +2\,\mathrm{i}\,\eta_{4}\,\overline{\chi}_{0}\,\tau
  -2\,\mathrm{i}\,m\,\overline{\phi}_{1}\,\overline{\chi}_{0}\,\tau 
  -2\,\mathrm{i}\,\bar\eta_{1}\,\chi_{1}\,\tau
  -2\,\mathrm{i}\,m\,\phi_{0}\,\chi_{1}\,\tau
  +2\,\mathrm{i}\,\eta_{1}\,\overline{\chi}_{1}\,\tau
  +2\,\mathrm{i}\,m\,\overline{\phi}_{0}\,\overline{\chi}_{1}\,\tau \nonumber\\
&\quad
  +2\,\mathrm{i}\,\bar\zeta_{2}\,\phi_{0}\,\tau
  -2\,\mathrm{i}\,\zeta_{3}\,\overline{\phi}_{1}\,\bar\tau
  -2\,\mathrm{i}\,\bar\eta_{2}\,\chi_{0}\,\bar\tau
  +2\,\mathrm{i}\,\eta_{3}\,\overline{\chi}_{1}\,\tau 
  +2\,\mathrm{i}\,\eta_{2}\,\overline{\chi}_{1}\,\sigma\nonumber\\
&\quad
  +\meth\,\TiPsi_{2}
  +2\,\mathrm{i}\,\chi_{0}\,\bigl(\meth'\,\bar\eta_{2}\bigr)
  -2\,\mathrm{i}\,\overline{\chi}_{1}\,\bigl(\meth'\,\eta_{3}\bigr)
  -2\,\mathrm{i}\,\phi_{0}\,\bigl(\meth'\,\bar\zeta_{2}\bigr)
  +2\,\mathrm{i}\,\overline{\phi}_{1}\,\bigl(\meth'\,\zeta_{3}\bigr), \label{thornprimePsi1}
\end{align}

\begin{align}
\mthorn'\tilde{\Psi}_2
&=2\,\mathrm{i}\,\eta_2\bar{\eta}_2
-4\,\mathrm{i}\,\eta_4\bar{\eta}_4
+2\,\mathrm{i}\,\eta_1\bar{\eta}_5
-2\,\mathrm{i}\,\zeta_2\bar{\zeta}_2
+4\,\mathrm{i}\,\zeta_4\bar{\zeta}_4
-2\,\mathrm{i}\,\zeta_1\bar{\zeta}_5 
+\frac{2\mathrm{i}}{3}m\,\eta_5\phi_0
+\frac{\mathrm{i}}{3}m\,\bar{\eta}_5\bar{\phi}_0\nonumber\\
&\quad
+\frac{7\mathrm{i}}{3}m\,\eta_4\phi_1
-2\,\bar{\zeta}_5\phi_0\bar{\phi}_0\phi_1
-2\,\bar{\zeta}_2\bar{\phi}_0\phi_1^2 
-\frac{\mathrm{i}}{3}m\,\bar{\eta}_4\bar{\phi}_1
-2\,\zeta_5\phi_0\bar{\phi}_0\bar{\phi}_1
+4\,\bar{\zeta}_4\phi_0\phi_1\bar{\phi}_1
+4\,\zeta_4\bar{\phi}_0\phi_1\bar{\phi}_1 \nonumber\\
&\quad
-2\,\zeta_2\phi_0\bar{\phi}_1^2
-\frac{2\mathrm{i}}{3}m\,\zeta_5\chi_0
+2\,\bar{\eta}_5\bar{\phi}_0\phi_1\chi_0
-\frac{\mathrm{i}}{3}m\,\bar{\zeta}_5\bar{\chi}_0
+2\,\eta_5\phi_0\bar{\phi}_1\bar{\chi}_0
-\frac{7\mathrm{i}}{3}m\,\zeta_4\chi_1
+2\,\bar{\eta}_2\bar{\phi}_0\phi_1\chi_1 \nonumber\\
&\quad
+2\,\bar{\zeta}_5\phi_0\bar{\chi}_0\chi_1
+2\,\bar{\zeta}_2\phi_1\bar{\chi}_0\chi_1
-4\,\zeta_4\bar{\phi}_1\bar\chi_0\chi_1
-2\,\eta_5\chi_0\bar{\chi}_0\chi_1
-2\,\bar\eta_2\bar{\chi}_0\chi_1^2
+\frac{\mathrm{i}}{3}m\,\bar{\zeta}_4\bar\chi_1 \nonumber\\
&\quad
-4\,\eta_4\bar{\phi}_0\phi_1\bar{\chi}_1
+2\,\eta_2\phi_0\bar{\phi}_1\bar{\chi}_1
+2\,\zeta_5\bar{\phi}_0\chi_0\bar{\chi}_1
-4\,\bar{\zeta}_4\phi_1\chi_0\bar{\chi}_1
+2\,\zeta_2\bar{\phi}_1\chi_0\bar{\chi}_1 \nonumber\\
&\quad
-4\,\bar\eta_4\phi_0\bar{\phi}_1\chi_1
-2\,\eta_5\chi_0\bar{\chi}_0\bar{\chi}_1
+4\,\bar{\eta}_4\chi_0\chi_1\bar{\chi}_1
+4\,\eta_4\bar{\chi}_0\chi_1\bar{\chi}_1
-2\,\eta_2\chi_0\bar{\chi}_1^2 \nonumber\\
&\quad
-2\,\mathrm{i}\,\zeta_3\bar{\phi}_1\lambda
+2\,\mathrm{i}\,\eta_3\bar{\chi}_1\lambda
-3\,\tilde{\Psi}_2\,\mu
+2\,\mathrm{i}\,\bar{\zeta}_4\phi_0\mu
+2\,\mathrm{i}\,\zeta_4\bar{\phi}_0\mu
-2\,\mathrm{i}\,\bar{\eta}_4\chi_0\mu \nonumber\\
&\quad
+2\,\mathrm{i}\,m\,\phi_1\chi_0\mu
-2\,\mathrm{i}\,\eta_4\bar{\chi}_0\mu
-2\,\mathrm{i}\,m\,\bar{\phi}_1\bar{\chi}_0\mu
-2\,\mathrm{i}\,m\,\phi_0\chi_1\mu \nonumber\\
&\quad
+2\,\mathrm{i}\,m\,\bar{\phi}_0\bar{\chi}_1\mu
+\mathrm{i}\,\zeta_4\bar{\phi}_1\pi
-\mathrm{i}\,\eta_4\bar{\chi}_1\pi
+\tilde{\Psi}_3\,\bar\pi
-3\,\mathrm{i}\,\zeta_5\bar{\phi}_0\bar\pi
+\mathrm{i}\,\bar{\zeta}_4\phi_1\bar\pi \nonumber\\
&\quad
-\mathrm{i}\,\zeta_2\bar{\phi}_1\bar\pi
+3\,\mathrm{i}\,\eta_5\bar{\chi}_0\bar\pi
- \mathrm{i}\,\bar{\eta}_4\chi_1\bar\pi
+ \mathrm{i}\,\eta_2\bar{\chi}_1\bar\pi
+2\,\mathrm{i}\,\zeta_5\bar{\phi}_1\rho
-2\,\mathrm{i}\,\eta_5\bar{\chi}_1\rho
+\Psi_4\,\sigma
-2\,\tilde{\Psi}_3\,\tau \nonumber\\
&\quad
+2\,\mathrm{i}\,\zeta_5\bar{\phi}_0\tau
-4\,\mathrm{i}\,\bar{\zeta}_4\phi_1\tau
+2\,\mathrm{i}\,\zeta_2\bar{\phi}_1\tau
-2\,\mathrm{i}\,\eta_5\bar{\chi}_0\tau
+4\,\mathrm{i}\,\bar{\eta}_4\chi_1\tau
-2\,\mathrm{i}\,\eta_2\bar{\chi}_1\tau
+2\,\mathrm{i}\,\bar{\chi}_1(\meth\eta_2) \nonumber\\
&\quad
-2\,\mathrm{i}\,\chi_1(\meth\bar{\eta}_4)
+2\,\mathrm{i}\,\bar{\chi}_0(\meth\eta_5)
-2\,\mathrm{i}\,\bar{\phi}_1(\meth\zeta_2)
+2\,\mathrm{i}\,\phi_1(\meth\bar{\zeta}_4)
-2\,\mathrm{i}\,\bar{\phi}_0(\meth\zeta_5)
+\meth\tilde{\Psi}_3 \nonumber\\
&\quad
-2\,\mathrm{i}\,\bar{\chi}_1(\meth'\eta_4)
+2\,\mathrm{i}\,\bar{\phi}_1(\meth'\zeta_4)\,, \label{thornprimePsi2}
\end{align}

\begin{align}
&\mthorn'\tilde{\Psi}_3
=-2\,\mathrm{i}\,\bar{\eta}_4\eta_5
+2\,\mathrm{i}\,\eta_2\bar{\eta}_5
+2\,\mathrm{i}\,\bar{\zeta}_4\zeta_5
-2\,\mathrm{i}\,\bar{\zeta}_2\zeta_5
+\mathrm{i}\,m\,\eta_5\phi_1
-\mathrm{i}\,m\,\zeta_5\chi_1
+2\,\mathrm{i}\,\bar{\zeta}_2\phi_1\lambda
-4\,\mathrm{i}\,\zeta_4\bar{\phi}_1\lambda \nonumber\\
&\quad
-2\,\mathrm{i}\,\bar{\eta}_2\chi_1\lambda
+4\,\mathrm{i}\,\eta_4\bar{\chi}_1\lambda
-4\,\tilde{\Psi}_3\,\mu
+4\,\mathrm{i}\,\zeta_5\bar{\phi}_0\mu
-4\,\mathrm{i}\,\bar{\zeta}_4\phi_1\mu
+2\,\mathrm{i}\,\zeta_2\bar{\phi}_1\mu
-4\,\mathrm{i}\,\eta_5\bar{\chi}_0\mu
+4\,\mathrm{i}\,\bar{\eta}_4\chi_1\mu \nonumber\\
&\quad
-2\,\mathrm{i}\,\eta_2\bar{\chi}_1\mu
-3\,\mathrm{i}\,\bar{\zeta}_5\phi_1\pi
+3\,\mathrm{i}\,\zeta_5\bar{\phi}_1\pi
+3\,\mathrm{i}\,\bar{\eta}_5\chi_1\pi
-3\,\mathrm{i}\,\eta_5\bar{\chi}_1\pi
+2\,\Psi_4\,\bar\pi
-\Psi_4\,\tau
+2\,\mathrm{i}\,\bar\zeta_5\phi_1\bar\tau
-2\,\mathrm{i}\,\zeta_5\bar{\phi}_1\tau \nonumber\\
&\quad
-2\,\mathrm{i}\,\bar\eta_5\chi_1\bar\tau
+2\,\mathrm{i}\,\eta_5\bar{\chi}_1\bar\tau
+\meth\Psi_4
-2\,\mathrm{i}\,\bar{\chi}_1\bigl(\meth'\eta_5\bigr)
+2\,\mathrm{i}\,\chi_1\bigl(\meth'\bar{\eta}_5\bigr)
+2\,\mathrm{i}\,\bar{\phi}_1\bigl(\meth'\zeta_5\bigr)
-2\,\mathrm{i}\,\phi_1\bigl(\meth'\bar{\zeta}_5\bigr), \label{thornprimePsi3}
\end{align}

\begin{align}
&\mthorn\tilde{\Psi}_1
=-2\,\mathrm{i}\,\eta_0\bar{\eta}_1
+2\,\mathrm{i}\,\bar{\eta}_0\eta_3
+2\,\mathrm{i}\,\zeta_0\bar{\zeta}_1
-2\,\mathrm{i}\,\bar{\zeta}_0\zeta_3
-\mathrm{i}\,m\,\eta_0\phi_0
+\mathrm{i}\,m\,\zeta_0\chi_0
-\Psi_0\,\pi
+\mathrm{i}\,\bar{\zeta}_0\phi_0\,\bar\pi \nonumber\\
&\quad
-\mathrm{i}\,\zeta_0\bar\phi_0\,\bar\pi
-\mathrm{i}\,\bar{\eta}_0\chi_0\,\bar\pi
+\mathrm{i}\,\eta_0\bar{\chi}_0\,\bar\pi
+4\,\tilde{\Psi}_1\,\rho
+4\,\mathrm{i}\,\bar{\zeta}_1\phi_0\,\rho
-2\,\mathrm{i}\,\zeta_3\bar{\phi}_0\,\rho
-4\,\mathrm{i}\,\zeta_0\bar{\phi}_1\,\rho \nonumber\\
&\quad
-4\,\mathrm{i}\,\bar{\eta}_1\chi_0\,\rho
+2\,\mathrm{i}\,\eta_3\bar{\chi}_0\,\rho
+4\,\mathrm{i}\,\eta_0\bar{\chi}_1\,\rho
-2\,\mathrm{i}\,\bar{\zeta}_3\phi_0\,\sigma
+4\,\mathrm{i}\,\zeta_1\bar{\phi}_0\,\sigma
+2\,\mathrm{i}\,\bar{\eta}_3\chi_0\,\sigma \nonumber\\
&\quad
-4\,\mathrm{i}\,\eta_1\bar{\chi}_0\,\sigma
+\tilde{\Psi}_1\,\omega
-2\,\mathrm{i}\,\bar{\chi}_0\bigl(\meth\eta_0\bigr)
+2\,\mathrm{i}\,\chi_0\bigl(\meth\bar{\eta}_0\bigr)
+2\,\mathrm{i}\,\bar{\phi}_0\bigl(\meth\zeta_0\bigr)
-2\,\mathrm{i}\,\phi_0\bigl(\meth'\bar{\zeta}_0\bigr)
+\meth'\Psi_0, \label{thornPsi1}
\end{align}

\begin{align}
\mthorn\tilde{\Psi}_2&
=-4\,\mathrm{i}\,\eta_1\bar{\eta}_1
+2\,\mathrm{i}\,\eta_3\bar{\eta}_3
+2\,\mathrm{i}\,\bar{\eta}_0\eta_4
+4\,\mathrm{i}\,\zeta_1\bar{\zeta}_1
-2\,\mathrm{i}\,\zeta_3\bar{\zeta}_3
-2\,\mathrm{i}\,\bar{\zeta}_0\zeta_4
-\tfrac{7\,\mathrm{i}}{3}m\,\eta_1\phi_0 \nonumber\\
&\quad
+\tfrac{\mathrm{i}}{3}m\,\bar{\eta}_1\bar{\phi}_0
-\tfrac{2\,\mathrm{i}}{3}m\,\eta_0\phi_1
+4\,\bar{\zeta}_1\phi_0\bar{\phi}_0\phi_1
-2\,\zeta_3\bar\phi_0^2\phi_1 
-\tfrac{\mathrm{i}}{3}m\,\bar{\eta}_0\bar{\phi}_1
-2\,\bar{\zeta}_3\phi_0^2\bar{\phi}_1\nonumber\\
&\quad
+4\,\zeta_1\phi_0\bar{\phi}_0\bar{\phi}_1
-2\,\bar\zeta_0\phi_0\phi_1\bar{\phi}_1
-2\,\zeta_0\bar{\phi}_0\phi_1\bar{\phi}_1 
+\tfrac{7\,\mathrm{i}}{3}m\,\zeta_1\chi_0
-4\,\bar{\eta}_1\bar{\phi}_0\phi_1\chi_0
+2\,\bar{\eta}_3\phi_0\bar{\phi}_1\chi_0 \nonumber\\
&\quad
-\tfrac{\mathrm{i}}{3}m\,\bar{\zeta}_1\bar{\chi}_0
+2\,\eta_3\bar{\phi}_0\phi_1\bar{\chi}_0
-4\,\eta_1\phi_0\bar{\phi}_1\bar{\chi}_0
+\tfrac{2\,\mathrm{i}}{3}m\,\zeta_0\chi_1 \nonumber\\
&\quad
+2\,\bar{\eta}_0\phi_0\bar{\phi}_1\chi_1
-4\,\bar{\zeta}_1\phi_0\bar{\chi}_0\chi_1
+2\,\zeta_3\bar{\phi}_0\bar{\chi}_0\chi_1
+2\,\zeta_0\bar{\phi}_1\bar{\chi}_0\chi_1
+4\,\bar{\eta}_1\chi_0\bar{\chi}_0\chi_1 \nonumber\\
&\quad
-2\,\eta_3\bar{\chi}_0^2\chi_1
+\tfrac{\mathrm{i}}{3}m\,\bar{\zeta}_0\bar{\chi}_1
+2\,\eta_0\bar{\phi}_0\phi_1\bar{\chi}_1
+2\,\bar\zeta_3\phi_0\chi_1\bar{\chi}_1
-4\,\zeta_1\bar\phi_0\bar{\chi}_1\chi_0 \nonumber\\
&\quad
-2\,\bar{\eta}_0\phi_0\bar{\chi}_1\chi_0
-2\,\eta_0\bar{\chi}_0\chi_1\bar{\chi}_1
-\Psi_0\,\lambda
-2\,\mathrm{i}\,\zeta_0\bar{\phi}_0\mu
+2\,\mathrm{i}\,\eta_0\bar{\chi}_0\mu
+\tilde{\Psi}_1\,\pi
+3\,\mathrm{i}\,\bar{\zeta}_1\phi_0\pi \nonumber\\
&\quad
-\,\mathrm{i}\,\zeta_3\bar{\phi}_0\pi
+\,\mathrm{i}\,\zeta_0\bar{\phi}_1\pi
-3\,\mathrm{i}\,\bar{\eta}_1\chi_0\pi
+\,\mathrm{i}\,\eta_3\bar{\chi}_0\pi
-\,\mathrm{i}\,\eta_0\bar{\chi}_1\pi
-\,\mathrm{i}\,\zeta_1\bar{\phi}_0\pi
+\,\mathrm{i}\,\eta_1\bar{\chi}_0\pi \nonumber\\
&\quad
+2\,\bar{\zeta}_0\phi_1\bar{\chi}_0\chi_1
-2\,\bar{\eta}_3\chi^2_0\bar{\chi}_1
+4\,\eta_1\chi_0\bar{\chi}_1\bar\chi_0
+3\,\tilde{\Psi}_2\,\rho
-2\,\mathrm{i}\,\bar{\zeta}_1\phi_1\,\rho  \nonumber\\
&\quad
-2\,\mathrm{i}\,\zeta_1\bar{\phi}_1\rho
-2\,\mathrm{i}\,m\,\phi_1\chi_0\,\rho
+2\,\mathrm{i}\,m\,\bar{\phi}_1\bar{\chi}_0\,\rho
+2\,\mathrm{i}\,\bar{\eta}_1\chi_1\,\rho
+2\,\mathrm{i}\,m\,\phi_0\chi_1\,\rho \nonumber\\
&\quad
+2\,\mathrm{i}\,\eta_1\bar{\chi}_1\,\rho
-2\,\mathrm{i}\,m\,\bar{\phi}_0\bar{\chi}_1\,\rho
+2\,\mathrm{i}\,\zeta_2\bar{\phi}_0\,\sigma
-2\,\mathrm{i}\,\eta_2\bar{\chi}_0\,\sigma
-2\,\mathrm{i}\,\bar{\chi}_0\bigl(\meth\eta_1\bigr)
+2\,\mathrm{i}\,\bar{\phi}_0\bigl(\meth\zeta_1\bigr) \nonumber\\
&\quad
+2\,\mathrm{i}\,\chi_1\bigl(\meth'\eta_0\bigr)
-2\,\mathrm{i}\,\chi_0\bigl(\meth'\bar{\eta}_1\bigr)
+2\,\mathrm{i}\,\bar{\chi}_0\bigl(\meth'\eta_3\bigr)
-2\,\mathrm{i}\,\bar{\phi}_1\bigl(\meth'\zeta_0\bigr)
+2\,\mathrm{i}\,\phi_0\bigl(\meth'\zeta_1\bigr)
-2\,\mathrm{i}\,\bar{\phi}_0\bigl(\meth'\zeta_3\bigr)
+\meth'\tilde{\Psi}_1 \label{thornPsi2}
\end{align}

\begin{align}
\mthorn\tilde{\Psi}_3
&=-2\,\mathrm{i}\,\bar{\eta}_1\eta_2
+2\,\mathrm{i}\,\bar{\eta}_3\eta_4
+2\,\mathrm{i}\,\bar{\zeta}_1\zeta_2
-2\,\mathrm{i}\,\bar{\zeta}_3\zeta_4
+\frac{\mathrm{i}}{3}m\,\eta_2\phi_0
-\frac{4\,\mathrm{i}}{3}m\,\bar{\eta}_4\bar{\phi}_0 
-\frac{4\,\mathrm{i}}{3}m\,\eta_1\phi_1
+\frac{\mathrm{i}}{3}m\,\bar{\eta}_3\bar{\phi}_1\nonumber\\
&\quad
-\frac{\mathrm{i}}{3}m\,\zeta_2\chi_0
+\frac{4\,\mathrm{i}}{3}m\,\bar{\zeta}_4\bar\chi_0 
+\frac{4\,\mathrm{i}}{3}m\,\zeta_1\chi_1
-\frac{\mathrm{i}}{3}m\,\bar{\zeta}_3\bar{\chi}_1
-2\,\tilde{\Psi}_1\,\lambda
-4\,\mathrm{i}\,\bar{\zeta}_1\phi_0\lambda
+2\,\mathrm{i}\,\zeta_3\bar{\phi}_0\lambda \nonumber\\
&\quad
+2\,\mathrm{i}\,\zeta_0\bar{\phi}_1\lambda
+4\,\mathrm{i}\,\bar\eta_1\chi_0\lambda
-2\,\mathrm{i}\,\eta_3\bar{\chi}_0\lambda
-2\,\mathrm{i}\,\eta_0\bar{\chi}_1\lambda
-4\,\mathrm{i}\,\zeta_1\bar{\phi}_0\mu
+2\,\mathrm{i}\,\bar{\zeta}_0\phi_1\mu
+4\,\mathrm{i}\,\bar\eta_1\chi_0\mu \nonumber\\
&\quad
-2\,\mathrm{i}\,\bar{\eta}_0\chi_1\,\mu
+3\,\tilde{\Psi}_2\,\pi
-2\,\mathrm{i}\,\bar{\zeta}_4\phi_0\,\pi
+2\,\mathrm{i}\,\zeta_4\bar{\phi}_0\,\pi
-2\,\mathrm{i}\,\bar{\zeta}_1\phi_1\,\pi
+2\,\mathrm{i}\,\zeta_1\bar{\phi}_1\,\pi
+2\,\mathrm{i}\,\bar{\eta}_4\chi_0\,\pi \nonumber\\
&\quad
-2\,\mathrm{i}\,m\,\phi_1\chi_0\,\pi
-2\,\mathrm{i}\,\eta_4\bar{\chi}_0\,\pi
+2\,\mathrm{i}\,m\,\bar{\phi}_1\bar{\chi}_0\,\pi
+2\,\mathrm{i}\,\bar{\eta}_1\chi_1\,\pi
+2\,\mathrm{i}\,m\,\phi_0\chi_1\,\pi \nonumber\\
&\quad
-2\,\mathrm{i}\,\eta_1\bar{\chi}_1\,\pi
-2\,\mathrm{i}\,m\,\bar{\phi}_0\bar{\chi}_1\,\pi
+3\,\mathrm{i}\,\zeta_2\bar{\phi}_0\,\bar{\pi}
-\mathrm{i}\,\bar{\zeta}_3\phi_1\,\bar{\pi}
-3\,\mathrm{i}\,\eta_2\bar{\chi}_0\,\bar{\pi}
+\mathrm{i}\,\bar{\eta}_3\chi_1\,\bar{\pi} \nonumber\\
&\quad
+2\,\tilde{\Psi}_3\,\rho
-2\,\mathrm{i}\,\zeta_2\bar{\phi}_1\,\rho
+2\,\mathrm{i}\,\eta_2\bar{\chi}_1\,\rho
-\tilde{\Psi}_3\,\omega
-2\,\mathrm{i}\,\bar{\chi}_0\bigl(\meth\eta_2\bigr) \nonumber\\
&\quad
+2\,\mathrm{i}\,\chi_1\bigl(\meth\bar{\eta}_3\bigr)
+2\,\mathrm{i}\,\bar{\phi}_0\bigl(\meth\zeta_2\bigr)
-2\,\mathrm{i}\,\phi_1\bigl(\meth\bar{\zeta}_3\bigr)
+\meth'\tilde{\Psi}_2, \label{thornPsi3}
\end{align}

\begin{align}
\mthorn\Psi_4
&=-4\,\mathrm{i}\,\eta_2\bar{\eta}_4
+4\,\mathrm{i}\,\bar{\eta}_3\eta_5
+4\,\mathrm{i}\,\zeta_2\bar{\zeta}_4
-4\,\mathrm{i}\,\bar{\zeta}_3\zeta_5
+2\,\mathrm{i}\,\Psi_4\,\phi_0\bar{\phi}_0
-3\,\mathrm{i}\,m\,\eta_2\phi_1
-2\,\mathrm{i}\,\tilde{\Psi}_3\,\bar{\phi}_0\phi_1 \nonumber\\
&\quad
-4\,\zeta_5\bar{\phi}_0^2\phi_1
+8\,\bar{\zeta}_4\bar{\phi}_0\phi_1^2
-4\,\bar{\zeta}_3\phi_1^2\bar{\phi}_1
+4\,\eta_5\bar{\phi}_0\phi_1\bar{\chi}_0
-4\,\eta_2\phi_1\bar{\phi}_1\bar{\chi}_0
-2\,\mathrm{i}\,\Psi_4\,\chi_0\bar{\chi}_0
+3\,\mathrm{i}\,m\,\zeta_2\chi_1 \nonumber\\
&\quad
-8\,\bar{\eta}_4\bar{\phi}_0\phi_1\chi_1
+4\,\bar{\eta}_3\phi_1\bar{\phi}_1\chi_1
+2\,\mathrm{i}\,\tilde{\Psi}_3\,\bar{\chi}_0\chi_1
+4\,\zeta_5\bar{\phi}_0\bar{\chi}_0\chi_1
-8\,\bar{\zeta}_4\phi_1\bar{\chi}_0\chi_1
+4\,\zeta_2\bar\phi_1\bar{\chi}_0\chi_1 \nonumber\\
&\quad
-4\,\eta_5\bar{\chi}_0^2\chi_1
+8\,\bar{\eta}_4\bar{\chi}_0\chi_1^2
+4\,\eta_2\bar{\phi}_0\phi_1\bar{\chi}_1
-4\,\zeta_2\bar{\phi}_0\chi_1\bar{\chi}_1
+4\,\bar{\zeta}_3\phi_1\chi_1\bar{\chi}_1
-4\,\bar{\eta}_3\chi_1^2\bar{\chi}_1
-3\,\tilde{\Psi}_2\,\lambda \nonumber\\
&\quad
+2\,\mathrm{i}\,\bar{\zeta}_4\phi_0\,\lambda
-2\,\mathrm{i}\,\zeta_4\bar{\phi}_0\,\lambda
+2\,\mathrm{i}\,\zeta_1\bar{\phi}_1\,\lambda
-2\,\mathrm{i}\,\bar{\eta}_4\chi_0\,\lambda
+2\,\mathrm{i}\,m\,\phi_1\chi_0\,\lambda
+2\,\mathrm{i}\,\eta_4\bar{\chi}_0\,\lambda \nonumber\\
&\quad
-2\,\mathrm{i}m\bar\phi_1\bar\chi_0\lambda
-2\,\mathrm{i}m\phi_0\chi_1\lambda
-2\,\mathrm{i}\,\eta_1\bar\chi_1\lambda
+2\,\mathrm{i}m\bar\phi_0\bar\chi_1\lambda
-2\,\mathrm{i}\,\zeta_2\bar\phi_0\mu \nonumber\\
&\quad
+2\,\mathrm{i}\,\eta_2\bar\chi_0\mu
+5\,\tilde\Psi_3\pi
-4\,\mathrm{i}\,\zeta_5\bar\phi_0\pi
+9\,\mathrm{i}\,\bar\zeta_4\phi_1\pi
-5\,\mathrm{i}\,\zeta_2\bar\phi_1\pi
+4\,\mathrm{i}\,\eta_5\bar\phi_0\pi
-9\,\mathrm{i}\,\bar\eta_4\chi_1\pi
+5\,\mathrm{i}\,\eta_2\bar\chi_1\pi \nonumber\\
&\quad
+\Psi_4\,\rho
-2\,\Psi_4\,\omega
+2\,\mathrm{i}\,\bar\chi_1\bigl(\meth'\eta_2\bigr)
-2\,\mathrm{i}\,\chi_1\bigl(\meth'\bar\eta_4\bigr)
-2\,\mathrm{i}\,\bar\phi_1\bigl(\meth'\zeta_2\bigr)
+2\,\mathrm{i}\,\phi_1\bigl(\meth'\bar\zeta_4\bigr)
+\meth'\tilde\Psi_3, \label{thornPsi4}
\end{align}

\subsection{Auxiliary structure equations }

\begin{align}
\mthorn\tilde{\tau} &= -2\mathrm{i}\bar\eta_0\eta_4+2\mathrm{i}\eta_0\bar\eta_4+2\mathrm{i}\bar\zeta_0\zeta_4-2\mathrm{i}\zeta_0\bar\zeta_4
+\frac{7\mathrm{i}}{3}m\eta_1\phi_0
-\frac{7\mathrm{i}}{3}m\bar\eta_1\bar\phi_0 \nonumber\\
&\quad
-\frac{\mathrm{i}}{3}m\eta_0\phi_1
-4\bar\zeta_1\phi_0\bar\phi_0\phi_1
+2\zeta_3\bar\phi_0^{2}\phi_1
+\frac{\mathrm{i}}{3}m\bar\eta_0\bar\phi_1
+2\bar\zeta_3\phi_0^{2}\bar\phi_1
-4\zeta_1\phi_0\bar\phi_0\bar\phi_1 \nonumber\\
&\quad+2\bar\zeta_0\phi_0\phi_1\bar\phi_1
+2\zeta_0\bar\phi_0\phi_1\bar\phi_1
-\frac{7\mathrm{i}}{3}m\zeta_1\chi_0
+4\bar\eta_1\bar\phi_0\phi_1\chi_0
-2\bar\eta_3\phi_0\bar\phi_1\chi_0
+\frac{7\mathrm{i}}{3}m\bar\zeta_1\bar\chi_0 \nonumber\\
&\quad-2\eta_3\bar\phi_0\phi_1\bar\chi_0
+4\eta_1\phi_0\bar\phi_1\bar\chi_0
+\frac{\mathrm{i}}{3}m\zeta_0\chi_1
-2\bar\eta_0\phi_0\bar\phi_1\chi_1
+4\bar\zeta_1\phi_0\bar\chi_0\chi_1
-2\zeta_3\bar\phi_0\bar\chi_0\chi_1 \nonumber\\
&\quad-2\zeta_0\bar\phi_1\bar\chi_0\chi_1
-4\bar\eta_1\chi_0\bar\chi_0\chi_1
+2\eta_3\bar\chi_0^{2}\chi_1
-\frac{\mathrm{i}}{3}m\bar\zeta_0\bar\chi_1
-2\eta_0\bar\phi_0\phi_1\bar\chi_1
-2\bar\zeta_3\phi_0\chi_0\bar\chi_1 \nonumber\\
&\quad+4\zeta_1\bar\phi_0\chi_0\bar\chi_1
-2\bar\zeta_0\phi_1\chi_0\bar\chi_1
+2\bar\eta_3\chi_0^{2}\bar\chi_1
-4\eta_1\chi_0\bar\chi_0\bar\chi_1
+2\bar\eta_0\chi_0\chi_1\bar\chi_1
+2\eta_0\bar\chi_0\chi_1\bar\chi_1\nonumber\\
&\quad+\Psi_0\lambda-\TiPsi_1\pi
-\mathrm{i}\bar\zeta_1\phi_0\pi+\mathrm{i}\bar\eta_1\chi_0\pi
+\mathrm{i}\zeta_1\bar\phi_0\bar\pi-\mathrm{i}\eta_1\bar\chi_0\bar\pi
-\TiPsi_2\rho
-2\mathrm{i}\bar\zeta_1\phi_1\rho+2\mathrm{i}\zeta_1\bar\phi_1\rho
+2\mathrm{i}m\phi_1\chi_0\rho \nonumber \\
&\quad
-2\mathrm{i}m\bar\phi_1\bar\chi_0\rho+2\mathrm{i}\bar\eta_1\chi_1\rho
-2\mathrm{i}m\phi_0\chi_1\rho
-2\mathrm{i}\eta_1\bar\chi_1\rho+2\mathrm{i}m\bar\phi_0\bar\chi_1\rho
-2\mathrm{i}\zeta_2\bar\phi_0\sigma \nonumber\\
&\quad
+2\mathrm{i}\eta_2\bar\chi_0\sigma+2\mathrm{i}\zeta_3\bar\phi_1\bar\sigma-2\mathrm{i}\eta_3\bar\chi_1\bar\sigma
+2\rho\tilde{\tau}
+2\mathrm{i}\bar\chi_0(\meth'\eta_1)-2\mathrm{i}\bar\phi_0(\meth'\zeta_1)
+\tau(\meth\bar\sigma)+\bar\sigma(\meth\tau) \nonumber\\
&\quad-2\mathrm{i}\chi_0(\meth'\bar\eta_1)+2\mathrm{i}\phi_0(\meth'\bar\zeta_1)
+\sigma(\meth'\pi)+\rho(\meth'\bar\pi)+\bar\pi(\meth'\rho)+\pi(\meth'\sigma)
+\bar\tau(\meth'\sigma)+\sigma(\meth'\bar\tau) \label{thorntitau}
\end{align}

\begin{align}
\mthorn'\tilde{\pi} &= -2\mathrm{i}\bar\eta_1\eta_5+2\mathrm{i}\eta_1\bar\eta_5+2\mathrm{i}\bar\zeta_1\zeta_5-2\mathrm{i}\zeta_1\bar\zeta_5
-\frac{\mathrm{i}}{3}m\eta_5\phi_0
+\frac{\mathrm{i}}{3}m\bar\eta_5\bar\phi_0 \nonumber\\
&\quad
+\frac{7\mathrm{i}}{3}m\eta_4\phi_1
-2\bar\zeta_5\phi_0\bar\phi_0\phi_1
-2\bar\zeta_2\bar\phi_0\phi_1^{2}
-\frac{7\mathrm{i}}{3}m\bar\eta_4\bar\phi_1
-2\zeta_5\phi_0\bar\phi_0\bar\phi_1 \nonumber\\
&\quad+4\bar\zeta_4\phi_0\phi_1\bar\phi_1
+4\zeta_4\bar\phi_0\phi_1\bar\phi_1
-2\zeta_2\phi_0\bar\phi_1^{2}
+\frac{\mathrm{i}}{3}m\zeta_5\chi_0
+2\bar\eta_5\bar\phi_0\phi_1\chi_0
-\frac{\mathrm{i}}{3}m\bar\zeta_5\bar\chi_0 \nonumber\\
&\quad+2\eta_5\phi_0\bar\phi_1\bar\chi_0
-\frac{7\mathrm{i}}{3}m\zeta_4\chi_1
+2\bar\eta_2\bar\phi_0\phi_1\chi_1
-4\bar\eta_4\phi_0\bar\phi_1\chi_1
+2\bar\zeta_5\phi_0\bar\chi_0\chi_1
+2\bar\zeta_2\phi_1\bar\chi_0\chi_1 \nonumber\\
&\quad-4\zeta_4\bar\phi_1\bar\chi_0\chi_1-
2\bar\eta_5\chi_0\bar\chi_0\chi_1-2\bar\eta_2\bar\chi_0\chi_1^{2}
+\frac{7\mathrm{i}}{3}m\bar\zeta_4\bar\chi_1
-4\eta_4\bar\phi_0\phi_1\bar\chi_1+2\eta_2\phi_0\bar\phi_1\bar\chi_1 \nonumber\\
&\quad+2\zeta_5\bar\phi_0\chi_0\bar\chi_1
-4\bar\zeta_4\phi_1\chi_0\bar\chi_1+2\zeta_2\bar\phi_1\chi_0\bar\chi_1
-2\eta_5\chi_0\bar\chi_0\bar\chi_1
+4\bar\eta_4\chi_0\chi_1\bar\chi_1
+4\eta_4\bar\chi_0\chi_1\bar\chi_1 \nonumber\\
&\quad
-2\eta_2\chi_0\bar\chi^2_1
-2\mathrm{i}\zeta_3\bar\phi_1\lambda+2\mathrm{i}\eta_3\bar\chi_1\lambda
+2\mathrm{i}\zeta_2\bar\phi_0\bar\lambda-2\mathrm{i}\eta_2\bar\chi_0\bar\lambda
-\TiPsi_2\mu \nonumber\\
&\quad
-2\mathrm{i}\bar\zeta_4\phi_0\mu+2\mathrm{i}\zeta_4\bar\phi_0\mu+2\mathrm{i}\bar\eta_4\chi_0\mu
+2\mathrm{i}m\phi_1\chi_0\mu
-2\mathrm{i}\eta_4\bar\chi_0\mu
-2\mathrm{i}m\bar\phi_1\bar\chi_0\mu \nonumber\\
&\quad
-2\mathrm{i}m\phi_0\chi_1\mu
+2\mathrm{i}m\bar\phi_0\bar\chi_1\mu
+\mathrm{i}\zeta_4\bar\phi_1\pi-\mathrm{i}\eta_4\bar\chi_1\pi-\bar\lambda\pi^2-2\mu\tilde{\pi}
-\mathrm{i}\bar\zeta_4\phi_1\bar\pi+\mathrm{i}\bar\eta_4\chi_1\bar\pi \nonumber\\
&\quad-\lambda\pi^2+\Psi_4\sigma-\TiPsi_3\tau+\lambda\tau^2+\bar\lambda\pi\bar\tau
-\mu\pi\bar\tau+\mu\tau\bar\tau+2\mathrm{i}\chi_1(\meth\bar\eta_4)-2\mathrm{i}\phi_1(\meth\bar\zeta_4) \nonumber\\
&\quad-\bar\pi(\meth\lambda)-\tau(\meth\lambda)-\bar\tau(\meth\mu)-\lambda(\meth\bar\pi)
-\lambda(\meth\tau)-\mu(\meth\bar\tau)-2\mathrm{i}\bar\chi_1(\meth'\eta_4)
+2\mathrm{i}\bar\phi_1(\meth'\zeta_4)-\pi(\meth'\bar\lambda)-\bar\lambda(\meth'\pi) 
\label{thornprimetipi}
\end{align}

\begin{align}
\mthorn'\tilde{\omega} &= 
2\mathrm{i}\eta_1\bar\eta_2-2\mathrm{i}\eta_3\bar\eta_4-2\mathrm{i}\zeta_1\bar\zeta_2+2\mathrm{i}\zeta_3\bar\zeta_4
+\mathrm{i}m\eta_4\phi_0+\mathrm{i}m\bar\eta_1\bar\phi_1
-\mathrm{i}m\zeta_4\chi_0-\mathrm{i}m\bar\zeta_1\bar\chi_1 \nonumber\\
&\quad 
-2\bar{\TiPsi}_1\bar\lambda+2\mathrm{i}\zeta_1\bar\phi_0\bar\lambda
-2\mathrm{i}\eta_1\bar\chi_0\bar\lambda-2\TiPsi_1\mu
-2\mathrm{i}\bar\zeta_1\phi_0\mu
+4\mathrm{i}\zeta_3\bar\phi_0\mu+2\mathrm{i}\bar\eta_1\chi_0\mu \nonumber\\
&\quad
-4\mathrm{i}\eta_3\bar\chi_0\mu
-2\mathrm{i}\bar\zeta_2\phi_0\pi-2\mathrm{i}\zeta_3\bar\phi_1\pi
+2\mathrm{i}\bar\eta_2\chi_0\pi+2\mathrm{i}\eta_3\bar\chi_1\pi
-2\TiPsi_2\bar\pi-3\mathrm{i}\bar\zeta_4\phi_0\bar\pi
+3\mathrm{i}\zeta_4\bar\phi_0\bar\pi \nonumber\\
&\quad-\mathrm{i}\bar\zeta_1\phi_1\bar\pi+ \mathrm{i}\zeta_1\bar\phi_1\bar\pi
+3\mathrm{i}\bar\eta_4\chi_0\bar\pi
+\frac{2\mathrm{i}}{3}m\phi_1\chi_0\bar\pi
-3\mathrm{i}\eta_4\bar\chi_0\bar\pi
-\frac{2\mathrm{i}}{3}m\bar\phi_1\bar\chi_0\bar\pi
+\mathrm{i}\bar\eta_1\chi_1\bar\pi \nonumber\\
&\quad
-\frac{2\mathrm{i}}{3}m\phi_0\chi_1\bar\pi
-\mathrm{i}\eta_1\bar\chi_1\bar\pi
+\frac{2\mathrm{i}}{3}m\bar\phi_0\bar\chi_1\bar\pi
-2\pi\bar\pi^{2}-4\mathrm{i}\bar\zeta_5\phi_0\rho
+6\mathrm{i}\zeta_4\bar\phi_1\rho
+4\mathrm{i}\bar\eta_5\chi_0\rho \nonumber\\
&\quad
-6\mathrm{i}\eta_4\bar\chi_1\rho
+4\TiPsi_3\sigma-4\mathrm{i}\zeta_5\bar\phi_0\sigma+10\mathrm{i}\bar\zeta_4\phi_1\sigma
-4\mathrm{i}\zeta_2\bar\phi_1\sigma+4\mathrm{i}\eta_5\bar\chi_0\sigma \nonumber\\
&\quad-10\mathrm{i}\bar\eta_4\chi_1\sigma+4\mathrm{i}\eta_2\bar\chi_1\sigma
-4\TiPsi_2\tau+6\mathrm{i}\bar\zeta_4\phi_0\tau-6\mathrm{i}\zeta_4\bar\phi_0\tau
+6\mathrm{i}\bar\zeta_1\phi_1\tau-6\mathrm{i}\zeta_1\bar\phi_1\tau
-6\mathrm{i}\bar\eta_4\chi_0\tau \nonumber\\
&\quad
+\frac{10\mathrm{i}}{3}m\phi_1\chi_0\tau
+6\mathrm{i}\eta_4\bar\chi_0\tau
-\frac{10\mathrm{i}}{3}m\bar\phi_1\bar\chi_0\tau
-6\mathrm{i}\bar\eta_1\chi_1\tau
-\frac{10\mathrm{i}}{3}m\phi_0\chi_1\tau \nonumber\\
&\quad
+6\mathrm{i}\eta_1\bar\chi_1\tau
+\frac{10\mathrm{i}}{3}m\bar\phi_0\bar\chi_1\tau
+2\pi\tau^{2}
+4\mathrm{i}\bar\zeta_2\phi_0\bar\tau
-4\mathrm{i}\zeta_3\bar\phi_1\bar\tau
-4\mathrm{i}\eta_2\chi_0\bar\tau
+4\mathrm{i}\eta_3\bar\chi_1\bar\tau
-2\bar\pi^2\bar\tau \nonumber\\
&\quad
+2\bar\pi\tau\bar\tau-\mu\tilde{\omega}+\bar\lambda\tilde{\omega}
-2\mathrm{i}\bar\chi_1(\meth\eta_1)
+2\mathrm{i}\chi_1(\meth\bar\eta_1)
-2\mathrm{i}\bar\chi_0(\meth\eta_4)
+2\mathrm{i}\chi_0(\meth\bar\eta_4)
+2\mathrm{i}\bar\phi_1(\meth\zeta_1)
-2\mathrm{i}\phi_1(\meth\bar\zeta_1) \nonumber\\
&\quad+2\mathrm{i}\bar\phi_0(\meth\zeta_4)
-2\mathrm{i}\phi_0(\meth\bar\zeta_4)
-2\bar\pi(\meth\pi)
-2\tau(\meth\pi)
-2\pi(\meth\bar\pi)
-2\bar\tau(\meth\bar\pi)
-2\pi(\meth\tau)
-2\bar\pi(\meth\bar\tau) \nonumber\\
&\quad+4\mathrm{i}\chi_0(\meth'\bar\eta_2)
-4\mathrm{i}\bar\chi_1(\meth'\eta_3)
-4\mathrm{i}\phi_0(\meth'\bar\zeta_2)
+4\mathrm{i}\bar\phi_1(\meth'\zeta_3)
-2\bar\lambda(\meth'\omega). \label{thornprimetiomega}
\end{align}

%%%%%%%%%%%%%%%%%%%%%%%%%%%%%%%%%%%%%%%%%%%%%%%%%%%%%%%%%%%%%%%%%%%%%%%%%%%%%%%

\end{document}